\documentclass[12pt]{article}

\usepackage[utf8]{inputenc}
\usepackage{amsmath}
\usepackage{amssymb}
\usepackage{amsthm}
\usepackage{graphics,xcolor}
\graphicspath{{figures/}}
\usepackage[colorlinks=true,linkcolor=blue,citecolor=magenta]{hyperref}

\makeatletter

\usepackage{listings}
\definecolor{colKeys}{rgb}{0,0,1}
\definecolor{colIdentifier}{rgb}{0,0,0}
\definecolor{colComments}{rgb}{0,1,0.3}
\definecolor{colString}{rgb}{0,0.5,0}

\definecolor{dkgreen}{rgb}{0,0.6,0}
\definecolor{gray}{rgb}{0.5,0.5,0.5}
\definecolor{lightgray}{rgb}{0.9,0.9,0.9}

\usepackage{float}
\newfloat{Code}{H}{myc}

\usepackage{geometry}
\usepackage{todonotes}

\newcommand{\R}{\mathbb{R}}
\newcommand{\C}{\mathbb{C}}
\newcommand{\cT}{\mathcal{T}}
\newcommand{\M}{\mathcal{M}}

\newtheorem{theorem}{Theorem}

\title{Robust FEM-based extraction\\ of finite-time coherent sets using\\ scattered, sparse, and incomplete trajectories}

\author{Gary Froyland\footnote{School of Mathematics and Statistics, The University of New South Wales, Sydney NSW 2052, Australia, \texttt{g.froyland@unsw.edu.au}}
\ and
Oliver Junge\footnote{Department of Mathematics, Technical University of Munich, 85747 Garching, Germany, \texttt{oj@tum.de}}}

\date{\today}

\makeatother

\begin{document}
\maketitle
\begin{abstract}
 Transport and mixing properties of aperiodic flows are crucial to a dynamical analysis of the flow, and often have to be carried out with limited information. Finite-time coherent sets are regions of the flow that minimally mix with the remainder of the flow domain over the finite period of time considered.  In the purely advective setting this is equivalent to identifying sets whose boundary interfaces remain small throughout their finite-time evolution.  Finite-time coherent sets thus provide a skeleton of distinct regions around which more turbulent flow occurs. They manifest in geophysical systems in the forms of e.g.\ ocean eddies, ocean gyres, and atmospheric vortices. In real-world settings, often observational data is scattered and sparse, which makes the difficult problem of coherent set identification and tracking even more challenging.  We develop three FEM-based numerical methods to efficiently approximate the dynamic Laplace operator, and introduce a new dynamic isoperimetric problem using Dirichlet boundary conditions.
 Using these FEM-based methods we rapidly and reliably extract finite-time coherent sets from models or scattered, possibly sparse, and possibly incomplete observed data.
\end{abstract}

\section{Introduction}

The quantification of transport and mixing processes in nonlinear dynamical systems is key to understanding their global dynamics.  In the autonomous setting, where the governing dynamical laws do not change over time, stable and unstable manifolds of low-period orbits have been used to describe the global dynamics.  For example, co-dimension 1 invariant manifolds represent an impenetrable barrier to trajectories.  In the case of periodically-forced flows, the intersections of stable and unstable manifolds form mobile parcels dubbed ``lobes'', which give rise to {lobe dynamics} \cite{mackay84,romkedar90} as a mechanism for transport across periodically varying invariant manifolds.

In the autonomous setting, the concept of almost-invariant sets, introduced in \cite{DJ99} and further developed in \cite{Froyland:2003jj,F05} uses the transfer (or Perron-Frobenius) operator to find regions in phase space which are either invariant, or in the absence of any open invariant sets, as close to invariant as possible.  Such sets decompose the phase space into regions which are as dynamically disconnected as possible, and between which transport is minimized.
Related approaches to finding ``ergodic partitions'' in phase space have also been developed using long orbits \cite{mezic99}
or using the Koopman operator \cite{budisic},
the dual of the transfer operator.

In the case of nonautonomous, aperiodic dynamics, the notion of almost-invariant sets was generalized to coherent sets \cite{FLS10} and shortly thereafter to finite-time coherent sets \cite{FSM10,F13}.  While almost-invariant sets remain fixed in phase space, when the underlying governing dynamics is time-varying, it is essential to allow sets that are parameterised by time.  Nevertheless, the underlying idea, minimising mixing with the surrounding phase space, is common to both almost-invariant and finite-time coherent sets.  As in the case of almost-invariant sets, one numerically constructs a transfer operator for the duration of the dynamics to be considered.  This transfer operator will push forward densities over the chosen time interval.  After some straightforward normalising procedures, one computes the subdominant singular spectrum and singular vectors of the transfer operator, which indicate the presence and location of finite-time coherent sets.  This construction relies on a small amount of diffusion being present in the dynamics;  this diffusion can be part of the model, or if the underlying dynamics is purely advective it can be easily numerically induced.

An equivalent definition of finite-time coherent sets was introduced in \cite{F15}, specifically aimed at the case of purely advective dynamics.  The underlying goal is to determine co-dimension 1 manifolds whose size remains small relative to enclosed\footnote{The boundary of the domain may take part in the enclosing.} volume, under evolution by the dynamics.  This idea extends the classical (static) isoperimetric problem \cite{Bl:05a} of determining the co-dimension 1 manifold with least size given a fixed enclosed volume.
The idea to define finite-time coherent sets as those sets with boundaries that have dynamically minimal size relative to enclosed volume is consistent with the minimal $L^2$-mixing approach of \cite{F13} because if small-amplitude diffusion is added, $L^2$ mixing can only occur along the boundary interfaces between coherent regions.
This fact has been formally established in \cite{F15} in the volume-preserving setting and \cite{FrKw16} in the non-volume-preserving case.
The persistently minimal size of the boundary of finite-time coherent sets in this sense is also broadly consistent with the approach of \cite{hallerblack}, although \cite{hallerblack} is limited to two-dimensions and closed curves.
Related work \cite{KK16} uses a time-dependent metric (with corresponding time-dependent diffusion) induced by the advective flow to define coherent sets as almost-invariant sets in the material manifold. In approximating the time dependent metric by a single one they end up with a similar operator as the dynamic Laplacian in \cite{F15,FrKw16}.

Numerical implementation of the method of \cite{F15} requires approximation of eigenfunctions of a \textit{dynamic Laplace operator}, which consists of the Laplace operator composed with transfer operators, or equivalently, combinations of Laplace-Beltrami operators with respect to Riemannian metrics pulled back under the nonlinear dynamics.
The papers \cite{F15,FrKw16} used low-order methods such as finite-differences to approximate the Laplace operator and Ulam's method \cite{ulam} to approximate the transfer operators.
 In \cite{FJ15}, the authors investigated the use of radial basis functions (RBFs) as an approximating basis for $L^2$, seeking to exploit the smoothness of the underlying dynamics to achieve accurate results with relatively sparse trajectory information.
This approach successfully achieved very accurate results with relatively sparse data, but due to known issues with RBF interpolation \cite{We:05a,Fa:07a} care must be taken when selecting the RBF centre spacing and RBF radii.  In \cite{BaKo17a}, diffusion maps on possibly scattered trajectory data were used for the approximation of the dynamic Laplacian, and in \cite{HKTH16}, an object is constructed which is related to a discretised version of the dynamic Laplace operator, again from possibly scattered trajectory data.

The dynamic Laplacian is a self-adjoint, elliptic differential operator and  arguably the finite element method (FEM) is the standard numerical method to discretize these.
In our specific dynamic context, there are a number of major advantages to this method:
\begin{itemize}
\item \emph{Sparsity:} We obtain a {sparse} discrete problem. Thus, the computational cost grows only linearly with the number of mesh elements and enables one to use high resolution meshes (in contrast to the RBF approach \cite{FJ15}).
\item  \emph{Structure preservation:} The dynamic Laplacian is self-adjoint and thus its weak form is symmetric.  This symmetry is inherited in the {discrete} eigenproblem, ensuring a purely real spectrum (in contrast to, e.g., the RBF  method).
\item \emph{Arbitrary domains: } FEMs {can elegantly handle arbitrarily shaped domains}, as often occur in applications, unlike spectral methods, which require tensor product domains.
\item \emph{Natural boundary conditions:} Homogeneous Neumann boundary conditions are the natural ones in our context and completely disappear in the weak formulation. Homogeneous Dirichlet conditions can also easily be incorporated, as we introduce and describe here.
\end{itemize}

We would further like to highlight advantages over other existing methods for extracting (finite time) coherent sets:
\begin{enumerate}
\item There is \textbf{no free parameter} to choose (like a cut-off radius, a neighborhood size etc.) experimentally or heuristically when computing the eigenvectors.  When extracting coherent sets from these, some heuristics may be needed to choose how many/which eigenvectors yield relevant information. In many cases, this choice is supported by an eigengap heuristic.
\item Even in the case of scattered and sparse trajectory data, we always obtain coherence information \textbf{for every point in phase space}.
\item Our method applies to \textbf{vector field data} and  \textbf{discretely sampled trajectory data} alike.  The data can be given on a \textbf{scattered grid}, can be \textbf{sparse} and \textbf{incomplete} Of course, the predictive power will decrease with decreasing data density, however, we observe robust and reliable results for highly nonlinear flows with relatively little data.
\item Our FEM approach \textbf{does not require derivative information} on the flow map, but can make use of it when it is available.
\end{enumerate}

An outline of the paper is as follows: In Section \ref{sec:background} and \ref{sec:2.1} we provide a brief overview of the dynamic isoperimetric problem, the associated dynamic Laplacian eigenproblem, and state the weak form of this eigenproblem.
In Section \ref{sec:2.2} we introduce and motivate the new Dirichlet form of the dynamic isoperimetric problem, and state fundamental theoretical results (dynamic Federer-Fleming theorem and dynamic Cheeger's inequality) corresponding to the existing Neumann theory.
In Section \ref{sec:fem} we first introduce the discrete form of the weak eigenproblem, which involves gradients of push-forwards of elements under the transfer operator.  In Section \ref{sec:CG} we convert the gradient of the push-forward of an element into the push-forward of the gradient of an element which results in a stiffness matrix containing the Cauchy-Green tensor;  standard FEM can then be applied.  In Section \ref{sec:TO} we instead estimate the push-forward of elements under the transfer operator.
We first describe the case of volume-preserving dynamics, and then in Section \ref{sect:nvp} detail the situations where (i) the dynamics is not volume-preserving or (ii) one wishes to extract coherent features with respect to non-uniform densities;  cases (i) and (ii) are handled identically.
To implement the transfer operator approach we use two forms of collocation described in Sections \ref{sec:collnonadap}--\ref{sec:colladap}.
In this case, one uses the standard stiffness matrix for the original elements, multiplied by a matrix representing the transfer operator.
In the second collocation form (Section \ref{sec:colladap}), if the input data comes from trajectories, estimation of the transfer operator is completely avoided (cf.~Algorithm~1 and the code in the appendix).

The transfer operator approach is extremely efficient and robust, particularly in the situation where the trajectory information is sparse.
Section \ref{sec:missing} describes the minor modifications required to handle trajectories where many or most data points are missing.
This is particularly useful if trajectories arise from physical observations, which may not be recorded at certain times due to mechanical breakdown, adverse weather, or other obstructions.  
Section \ref{sec:extract} recalls how to extract multiple coherent sets from multiple eigenfunctions, and Section \ref{sec:graph} interprets the adaptive transfer operator collocation method (section \ref{sec:colladap}) in terms of graph Laplacians.

Section \ref{sec:numerical} contains numerical experiments that illustrate the robustness and speed of the proposed methods on a variety of models and data availability. The dynamical systems tested include two- and three-dimensional idealized models, including a model of ocean flow, and experiments with scattered, strongly irregularly distributed, and missing data.
In the appendices, we provide brief theoretical background for the non-volume preserving case and provide sample code for one the examples.
The corresponding MATLAB package FEMDL can be downloaded from \href{{https://github.com/gaioguy/FEMDL}}{\texttt{https://github.com/gaioguy/FEMDL}}.

\section{Finite-time coherent sets and dynamic Laplacians}
\label{sec:background}

Let $\M$  be a compact subset of $\mathbb{R}^d$ with empty or piecewise smooth boundary, or more generally  a compact, smooth, Riemannian manifold.
For $t\in[0,t_F]$, we consider diffeomorphisms $\Phi^t\colon \M\subset\mathbb{R}^{d}\to \Phi^t(\M)\subset\mathbb{R}^{d}$, representing e.g.\ flow maps from time 0 to time $t$.
Let $\Phi^t_{*}:C^\infty(\M)\to C^\infty(\Phi^t(\M))$ denote the pushforward of smooth functions on $\M$ by
$\Phi^t$, i.e., $\Phi^t_{*}\varphi=\varphi\circ \Phi^{-t}$.
Similarly define the pullback $(\Phi^t)^{*}:C^\infty(\Phi^t(M))\to C^\infty(\M)$ by $(\Phi^t)^*\varphi=\varphi\circ \Phi^t$.
We note that for volume-preserving $\Phi^t$, the pushforward
$\Phi^t_*$ is the transfer (or Perron-Frobenius) operator and for general $\Phi^t$, the pullback $(\Phi^t)^*$ is the Koopman operator.

We will say that regions in phase space remain coherent if their boundaries resist filamentation under the imposed nonlinear dynamics.
Lack of filamentation is connected to mixing in an $L^2$ sense (see e.g.\ Corollary 4.4.1 \cite{lasotamackey} and the finite-time counterpart in Proposition 2 \cite{F13}) because long boundaries present large interfaces over which diffusive mixing can occur in the presence of small diffusion.
Lack of filamentation is also consistent with fluid mixing norms for advective flows \cite{mathew_etal,thiffeault_review}, which are necessarily computed with Sobolev-type norms to mathematically introduce an effect akin to diffusion, as well as with recent variational approaches to Lagrangian coherent structures \cite{hallerblack}, which control evolved length of material curves.

\subsection{The dynamic Laplacian with Neumann boundary conditions}
\label{sec:2.1}
The framework of \cite{F15} introduces the dynamic isoperimetric problem:  The question is  how to disconnect the given manifold $\M$ by a co-dimension 1 manifold $\Gamma$ in such a way that the average evolved size of $\Gamma$ is minimal relative to the volume of the two disconnected pieces.  Let $\mathcal{T}\subset [0,t_F]$, $|\mathcal{T}|>1$, $0\in\mathcal{T}$, be the collection of times at which we wish to check the size of our evolving disconnector.  We define the dynamic Cheeger constant for a smooth co-dimension 1 manifold $\Gamma\subset \M$ that disconnects $\M$ into two pieces $\M_1, \M_2\subset \M$, as
\begin{equation}
\label{dyniso}
\mathbf{h}(\Gamma):=\frac{\frac{1}{|\mathcal{T}|}\sum_{t\in\mathcal{T}}\ell_{d-1}(\Phi^t\Gamma)}{\min\{\ell(\M_1),\ell(\M_2)\}},
\end{equation}
where $\ell_{d-1}$ is co-dimension 1 volume and $\ell$ is $d$-dimensional volume.
The numerator of (\ref{dyniso}) is the average size of the evolved hypersurfaces $\Phi^t\Gamma$ over the time instants in $\mathcal{T}$.
We wish to solve
\begin{equation}
\label{optdyniso}
\mathbf{h}:=\min\{\mathbf{h}(\Gamma):\Gamma\mbox{ is a $C^\infty$ co-dimension 1 manifold disconnecting $\M$}\}.
\end{equation}
The minimising $\Gamma$ generates a family $\{\Phi^t\Gamma\}_{t\in\mathcal{T}}$ of material (evolving with the flow) disconnectors that are of minimal average size relative to the volume of the disconnected pieces of $\M$.  Note that the minimisation problem (\ref{optdyniso}) and its solution $\Gamma$ are frame-invariant because $\ell_{d-1}$ and $\ell$ are invariant under time-dependent translations and rotations.


A sharp relationship between $\mathbf{h}$ and functional minimisation is provided by the dynamic Federer-Fleming theorem \cite{F15}, which states that $\mathbf{h}=\mathbf{s}$, where
\begin{equation}
\label{sobolev}
\mathbf{s}:=\inf_{f\in C^\infty(\M,\mathbb{R})}\frac{\frac{1}{|\mathcal{T}|}\sum_{t\in \mathcal{T}}\|\nabla (\Phi^t_*f)\|_1}{\inf_{\alpha\in\mathbb{R}}\|f-\alpha\|_1}
\end{equation}
is the dynamic Sobolev constant.
The same result for non-volume-preserving transformations on weighted Riemannian manifolds with general Riemannian metrics is developed in \cite{FrKw16}.

When replacing the $L^1$ optimisation with $L^2$ optimisation (which may be thought of as a regularisation), the problem can be solved exactly and efficiently as an eigenproblem for a dynamic Laplace operator:
\begin{equation}
\left(\frac{1}{|\mathcal{T}|}\sum_{t\in\mathcal{T}}(\Phi^t)^{*}\Delta \Phi^t_{*}\right)v  =\lambda v\quad \mbox{ on int}(\M),\label{eq:eigenproblem}
\end{equation}
subject to homogeneous Neumann boundary conditions.
Regularising the $L^1$ optimisation problem and computing the exact solution of the $L^2$ version of the problem is a standard approach in optimisation, and when evaluating the quality of solutions provided by level sets of the $L^2$ eigenfunctions, one can make this assessment using the original $L^1$ criterion (\ref{optdyniso}).
The minimum of this regularised problem is given by the eigenfunction of (\ref{eq:eigenproblem}) corresponding to the second eigenvalue $\lambda_2$, and the optimal value of the $L^1$ minimisation (which equals $\mathbf{h}$) satisfies the bound $\mathbf{h}\le 2\sqrt{-\lambda_2}$ (see Corollary 3.4 \cite{F15} and the discussion following Theorem 4.5 \cite{FrKw16}).




\paragraph{The weak form of the eigenproblem}\hspace*{-4mm} is developed by first integrating against a test function $\psi\in H^1(\M)$, applying the transformation formula for $\Phi^t$ (which we firstly assume to be volume-preserving in the calculation below) and then applying Green's formula, noting that the boundary conditions are the natural ones:
\begin{eqnarray*}
\int_{\M}\left(\frac{1}{|\mathcal{T}|}\sum_{t\in\mathcal{T}}(\Phi^t)^{*}\Delta \Phi^t_{*}\right)v\cdot \psi\ d\ell
&=& \frac{1}{|\mathcal{T}|}\sum_{t\in\mathcal{T}}\int_{\M}\left((\Phi^t)^{*}\Delta \Phi^t_{*}\right)v\cdot \psi\ d\ell\\
&=& \frac{1}{|\mathcal{T}|}\sum_{t\in\mathcal{T}}\int_{\Phi^t(\M)}\Delta \Phi^t_{*}v\cdot \Phi^t_{*}\psi\ d\ell\\
&=&-\frac{1}{|\mathcal{T}|}\sum_{t\in\mathcal{T}}\int_{\Phi^t(\M)}\nabla(\Phi^t_* v)\bullet\nabla(\Phi^t_{*}\psi)\ d\ell
\end{eqnarray*}
Thus we obtain the weak form of the eigenproblem \eqref{eq:eigenproblem}: Find eigenpairs $(\lambda,v)$ such that
\begin{equation}
\label{weakeigen}
-\frac{1}{|\mathcal{T}|}\sum_{t\in\mathcal{T}}\int_{\Phi^t(\M)}\nabla(\Phi^t_* v)\bullet\nabla(\Phi^t_{*}\psi)\ d\ell=\lambda\int_{\M} v \psi \ d\ell,\quad\text{for all }\psi\in
H^1(\M).
\end{equation}
Note that we have \textit{completely eliminated any consideration of the boundary}.
If we want to have a continuous-time version of the above eigenproblem, we simply replace the sum over $t\in\mathcal{T}$ with an integral.
A related discussion for the case of $T$ not volume preserving is given in the appendix.


\subsection{The dynamic Laplacian with Dirichlet boundary conditions}
\label{sec:2.2}
In some situations, one may wish to restrict coherent sets to the class of submanifolds that do not intersect the boundary of $\M$ (which is assumed to be nonempty, otherwise one uses the setting of Section \ref{sec:2.1}).
We will illustrate one such possibility in Section \ref{exp:ocean_CG}, where an initial rectangular ocean domain becomes highly distorted.
An appropriate isoperimetric problem in this setting is to search for an open submanifold $A\subset \mbox{int}(\M)$ with compact closure, and with an evolving boundary $\Phi^t(\partial A)$ that remains small relative to the volume of $A$ for $t\in\mathcal{T}$.
This can be achieved by modifying the expression (\ref{dyniso}) to
\begin{equation}
\label{dyniso-dir}
\mathbf{h}_d(A):=\frac{\frac{1}{|\mathcal{T}|}\sum_{t\in\mathcal{T}}\ell_{d-1}(\Phi^t(\partial A))}{\ell(A)},
\end{equation}
and defining the dynamic (Dirichlet) Cheeger constant $\mathbf{h}_d$ by
\begin{equation}
\label{optdyniso-dir}
\mathbf{h}_d:=\min\{\mathbf{h}(A): A \text{ open, submanifold of $\mbox{int}(\M)$, compact closure, $C^\infty$ boundary}\}.
\end{equation}
We will call the minimising $A\subset\mbox{int}(\M)$ the dominant finite-time coherent set.

The corresponding version of the (Dirichlet) Sobolev constant is:
\begin{equation}
\label{sobolev-dir}
\mathbf{s}_d:=\inf_{f\in C^\infty_c(\M,\mathbb{R})}\frac{\frac{1}{|\mathcal{T}|}\sum_{t\in \mathcal{T}}\|\nabla (\Phi^t_*f)\|_1}{\|f\|_1},
\end{equation}
where $C^\infty_c(\M,\mathbb{R})$ denotes the set of non-identically vanishing, compactly supported, $C^\infty$ real-valued functions on $\mbox{int}(\M)$.
One can prove a dynamic (Dirichlet) Federer-Fleming theorem:
\begin{theorem}
\label{thm:dFF}
Let $\M$ be a compact subset of $\mathbb{R}^d$ with nonempty $C^\infty$ boundary, and $\mathbf{h}_d$ and $\mathbf{s}_d$ defined as in (\ref{optdyniso-dir}) and (\ref{sobolev-dir}), respectively.
Then $\mathbf{h}_d=\mathbf{s}_d$.
\end{theorem}
\begin{proof}
In the setting of $T$ volume-preserving, this result follows from the obvious modifications to the proof of Theorem 3.1 \cite{F15}.
\end{proof}
The solution of the $L^2$ version of the $L^1$ optimisation problem (\ref{sobolev-dir}) is given by the eigenfunction of the eigenproblem (\ref{eq:eigenproblem-dir})--(\ref{eq:dircond}) corresponding to the largest eigenvalue $\lambda_1<0$.
\begin{eqnarray}
\label{eq:eigenproblem-dir}\left(\frac{1}{|\mathcal{T}|}\sum_{t\in\mathcal{T}}(\Phi^t)^{*}\Delta \Phi^t_{*}\right)v  &=&\lambda v\quad \mbox{ on int}(\M),\\
\label{eq:dircond}v&=&0\quad\mbox{ on }\partial\M
\end{eqnarray}
Note that in comparison to (\ref{eq:eigenproblem}), we have simply added the Dirichlet condition (\ref{eq:dircond}), and thus the left-hand-side of (\ref{eq:eigenproblem-dir}) is the dynamic Laplace operator acting on a function $v$, and this operator is elliptic and self-adjoint as in the Neumann boundary condition case.
In this Dirichlet case, we also have a dynamic Cheeger inequality relating $\mathbf{h}_d$ to $\lambda_1$.
\begin{theorem}
\label{cheeger-dir}
Let $\M$ be a compact subset of $\mathbb{R}^d$ with nonempty $C^\infty$ boundary, and $\mathbf{h}_d$ be defined as in (\ref{optdyniso-dir}).
Then $\mathbf{h}_d\le \sqrt{-2\lambda_1}$, where $\lambda_1$ is the largest eigenvalue of (\ref{eq:eigenproblem-dir})--(\ref{eq:dircond}).
\end{theorem}
\begin{proof}
In the setting of $T$ volume-preserving, this follows from the obvious modifications to the proof of Theorem 3.2 \cite{F15}.
\end{proof}
The weak form of the eigenproblem (\ref{eq:eigenproblem-dir})--(\ref{eq:dircond}) is identical to the weak form (\ref{weakeigen}), except that we replace the space of test functions $H^1(\M)$ with $H^1_0(\M)$.
\begin{equation}
\label{weakeigen-dir}
-\frac{1}{|\mathcal{T}|}\sum_{t\in\mathcal{T}}\int_{\Phi^t(\M)}\nabla(\Phi^t_* v)\bullet\nabla(\Phi^t_{*}\psi)\ d\ell=\lambda\int_{\M} v \psi \ d\ell,\quad\text{for all }\psi\in
H^1_0(\M).
\end{equation}

In terms of heat flow (cf.\ \cite{KK16}), one interprets the Dirichlet boundary conditions as ``refrigeration'' fixed at temperature zero on the boundary of $\M$.
The eigenfunction $\phi_1$ corresponding to $\lambda_1$ is strictly positive (by convention) in the interior of $\M$ and $\phi_1$ is designed so that on average (over $t\in\mathcal{T}$) the rate of heat loss through the evolving boundaries $\Phi^t(\M), t\in\mathcal{T}$ is minimised.
This is in contrast to the previous Neumann (no flux) boundary conditions of Section \ref{sec:2.1}, where the second eigenfunction describes a signed heat distribution that approaches the constant equilibrium heat distribution at the slowest average rate (again averaged over $t\in\mathcal{T}$).

\section{Finite Element discretization}
\label{sec:fem}

We are going to use a Ritz-Galerkin discretization of the weak eigenproblem (\ref{weakeigen}) and a  finite element basis for the underlying ansatz and test function spaces.  As mentioned in the introduction, one particular advantage of this approach is that if we choose the same basis for the ansatz and the test function space then we end up with symmetric stiffness and mass matrices, i.e.\ we inherit the self-adjointness of the dynamic Laplacian.

Let $V^0_n\subset H^1(\M)$ be a finite-dimensional subspace with basis $\varphi_1^0,\ldots,\varphi_n^0$.
For brevity, we write $\varphi_i=\varphi^0_i$ in the following.
The discrete eigenproblem corresponding to (\ref{weakeigen}) is to find pairs $(\lambda,v)$, $\lambda\in\C$, $v=\sum_{i=1}^n u_i\varphi_i\in V^0_n$,
such that
\begin{equation}
\label{weakeigen_galerkin}
-\left(\frac{1}{|\mathcal{T}|}\sum_{t\in\mathcal{T}}\hat{D}^t\right)u = \lambda Mu,
\end{equation}
where $u=(u_1,\ldots,u_n)^T\in\C^n$,
\begin{equation}
 \label{Deqn}
 \hat{D}^t_{ij} = \int_{\Phi^t(\M)}\nabla\left(\Phi^t_{*}\varphi_i\right)\bullet\nabla\left(\Phi^t_{*} \varphi_j\right)\ d\ell \quad
\text{and}\quad
     M_{ij} = \int_{\M}\varphi_i \varphi_j\ d\ell
     \end{equation}
are the stiffness and mass matrices.
In the case where $\Phi^t$ is not volume-preserving and for initial mass distribution $\mu^0$ (which evolves to $\mu^t$ by push-forward:  $\mu^t=\mu^0\circ \Phi^{-t}$), one uses
\begin{equation}
   \label{Dweightedeqn}
 \hat{D}^t_{ij} = \int_{\Phi^t(\M)}\nabla\left(\Phi^t_{*}\varphi_i\right)\bullet\nabla\left(\Phi^t_{*} \varphi_j\right)\ d\mu^t \quad
\text{and}\quad
     M_{ij} = \int_{\M}\varphi_i \varphi_j\ d\mu^0,
\end{equation}
(see Section~\ref{sec:TO} on how to approximate $\mu^t$).

For ansatz and test functions, we are going to use the usual piecewise linear nodal basis on a triangulation of some set of nodes $x_1^0,\ldots,x_n^0\in\M$, i.e.~triangular $P_1$ Lagrange elements. We provide a compact MATLAB implementation  based on the code by Strang \cite{Str07} in the Appendix and a faster one based on code from the iFEM package by Chen \cite{Chen:08} on \href{https://github.com/gaioguy/FEMDL}{\texttt{https://github.com/gaioguy/FEMDL}}.

We describe two distinct approaches: the first is based on transforming the weak eigenproblem (\ref{weakeigen}) into an equivalent form which does not require integration on $\Phi^t(\M)$.  This approach, however, uses the pullback of the Euclidean metric which is described
by the Cauchy-Green deformation tensor and therefore requires the evaluation of $D\Phi^t$ (i.e.\ the integration of the variational equation in case of an ordinary differential equation).
In the second approach, we try to avoid evaluation of $D\Phi^t$.  Then, however, quadrature on $\Phi^t(\M)$ and the evaluation of the transfer operator will be required, but the former is relatively cheap and the latter can be even completely eliminated.  To this end, we use ansatz/test function spaces  $V^t_n\subset H^1(\Phi^t(\M))$ with bases $\varphi_1^t,\ldots,\varphi_n^t$ on $\Phi^t(\M)$ for each $t\in\mathcal{T}$.


\subsection{The Cauchy-Green approach}
\label{sec:CG}

Note that for general smooth $\Phi^t$ and smooth initial mass $\mu^0$ (in the case of volume-preserving $\Phi^t$, one may take $\mu^0=\mu^t=\ell$), we have
\begin{align}\label{eq:CG_1}
\nonumber \int_{\Phi^t(\M)}\nabla\left(\Phi^t_{*}\varphi_{i}\right)\bullet\nabla\left(\Phi^t_{*}\varphi_{j
}\right)\ d\mu^t
 & =\int_{\M}\nabla(\Phi^t_{*}\varphi_{i})\circ \Phi^t\bullet\nabla(\Phi^t_{*}\varphi_{j})\circ \Phi^t \ d\mu^0 \\
\nonumber &=\int_{\M}(D\Phi^t)^{-\top}\nabla\varphi_{i}\bullet(D\Phi^t)^{-\top}\nabla\varphi_{j}\ d\mu^0 \qquad\mbox{(chain rule)}\\
 & =\int_{\M}\nabla\varphi_{i}\bullet C^{-1}_t\nabla\varphi_{j}\ d\mu^0
\end{align}
where, $C_t:= (D\Phi^{t})^{\top}D\Phi^{t}$ is the (right) Cauchy–Green
deformation tensor field.  We thus have $\hat D_{ij}^t = \int_\M \nabla\varphi_i\bullet C_t^{-1}\nabla\varphi_j\;d\mu^0$ and can solve (\ref{weakeigen_galerkin}).

One advantage of this approach is that we do not need a numerical approximation of $\mu^t$. However, depending on the properties of the given transformation $\Phi^t$, the numerical evaluation of the integral (\ref{eq:CG_1}) might be challenging as the tensor $C_t^{-1}$ might be of large variation locally.  In all the following experiments, we employ a Gauss quadrature of varying degree.
This approach has been applied to the numerical experiments in \cite{KK16}.

\subsection{The transfer operator approach}
\label{sec:TO}

Instead of approximating the pullbacks of the Euclidean metric under the flow map, we can instead approximate pushforwards of the basis functions $\varphi_i$ by $\Phi_*^t$.  We do this by collocation in two ways:
\begin{enumerate}
 \item Each manifold $\Phi^{t}(\M), t\in\mathcal{T}$, is meshed independently (using nodes that do not necessarily arise from trajectories) and thus the bases $\varphi^{t}_i$, $i=1,\ldots,n$ are not related for different $t\in\mathcal{T}$.
\item The mesh for each $\Phi^{t}(\M)$, $t\in\mathcal{T}$, is a triangulation of $\{\Phi^{t}x_i\}_{i=1}^n$, namely, a triangulation of the images of the nodes $x_i, i=1,\ldots,n$, of the mesh used for $\M$.
As we will see, this construction completely removes any need to estimate the action of $\Phi^t_*$.
\end{enumerate}
Suppose now that we have found coefficients $\alpha^i_k\in\R$ such that
\begin{equation}
\label{alpha}
\Phi^t_*\varphi_i\approx \sum_{k=1}^n \alpha^i_k\varphi^t_k, \quad i=1,\ldots,n,
\end{equation}
in a sense to be made precise in the next two subsections.
We then have
\begin{eqnarray}
\nonumber
\hat{D}^t_{ij}=\int_{\Phi^t(\M)}\nabla\left(\Phi^t_{*}\varphi_{i}\right)\bullet\nabla\left(\Phi^t_{*}\varphi_{j}\right)\ d\ell
& \approx & \int_{\Phi^t(\M)}\sum_{k=1}^n\alpha^i_k\nabla\varphi_{k}^t\bullet\sum_{\ell=1}^n\alpha^j_\ell\nabla\varphi^t_{\ell}\ d\ell\\
\label{Dunweight}
& = & \sum_{k,\ell=1}^n \alpha^i_k\alpha^j_\ell\underbrace{\int_{\Phi^t(\M)}\nabla\varphi^t_{k}\bullet\nabla\varphi^t_{\ell}\ d\ell}_{=:{D}^t_{k\ell}}\\
\nonumber&=&(\alpha^i)^\top {D}^t\alpha^j,
\end{eqnarray}
where $\alpha^j=(\alpha^j_1,\ldots,\alpha^j_n)^\top$ and $D^t = (D_{k\ell}^t)_{k\ell}$.

Note that the stiffness matrix ${D}^t$ is symmetric and sparse as the $\varphi^t_i$ are locally supported;  in particular, the time to compute $D^t$ scales linearly with $n$.
In fact, the vectors $\alpha^i$ should also be sparse due to the continuity of $\Phi^t$ and the local support of the bases $\varphi^t_1,\ldots,\varphi^t_n$.
If so, then the matrices $\hat{D}^t$ will also be sparse and symmetric.
Carrying out the above computations for each $t\in \mathcal{T}$, we then directly solve (\ref{weakeigen_galerkin}).

\subsubsection{The non-volume-preserving case}
\label{sect:nvp}
 In the case of non-volume-preserving $\Phi^t$ and general initial (probability) measure $\mu^0$ we obtain
\begin{equation}
 \label{Dweight}
D^t_{k\ell}:=\int_{\Phi^t(\M)}\nabla\varphi^t_{k}\bullet\nabla\varphi^t_{\ell}\ d\mu^t
\end{equation}
in (\ref{Dunweight}).  If we approximate $\mu^0$ by point masses on the nodes $x_i^0$, e.g.\ $\mu^0 \approx \tilde\mu^0 = \sum_{i=1}^n a_i\delta_{x_i^0}$, with $a_i = \int_{\M} \varphi_i^0\; d\mu^0$, then the pushforward of $\tilde\mu^0$ under $\Phi^t$ is exactly given by
$\tilde\mu^t = \sum_{i=1}^n a_i\delta_{x_i^t}$. From $\tilde \mu^t$, we obtain a corresponding approximate density of $\mu^t$ as
\[
\tilde h^t = \sum_{i=1}^n \frac{a_i}{\ell_i^t} \varphi_i^t,
\]
where $\ell_i^t=\int_{\Phi^t\M} \varphi_i^t\; d\ell$, which we can now use in order to approximate the entries of $D^t$:
\begin{equation}
\label{eq:Dtkl}
D_{k\ell}^t
 = \int_{\Phi^t(\M)}(\nabla\varphi^t_{k}\bullet\nabla\varphi^t_{\ell})\ d\mu^t
 \approx \sum_{e^t}\int_{e^t}(\nabla\varphi^t_{k}\bullet\nabla\varphi^t_{\ell})\ \tilde h^t\; d\ell.
 \end{equation}
The outer sum in (\ref{eq:Dtkl}) over all elements $e^t$ from the triangulation of $\Phi^t(\M)$ which are in the intersection of the supports of $\varphi_k$ and $\varphi_\ell$. Further,
 \begin{align*}
\int_{e^t}(\nabla\varphi^t_{k}\bullet\nabla\varphi^t_{\ell})\ \tilde h^t\; d\ell
= \sum_{x_i^t \in e^t} \frac{a_i}{\ell_i^t} \int_{e^t}(\nabla\varphi^t_{k}\bullet\nabla\varphi^t_{\ell})\ \varphi_i^t\; d\ell.
\end{align*}
For the case of piecewise linear triangular elements, the gradients of the $\varphi_j$ are constant on each element and so in this case we get
\begin{align}
\label{eq:muapprox}
\int_{e^t}(\nabla\varphi^t_{k}\bullet\nabla\varphi^t_{\ell})\ \tilde h^t\; d\ell
\nonumber & = (\nabla\varphi^t_{k}\bullet\nabla\varphi^t_{\ell})\sum_{x_i^t \in e^t}  \frac{a_i}{\ell_i^t} \int_{e^t}\ \varphi_i^t\; d\ell \\
 & = (\nabla\varphi^t_{k}\bullet\nabla\varphi^t_{\ell})\; \text{area}(e^t)\; \frac13 \sum_{x_i^t \in e^t}  \frac{a_i}{\ell_i^t}.
\end{align}
We now briefly compare (\ref{eq:muapprox}) to the integral in the right hand side of (\ref{Dunweight}): Note that in the volume-preserving case $a_i=\ell_i$ and the summand in the RHS of (\ref{eq:muapprox}) becomes $\ell_i/\ell^t_i\approx 1$ for sufficiently fine meshes.
If $\Phi^t$ is affine on the support of $\varphi_i^0$, then $\ell_i/\ell^t_i=1$;  with increasing nonlinearity of the map, a good estimation requires finer mesh elements.
The factor $\text{area}(e^t)\; \frac13$ also appears in (\ref{Dunweight}) and thus (\ref{eq:muapprox}) approximately reduces to (\ref{Dunweight}) in the volume-preserving case.

Naturally the symmetry and sparseness properties of the resulting $\hat{D}^t$ are retained.  We will see in sections \ref{sec:collnonadap} and \ref{sec:colladap} that this is the only change required;  in particular, the expressions for $\alpha$ do not change at all for non-volume-preserving $\Phi^t$ and general initial measure $\mu^0$.

The next two subsections describe two different approaches to computing the $\alpha^i_k$.

\subsubsection{Collocation on non-adapted meshes}
\label{sec:collnonadap}

In this variant, the nodes $x^t_i$ at time $t\in\cT$ are in general unrelated to $x^s_i$ for some other time $s\in\mathcal{T}$,  in particular $\{x^t_i\}_{t\in \mathcal{T}}$ need not be on a trajectory of $\Phi^t$.  We wish to obtain the matrix $\alpha$ by collocation. To this end, in the case where $\Phi^t$ is volume-preserving, using (\ref{alpha}) we should solve the following problem:
\begin{equation}
\label{H1gal}
{\langle \Phi^t_*\varphi_i,\delta_{x^t_m}\rangle}=\sum_{k=1}^n \alpha^i_k{\langle \varphi^t_k,\delta_{x^t_m}\rangle},\qquad\mbox{for all $1\le i,m\le n$},
\end{equation}
where the $\delta_{x^t_m}$ are Dirac deltas. Since we are using a nodal basis $\varphi^t_1,\ldots,\varphi^t_n$ at each $t\in\mathcal{T}$, we have $\langle \varphi^t_k,\delta_{x^t_m}\rangle=\varphi^t_k(x^t_m)=\delta_{k,m}$ (Kronecker delta) and so the right hand side of (\ref{H1gal}) is simply $\alpha_m^i$.
Thus we obtain the explicit expression
 \begin{equation}
\label{H1coll_alpha}
\alpha_m^i = \Phi^t_*\varphi_i(x^t_m) = \varphi_i(\Phi^{-t}(x^t_m)).
\end{equation}
It is clear that $\alpha^i_k$ is nonnegative and summing the right hand side of (\ref{H1coll_alpha}) over $i$ we see that $\sum_i \alpha^i_k=1$ for each $k=1,\ldots,n$;  thus the matrix $\alpha=(\alpha^i_k)_{ik}$ is column-stochastic.

In the case of initial probability measure $\mu^0$ and non-volume-preserving $\Phi^t$, let $h_{\mu^t}=d\mu^t/d\ell$ be the density of $\mu^t$ with respect to Lebesgue.
The inner products in (\ref{H1gal}), now weighted according to $\mu^t$ become
\begin{equation}
\label{H1coll}
{\Phi^t_*\varphi_i(x^t_m)h_{\mu^t}(x^t_m)}=\sum_{k=1}^n \alpha^i_k{\varphi^t_k(x^t_m)h_{\mu^t}(x^t_m)},\qquad\mbox{for all $1\le i,m\le n$}.
\end{equation}
Cancelling $h_{\mu^t}(x^t_m)$ on both sides, we again obtain the simple formula (\ref{H1coll_alpha}) for $\alpha^i_m$.


%
%
%

\subsubsection{Collocation on adapted meshes}
\label{sec:colladap}
We again compute the matrix $\alpha$ by collocation, but instead of choosing a mesh at $t\in\mathcal{T}$ independently of the initial mesh we instead define the nodes of the mesh at time $t$ by $x_i^t=\Phi^tx_i^0$, namely the images of the nodes at the initial time.
The expression (\ref{H1coll_alpha}) now becomes
\begin{equation}
\label{H1coll_alpha_adap}
\alpha^i_m=\varphi_i(\Phi^{-t}x_m^t)=\varphi_i(x_m)=\delta_{i,m},
\end{equation}
  thus $\alpha$ is the $n\times n$ identity matrix.

This approach has several major advantages, particularly when one does not have a full model of the dynamical system, but instead only has access to $n$ trajectories $\{x_i^t\}_{t\in \mathcal{T}}, i=1,\ldots,n$, although the approach is also very effective even with a full model which can generate trajectories.
The algorithm is as follows.
\vspace{.3cm}

\noindent\textbf{Algorithm 1:} (Input -- trajectories $\{x_i^t\}_{t\in \mathcal{T}}, i=1,\ldots,n$; Output -- approximate eigenfunctions of dynamic Laplacian)
\begin{enumerate}
 \item for each $t\in \mathcal{T}$,
 \begin{enumerate}
  \item triangulate  the set of nodes $\{x_1^t,\ldots,x_n^t\}$, e.g.~by a Delaunay triangulation or, in case of highly nonlinear $\Phi^t$, using some alpha complex  of the nodes \footnote{A Delaunay triangulation always yields a triangulation of the convex hull of a given set of points.  In the case of highly nonlinear $\Phi^t$, this may lead to large triangles connecting nodes which a far apart.  In contrast, alpha complexes \cite{EdKiSe:83a} yield triangulations with triangles being constructed only locally.  In MATLAB, alpha complexes can be computed using \texttt{alphaShape}/\texttt{alphaTriangulation}.}.
 \item compute $D^t$ using (\ref{Dunweight}) or (\ref{Dweight}) as appropriate,
 \end{enumerate}
 \item compute $M$ using (\ref{Deqn}) or (\ref{Dweightedeqn}) as appropriate,
 \item form $\frac{1}{|\mathcal{T}|}\sum_{t\in\mathcal{T}}D^t$ and solve the eigenproblem (\ref{weakeigen_galerkin}).
 \end{enumerate}
The formation of the triangulations and the computation of each $D^t$ and the single $M$ are  very fast.  In fact, in our experiments, solving the sparse, symmetric eigenproblem (\ref{weakeigen_galerkin}) for the eigenvalues near to zero, while also fast, is the most time-consuming step.
We emphasise that no computation of the transfer operator is required as it is implicit in the trajectory information.

\subsubsection{Computing with incomplete trajectory data}
\label{sec:missing}

We now consider the situation where we have trajectory data $\{x_i^t\}, i=1,\ldots,n, t=1,\ldots,T$, but some of these data points are ``empty'' (we don't have the position data at certain times).
We discuss how to modify the method of section \ref{sec:colladap} to handle this situation.
Let $\cT_i=\{t\in\mathcal{T}: x_i^t\mbox{ exists}\}$ and $I_t=\{i\in \{1,\ldots,n\}: x_i^t\mbox{ exists}\}$.
Respectively, these sets are the \emph{times} at which we have positions for particle $i$, and \emph{indices} of those particles for which we have positions at time $t$.
We will need to compare distances between trajectories, so we assume that $\cT_i\cap \cT_j\neq \emptyset$ for all pairs $1\le i,j\le n$;  that is, for each point pair $(i,j)$ there exists at least one time $t$ (depending on the pair $(i,j)$) for which $i,j\in I_t$.
Note that this in particular includes situations where there is no single time at which all data pairs $i,j$ are available.

At each $t\in \mathcal{T}$ we triangulate the points $\{x_i^t :i\in I_t\}$.  Using the corresponding basis $\{\varphi_i^t:i\in I_t\}$, we compute ${D}^t$  as
\begin{equation}
\label{Dmissing}
D^t_{ij}:=\left\{
            \begin{array}{ll}
              \int_{\Phi^t(\M)} \nabla\varphi_i^t\cdot \nabla\varphi_j^t\ d\ell, & \hbox{$i,j\in I_t$}, \\
              0, & \hbox{otherwise.}
            \end{array}
          \right.
\end{equation}


For the mass matrix $M$, first note that
\[
\int_\M \varphi_i^0\cdot \varphi_j^0\ d\ell=\int_{\Phi^t(\M)} \varphi^0_i\circ\Phi^{-t}\cdot \varphi^0_j\circ\Phi^{-t}\ d\ell.
\]
Motivated by this, we set
\begin{equation}
\label{eqn25a}
M^t_{ij}:=\left\{
            \begin{array}{ll}
              \int_{\Phi^t(\M)} \varphi_i^t\cdot \varphi_j^t\ d\ell, & \hbox{$i,j\in I_t$}, \\
              0, & \hbox{otherwise,}
            \end{array}
          \right.
\end{equation}
and use
\begin{equation}
\label{Mmissing}
M := \frac{1}{|\cT|}\sum_{t\in \cT} M^t.
\end{equation}
The rest of the calculations proceed as in the previous sections.

In the non-volume-preserving case, replace $d\ell$ in (\ref{Dmissing}) and (\ref{eqn25a}) at index $t$ with $d\mu^t$.
We described how to compute the entries in $D^t$ in this case in Section \ref{sect:nvp}.
For $i,j\in I_t$ the entries of the time-dependent mass matrices can be approximated by
\begin{equation}
\label{eq:nvpM}
M_{ij}^t  = \int_{\Phi^t(\M)} \varphi_i\;\varphi_j\;d\mu^t
 \approx \int_{\Phi^t(\M)} \varphi_i\;\varphi_j\;\tilde h^t\; d\ell
 = \sum_{k=1}^n \frac{a_k}{\ell_k^t}\int_{\Phi^t(\M)} \varphi_i\;\varphi_j\; \varphi_k\; d\ell
\end{equation}
and the integrals $\int_{\Phi^t(\M)} \varphi_i\;\varphi_j\; \varphi_k\; d\ell$ can be computed exactly, using the fact that the $\varphi_i$ are affine on each triangle.


\vspace{.3cm}

\noindent\textbf{Algorithm 2:} (Input -- incomplete trajectories $\{x_i^t\}_{t\in \mathcal{T}}, i=1,\ldots,n$;  Output -- approximate eigenfunctions of dynamic Laplacian)
\begin{enumerate}
 \item for each $t\in \mathcal{T}$,
 \begin{enumerate}
  \item  triangulate  the set of nodes $\{x_i^t:i\in I_t\}$ (Delaunay or alpha complex),
 \item compute ${D}^t$ as per (\ref{Dmissing}) or the equivalent version of (\ref{Dweight}), as appropriate,
 \item compute $M^t$ as per (\ref{eqn25a}) or (\ref{eq:nvpM}), as appropriate,
     \end{enumerate}
\item form $\frac{1}{|\mathcal{T}|}\sum_{t\in\mathcal{T}}D^t$ and use $M$ from (\ref{Mmissing}) to solve the eigenproblem (\ref{weakeigen_galerkin}).
 \end{enumerate}

\subsubsection{Discussion of the three approaches}

The three variants of the FEM approach to the eigenproblem (\ref{eq:eigenproblem})  in the preceeding sections come with different advantages/disadvantages which we now briefly discuss.

The Cauchy-Green approach from Section~\ref{sec:CG} only requires a single triangulation of the underlying manifold $\M$, avoiding potential complications with a complicated geometry of the image manifold(s) $\Phi^t(\M)$.
Additionally, this leads to fewer nonzero entries in the stiffness and mass matrices in comparison to the other two methods, so that the eigenproblem is solved faster (cf.\ the experiments in the next section).  It exploits higher order information on $\Phi^t$ via the Cauchy-Green tensor, although this comes at a higher computational cost since $D\Phi^t$ has to be computed at every quadrature node (of which typically several are required within each mesh element).  Also, for highly nonlinear flows, $D\Phi^t$ tends to vary very rapidly in space, which sometimes renders the quadrature on the elements challenging (cf.~the experiment on the rotating double gyre in Section~\ref{exp:rotating_double_gyre}).

In the transfer operator approach (Section~\ref{sec:collnonadap}), we may use a single triangulation only in the case that $\M=\Phi^t(\M)$, since then we can choose the same basis of $V_n^t$ for each $t\in\mathcal{T}$.
On the other hand, a major advantage is the removal of the need to compute the derivative of the flow map.
Instead, we need to evaluate the inverse flow map, which is possible if $\{x^t_m\}_{t\in\mathcal{T}}$ arises as a trajectory, or in other settings where the preimages are known or can be easily computed.  The matrices $D_{kl}^t$ and $(\alpha_k^i)_{ik}$ are sparse with the number of nonzero entries per row/column determined by how many elements support a basis function (six in the case of triangular elements in 2D), although their product $\hat D^t$ may have considerably more nonzero entries than $D_{kl}^t$ and $(\alpha_k^i)_{ik}$, which would make solving the eigenproblem more expensive.

In the adaptive transfer operator approach (Section~\ref{sec:colladap}), we need to construct a separate triangulation for every $t\in \mathcal{T}$ (even in the case $\M=\Phi^t(\M)$).
For each of our examples we find $|\mathcal{T}|=2$ already yields good results. There is \emph{no estimation of the transfer operator} through the $(\alpha_k^i)_{ik}$, so the stiffness matrix calculation is completely standard for FEM.  Further, the only input data are the positions of the nodes at each time instance, which, e.g., is readily available in most experimental or real-world settings.  The number of nonzeros of the overall stiffness matrix $\sum_t \hat D^t$ grows linearly with $|\mathcal{T}|$, so that the solution of the eigenproblem becomes more expensive.

\subsection{Extraction of (multiple) coherent sets from eigenfunctions}
\label{sec:extract}

By construction of the dynamic Laplacian, and the theory in Section \ref{sec:background} and the references therein, an optimal 2-partitioning of the domain into 2 coherent sets is given in terms of level sets of the second eigenfunction (often corresponding to values near zero).
Extending this idea heuristically to multiple coherent sets, one searches for multiple regions defined by several level sets.
A straightforward way to extract these is by computing an optimal level set for each eigenfunction, cf.~\cite{F15,FJ15}.
In general, this does not guarantee that the resulting coherent set is connected, e.g.~it might be the union of two vortices in a fluid flow.
This is similar to the situation in the context of time-independent systems where one is interested in \emph{almost-invariant} sets, methods based on the zero level set \cite{DJ99} (the sign structure) or other thresholds \cite{F05} of the relevant eigenvectors.

Various approaches for the extraction of ``individual'' coherent sets have been proposed:  For almost-invariant sets, a least-squares assignment strategy \cite{DeHuFiSc00a} have been proposed.  Later, techniques drawn from spectral graph partitioning \cite{chan,alpert} based on clustering embedded eigenvectors via fuzzy c-means clustering \cite{Froyland:2003jj,F05} have been proposed.  In our examples, we use a variant of this latter approach by employing MATLAB's \texttt{kmeans} command.

In many cases, gaps in the spectrum of the dynamic Laplacian can be used in order to obtain a heuristic on the number of coherent sets: If there is a large gap after the $k$-th eigenvalue, then one should search for $k-1$ connected coherent sets\footnote{for Neumann boundary conditions, since the constant eigenfunction at $\lambda=0$ does not contribute.}.  This heuristic is motivated by perturbation arguments for the spectra of linear operators, cf.\ \cite{DJ99,F05,KK16}.

\subsection{Link with graph-based methods}
\label{sec:graph}

The transfer operator approach of \S\ref{sec:colladap} can be interpreted in terms of graph Laplacians.
For any $t$ the matrix $\hat{D}^t$ from (\ref{Deqn}) or (\ref{Dweightedeqn}) is clearly symmetric, and also has zero column and row sums.
For the latter property, note that summing both sides in (\ref{Deqn}) over $j$ yields
$$\sum_j \hat{D}^t_{ij} = \int_{\Phi^t(\M)}\nabla\left(\Phi^t_{*}\left(\sum_j\varphi_j\right)\right)\bullet\nabla\left(\Phi^t_{*} \varphi_i\right)\ d\ell=0,$$
as $\sum_j \varphi_j\equiv \mathbf{1}$ by construction (and $\Phi^t_*\mathbf{1}=\mathbf{1}$ and $\nabla\mathbf{1}=\mathbf{0}$).

Moreover, the diagonal entries of $\hat{D}^t$ are clearly positive.
It is known (see Proposition 3 \cite{vanselow}, also p5 \cite{alouges}) that the off-diagonal entries of $\hat{D}^t$ are non-positive for two-dimensional Delaunay meshes and for three-dimensional Delaunay meshes when all dihedral angles of the tetrahedra are less than $\pi/2$.
Indeed, in all of our numerical experiments, including in three dimensions, we observe all interior triangles/tetrahedra yield non-positive entries.
In view of the above, we proceed under the assumption that all off-diagonal entries of $\hat{D}^t$ are all non-positive for the purposes of relating our numerical eigenproblem to a graph Laplacian eigenproblem.

Consider the graph $G^t$ which is our triangulation of images of data points at time $t$.
Because of the symmetry, zero column sum, and sign structure of $D^t$ discussed above, we can write $\hat{D}^t=\Pi^t-W^t$, where
$$W^t_{ij} = \left\{
            \begin{array}{ll}
              -\hat{D}^t_{ij}, & \hbox{$i\neq j$,} \\
              0, & \hbox{$i=j$,}
            \end{array}
          \right.$$
are the \textit{weights} assigned to edges $(i,j)$ in our graph $G^t$, and $\Pi^t_{ii}=\hat{D}^t_{ii}$ is a diagonal matrix with the $i^{\rm th}$ diagonal element equal to the \textit{degree} of node $i$ in $G^t$, namely $\Pi^t_{ii}=-\sum_j \hat{D}^t_{ij}$.
Thus, we can immediately interpret $D^t$ as an (unnormalized) graph Laplacian \cite{Chung}.

We now interpret these off diagonal arc weights in terms of the trajectory dynamics.
Suppose that at time $t$, the positions of two nearby, distinct, evolved data points $\Phi^t(x_i), \Phi^t(x_j)$ are sufficiently close that $(i,j)$ is an arc in $G^t$ (this means $\Phi^t(x_i), \Phi^t(x_j)$ are sufficiently close that the supports of $\varphi^t_i$ and $\varphi^t_j$ intersect).
Then the corresponding entry of $\hat{D}^t_{ij}$ will be negative
and will decrease in magnitude as the distance between these points increases.
This is because the integrand in (\ref{Deqn}) or (\ref{Dweightedeqn}) is proportional to a product of two approximately linearly decreasing gradients, which outweighs the approximately linear increase in the integration domain given by the mesh volume.
If we continue to increase the distance between $\Phi^t(x_i), \Phi^t(x_j)$, eventually the arc $(i,j)$ will no longer exist in $G^t$ because the Delaunay mesh will find a more efficient way to mesh the points.
In summary, we see that the edge weights given by $\hat{D}^t_{ij}$ reflect distance between nearby points $\Phi^t(x_i), \Phi^t(x_j)$.

When we solve the problem $-\frac{1}{|\mathcal{T}|}\sum_{t\in\mathcal{T}} \hat{D}^{t}y=\lambda My$, we are averaging the graph Laplacians, normalized
by the symmetric, nonnegative mass matrix $M$.
There are two differences in this normalization to the standard graph Laplacian normalization, which is based simply on node degree.
Firstly, $M$ is not diagonal, although in any given row or column, the largest entry will lie on the diagonal.
Moreover, the small number of off-diagonal entries of $M$ coincide with the off-diagonal entries of $\hat{D}^0$ and correspond to arcs $(i,j)$ that are in the graph $G^0$.
Returning to the interpretation of off-diagonal entries as distances $\hat{D}^t_{ij}$, we see that normalizing by the mass matrix \emph{automatically handles nonuniformly distributed data}, because if initial points $x^0_i, x^0_j$ are far apart the value of $M_{ij}$ will be commensurately larger;  see Figure \ref{fig:Meiss_TO_sampling} for a dramatic illustration of this automatic handling.
This is in contrast to e.g.\ the approaches of \cite{FPG15,HKTH16}, where only the linear distance between points is taken into account, and not additionally the volume in phase space ``represented'' by those points.
This is important in the case where the data is not uniformly distributed in space, such as e.g.\ trajectory data from ocean drifters.

\section{Numerical experiments}
\label{sec:numerical}

In all subsequent experiments, we employ standard linear triangular/tetrahedral elements (i.e.~$P_1$ Lagrange elements) and an explicit Runge-Kutta scheme of order (4,5) with adaptive step size control as implemented in MATLAB's \texttt{ode45} integrator, using the relative and absolute error tolerance $10^{-3}$.  The derivative of the flow map is approximated using central finite differences.

\subsection{Experiment: The rotating double gyre}

\label{exp:rotating_double_gyre}

We start with the rotating double gyre flow introduced in
\cite{MoMe11a}, a Hamiltonian system with Hamiltonian $H=-\psi$, where $\psi$ is the stream function
\begin{align*}
\psi(x,y,t) &= (1-s(t))\psi_P(x,y) + s(t)\psi_F(x,y)\\
\psi_P(x,y) &= \sin(2\pi x)\sin(\pi y)\\
\psi_F(x,y) &= \sin(\pi x)\sin(2\pi y)
\end{align*}
and transition function
\[
s(t) = \left\{\begin{array}{cl} 0 & \text{for } t<0,\\ t^2(3-2t) & \text{for }t\in [0,1],\\ 1 & \text{for }t>1.\end{array}\right.
\]
On the square $\M=[0,1]^2$, the vector field exhibits two gyres with centers at $(\frac12,\frac12)$ and $(\frac32,\frac12)$ initially (at $t=0$) which rotate by $\pi/2$ during the flow time $T=1$.

\paragraph{Cauchy-Green approach.}

We use $\mathcal{T}=\{0,1\}$ and employ Gauss quadrature of degree 5 for the numerical integration of (\ref{eq:CG_1}).  A lower degree yields degraded results.  Indeed, as shown in Fig.~\ref{fig:Meiss_CG_tensor}, the (inverse) Cauchy-Green tensor is of quite large variation (cf.\ our earlier remarks).
\begin{figure}[htbp]
\begin{center}
\includegraphics[width=0.4\textwidth]{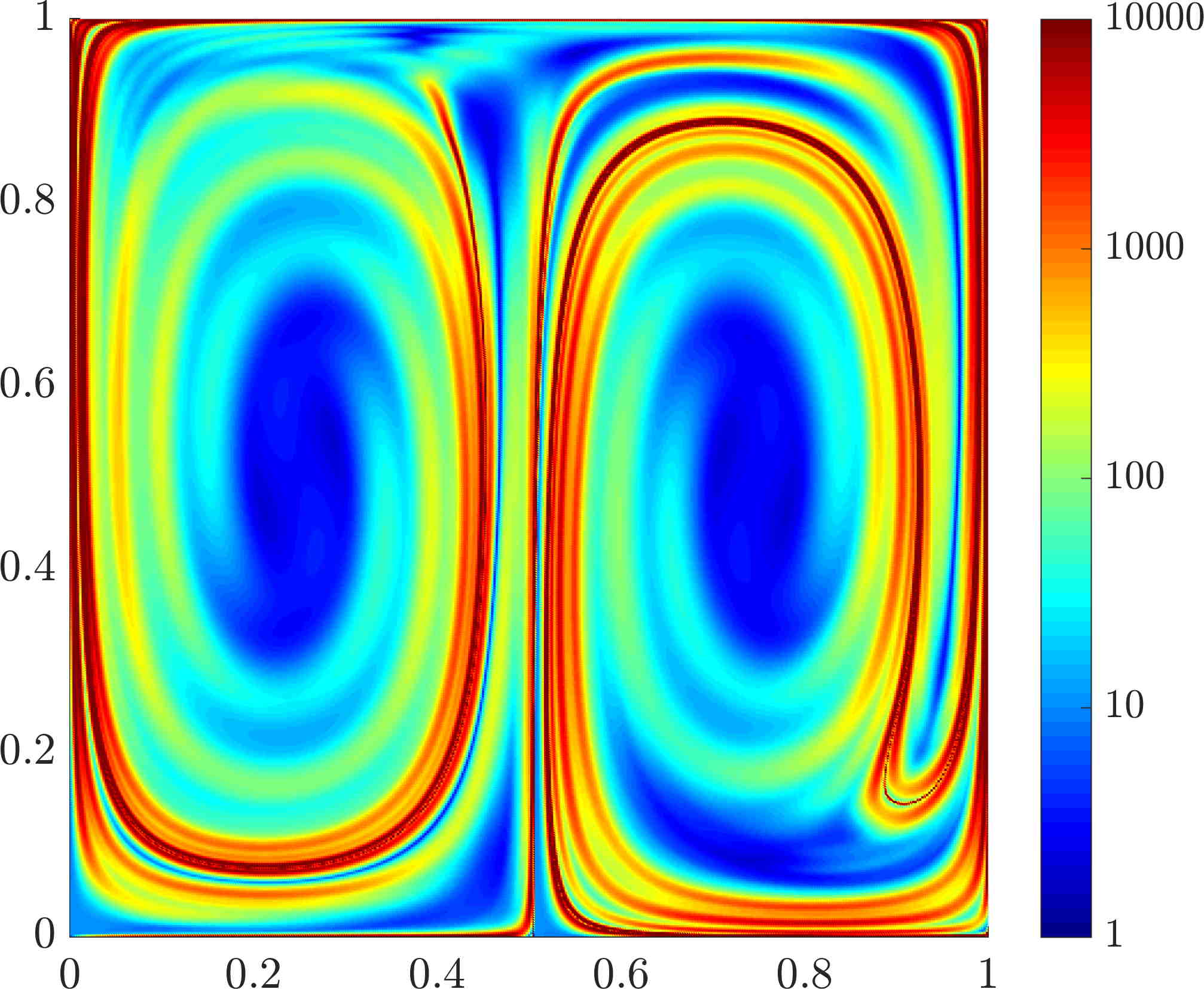}
\caption{Rotating double gyre: trace of the inverse Cauchy-Green tensor.}
\label{fig:Meiss_CG_tensor}
\end{center}
\end{figure}

 Figure~\ref{fig:Meiss_CG_spectra} shows the spectra on a triangulation of $25\times 25$ equidistant points (1152 triangles) as well as on the Delaunay triangulation of a set of 625 scattered nodes (1236 triangles).  There appears to be a gap after the fourth eigenvalue. The method of Section 3.2, however, has a gap after the third eigenvalue for both the regular and scattered data, and for this reason and to enable comparison with prior work \cite{FPG14}, we used the leading 3 eigenvectors in all cases. Figures~\ref{fig:Meiss_CG_u} and \ref{fig:Meiss_CG_r} show the 2nd and 3rd eigenvectors as well as the resulting decomposition into coherent sets obtained by kmeans-clustering of these eigenvectors as described in Section~\ref{sec:extract}.  Already at this comparatively low resolution, the results are broadly similar to those in \cite{FPG14}.  Note, however, that we need to evaluate $C_t^{-1}$ (for which we evaluate $D\Phi^t$ by finite differencing) at roughly $8000$ points, which takes around 2 seconds.  The assembly of the matrices takes 0.02 s and the solution of the eigenproblem 0.05 seconds\footnote{All timings measured on a dual core Intel i5 with 2.4 GHz and MATLAB R2016a.}.
\begin{figure}[htbp]
\begin{center}
\includegraphics[width=0.45\textwidth]{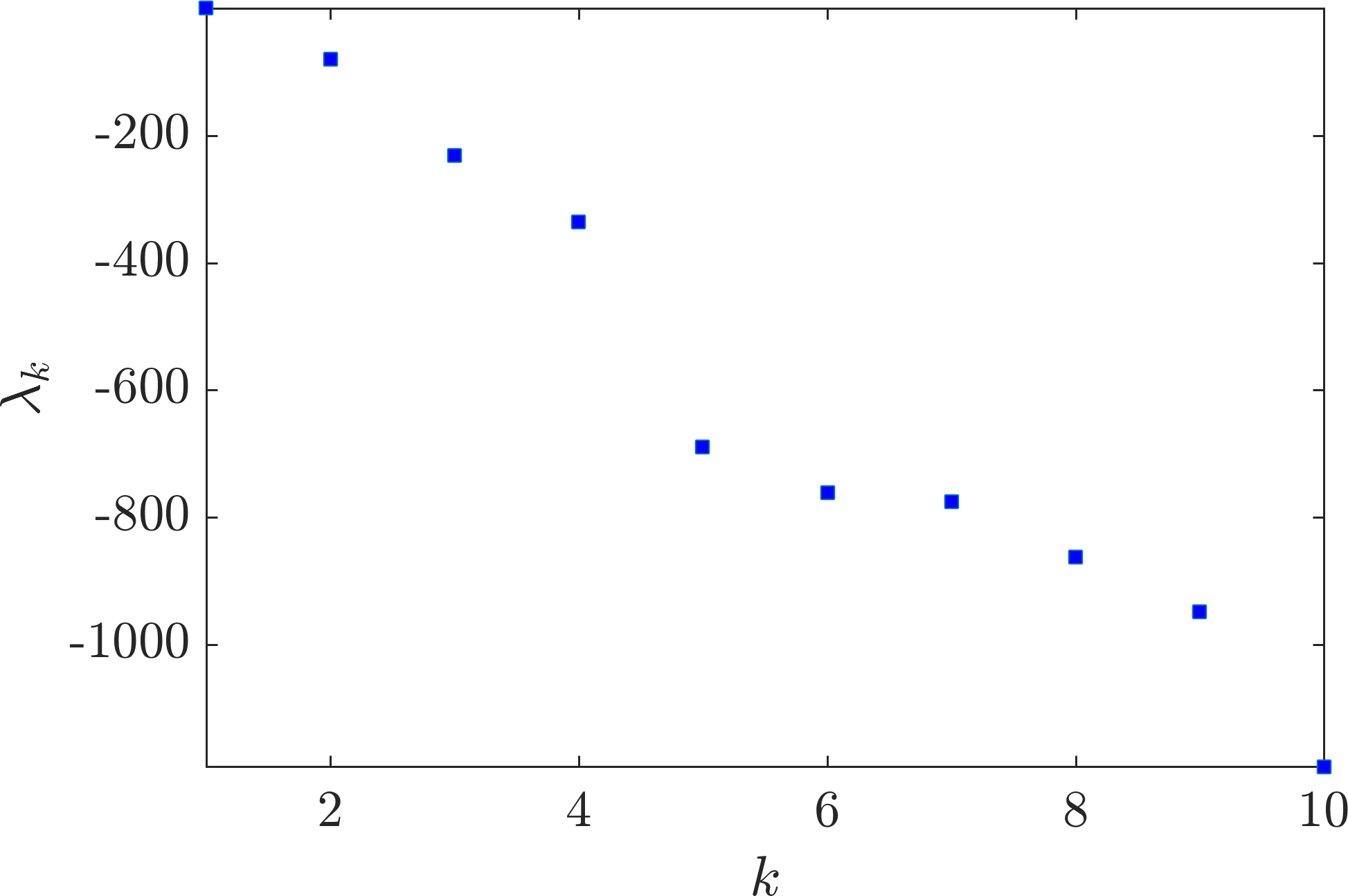}
\includegraphics[width=0.45\textwidth]{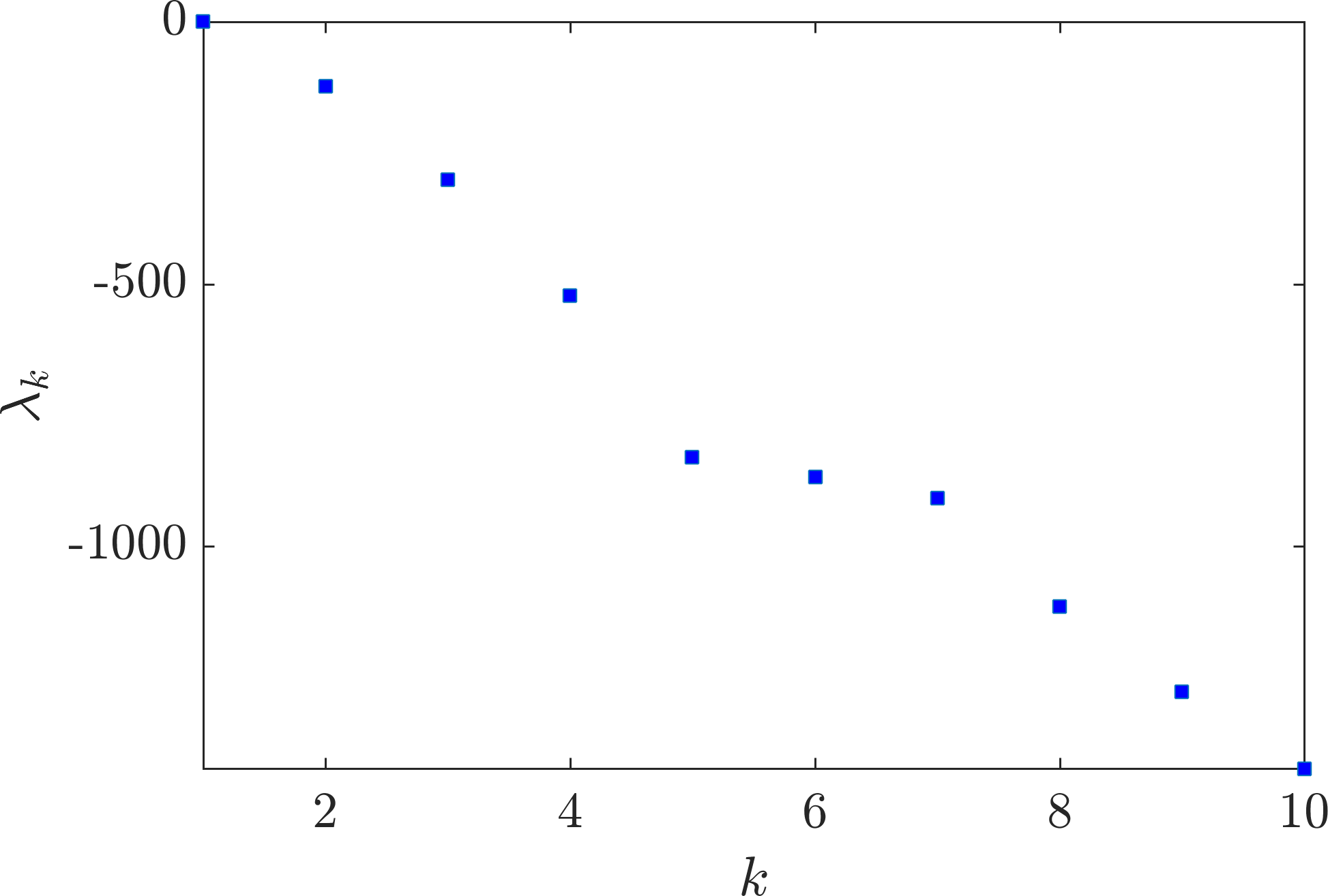}
\caption{Rotating double gyre: spectrum of the dynamic Laplacian on a triangulation of a regular $25\times 25$ grid (left) and of 625 scattered points (right) using the approach from Section~\ref{sec:CG}.}
\label{fig:Meiss_CG_spectra}
\end{center}
\end{figure}
\begin{figure}[htbp]
\begin{center}
\includegraphics[width=0.32\textwidth]{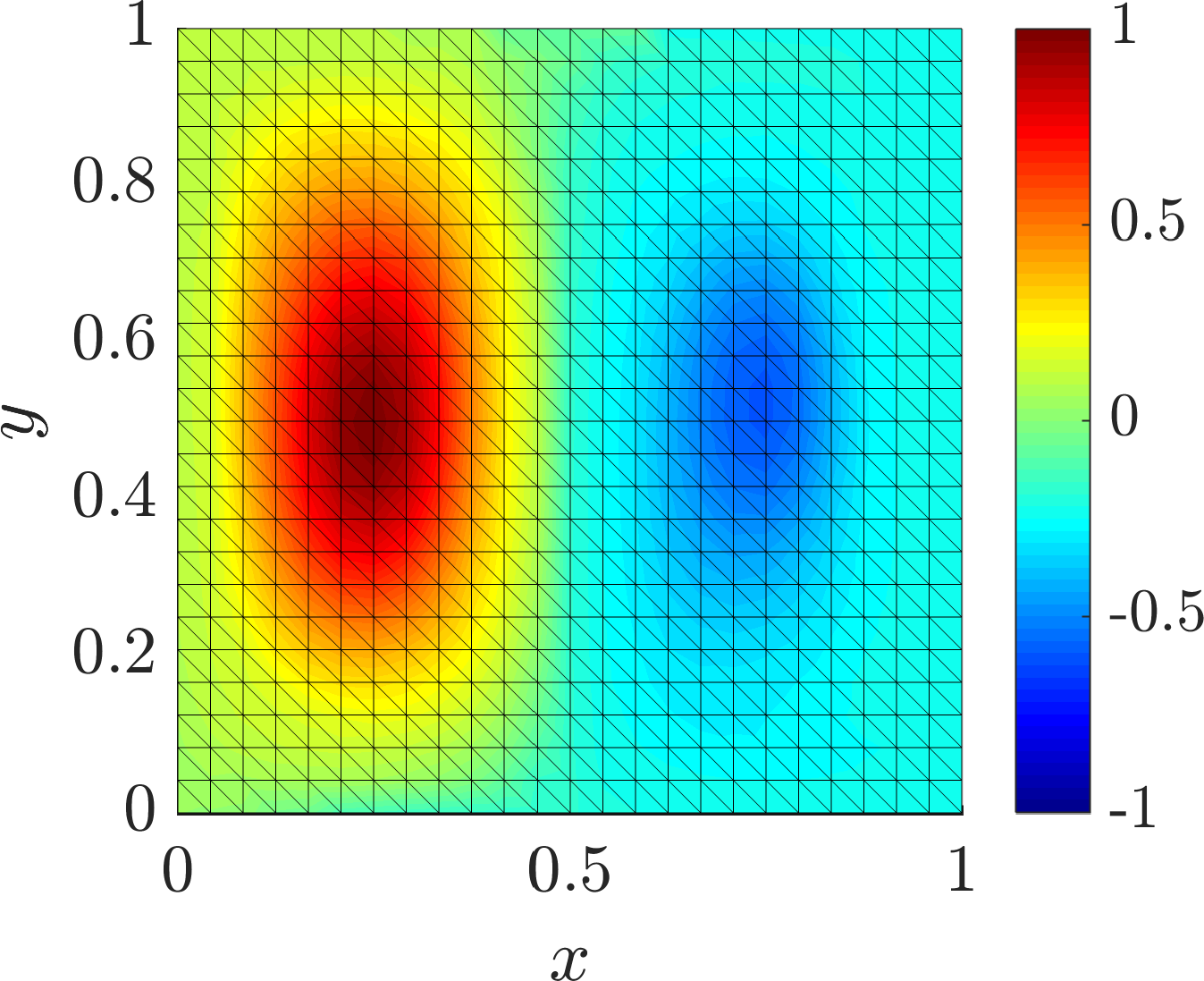}
\includegraphics[width=0.32\textwidth]{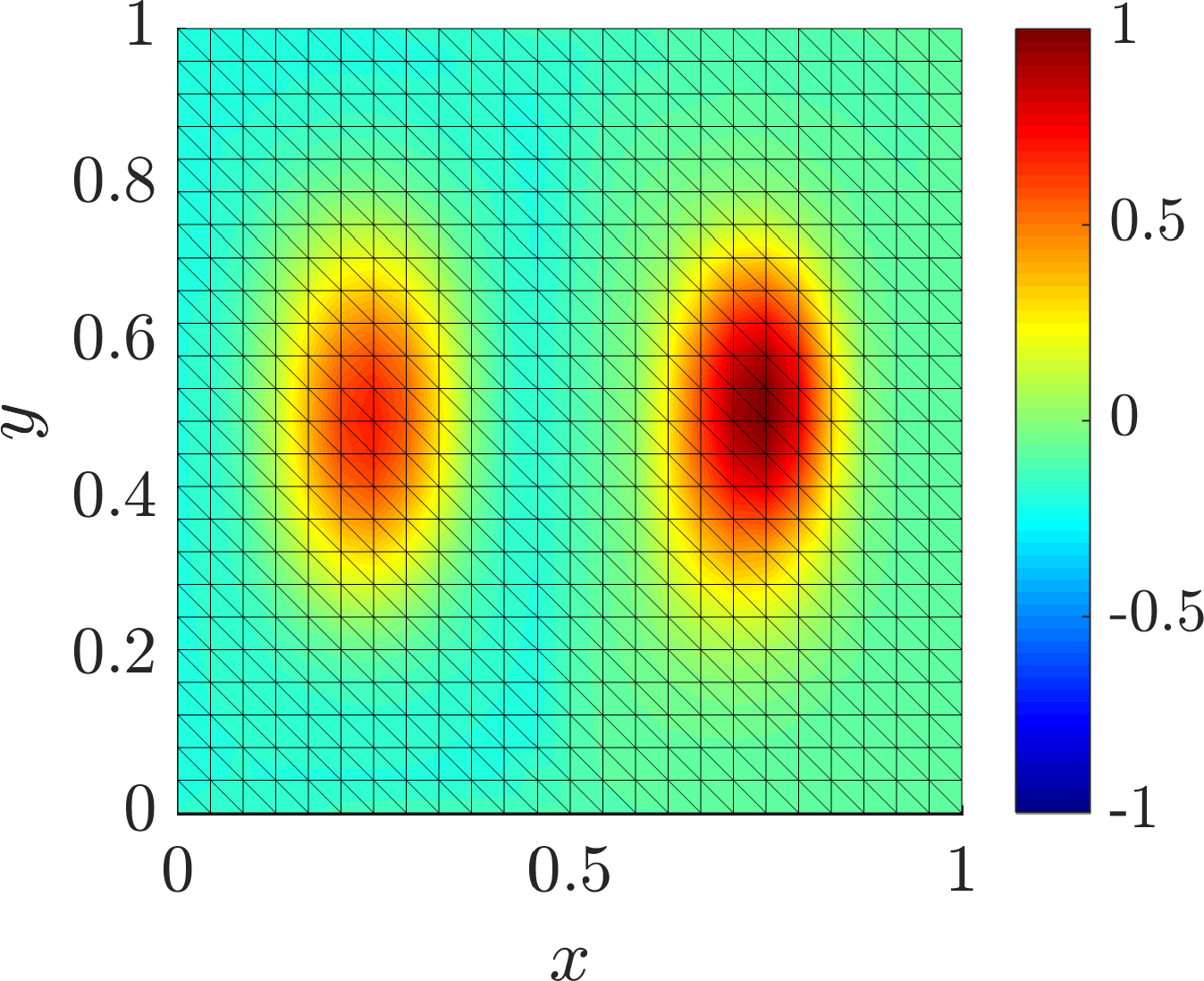}
\includegraphics[width=0.32\textwidth]{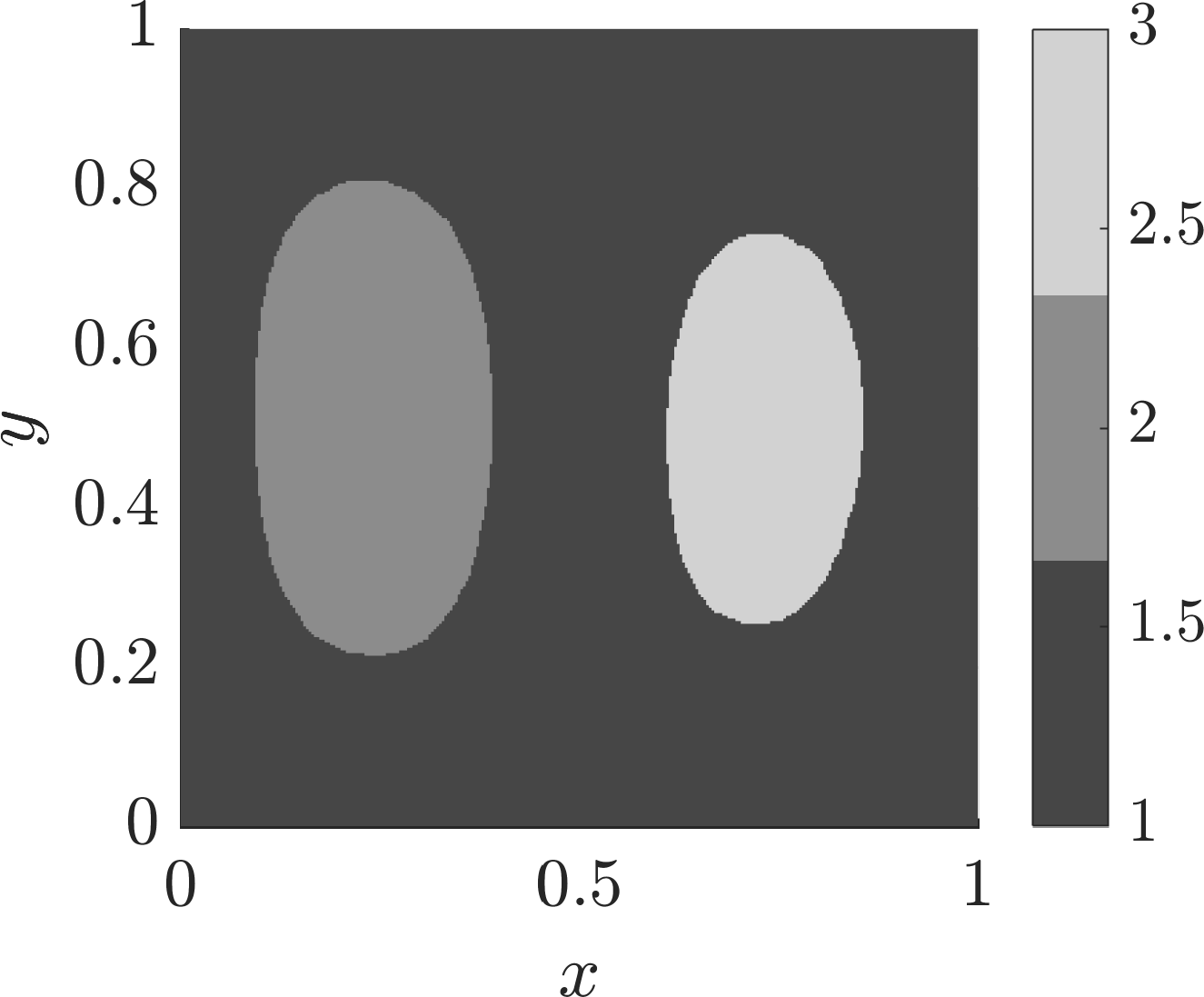}
\caption{Rotating double gyre: 2nd and 3rd eigenvector as well as the resulting coherent partition (from left to right) on a triangulation of a regular $25\times 25$ grid using the Cauchy-Green approach from Section~\ref{sec:CG}. The tensor $C_t^{-1}$ is evaluated at 8064 points.}
\label{fig:Meiss_CG_u}
\end{center}
\end{figure}
\begin{figure}[htbp]
\begin{center}
\includegraphics[width=0.32\textwidth]{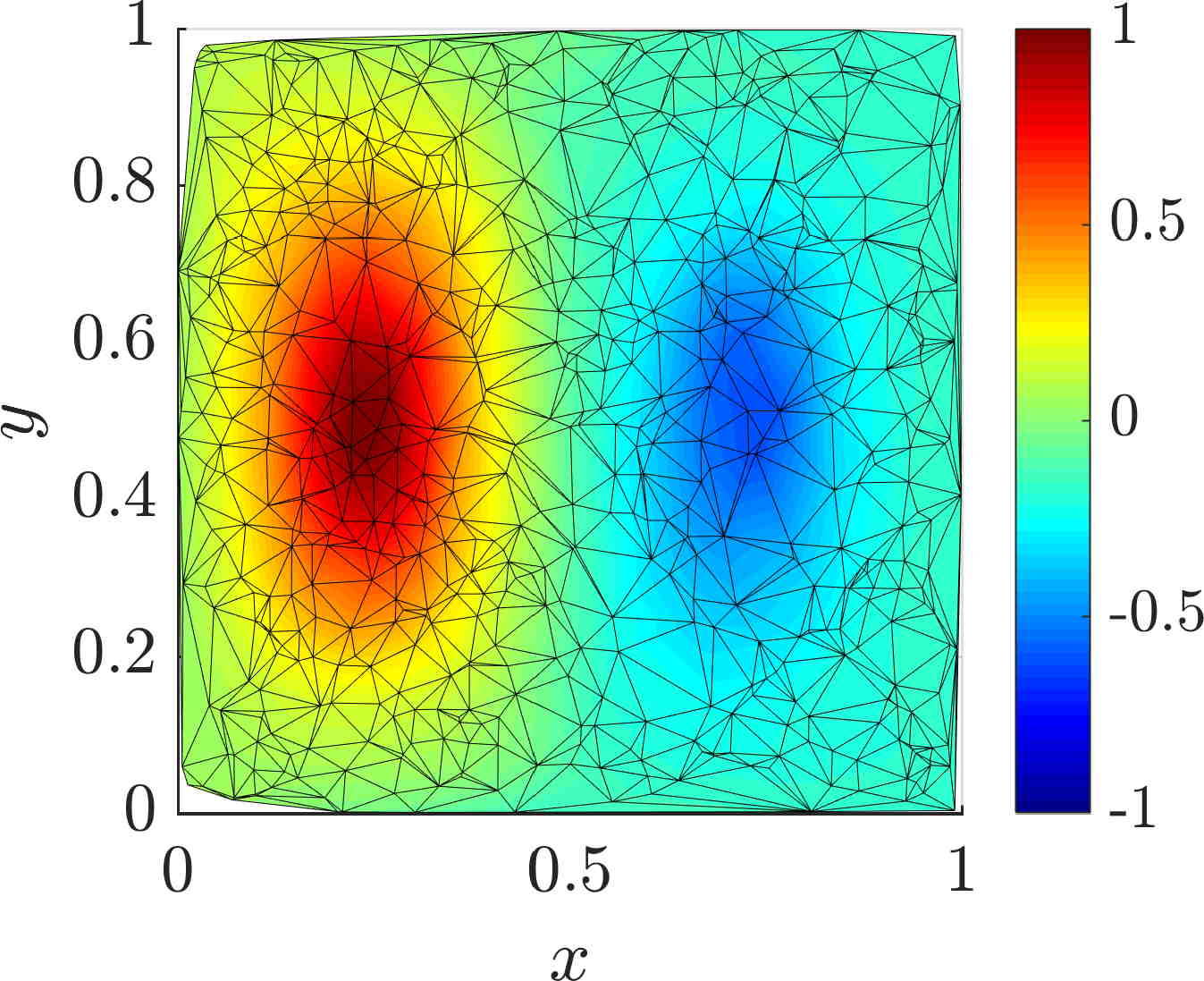}
\includegraphics[width=0.32\textwidth]{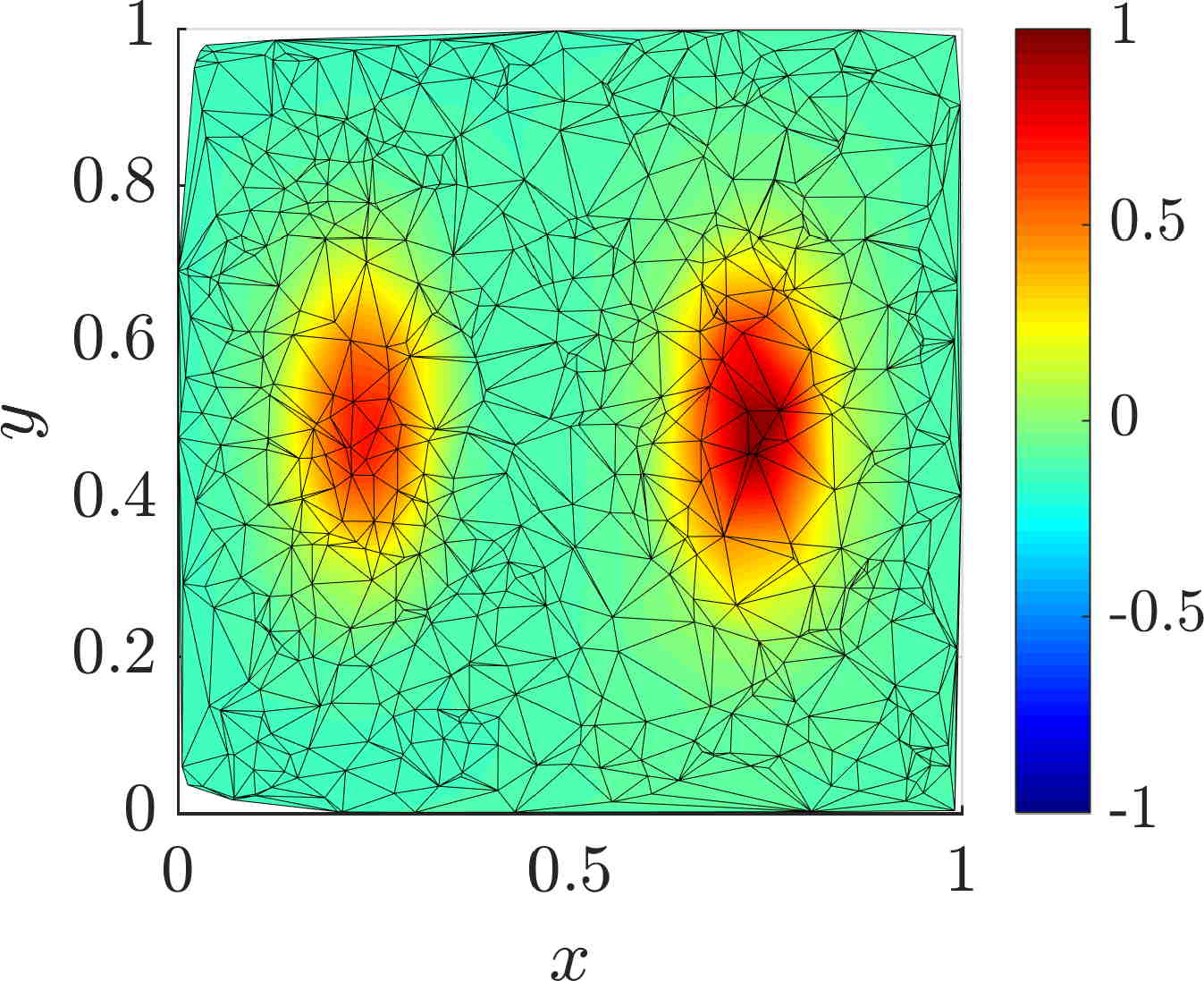}
\includegraphics[width=0.32\textwidth]{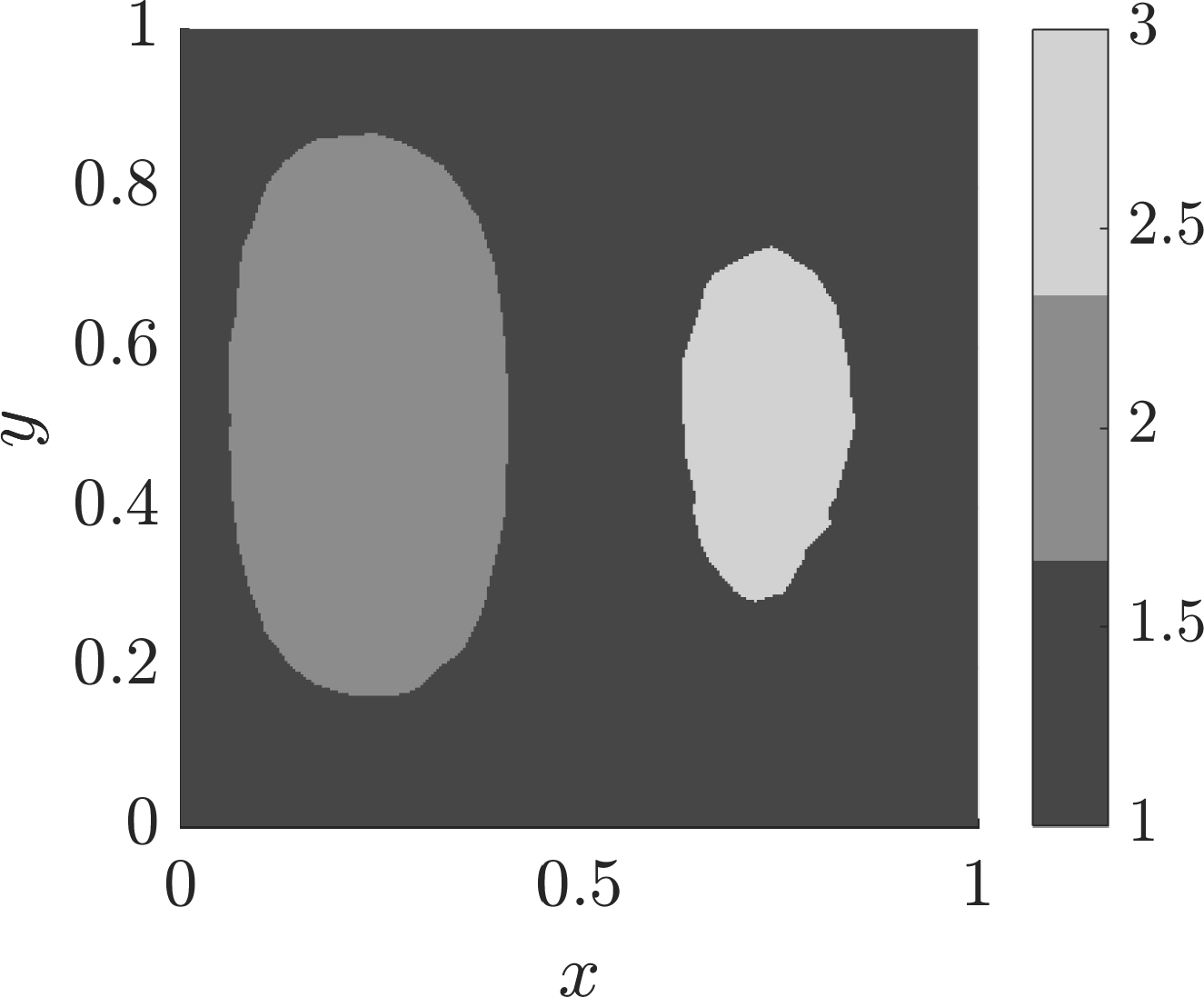}
\caption{Rotating double gyre: 2nd and 3rd eigenvector as well as the resulting coherent partition (from left to right) on a triangulation of a set of 625 randomly scattered points using the Cauchy-Green approach from Section~\ref{sec:CG}. The tensor $C_t^{-1}$ is evaluated at 8617 points.}
\label{fig:Meiss_CG_r}
\end{center}
\end{figure}

\paragraph{Transfer operator approach.}

We repeat the same experiment with $\mathcal{T}=\{0,1\}$ using the approach from Section~\ref{sec:collnonadap}.  Figure~\ref{fig:Meiss_TO_spectra} (left) shows the spectrum (with a small gap after the third eigenvalue), Figure~\ref{fig:Meiss_TO} the second and third eigenvectors as well as the resulting coherent 3-partition.  Here, the inverse flow map (without variational equation) had to be evaluated $25\cdot 25=625$ times only, which takes 0.1 s, the assembly of the matrices takes 0.02 s and the solution of the eigenproblem again 0.1 seconds.  Note that the only data that we input to our method is the initial and final positions of the 625 nodes.
\begin{figure}[htbp]
\begin{center}
\includegraphics[width=0.45\textwidth]{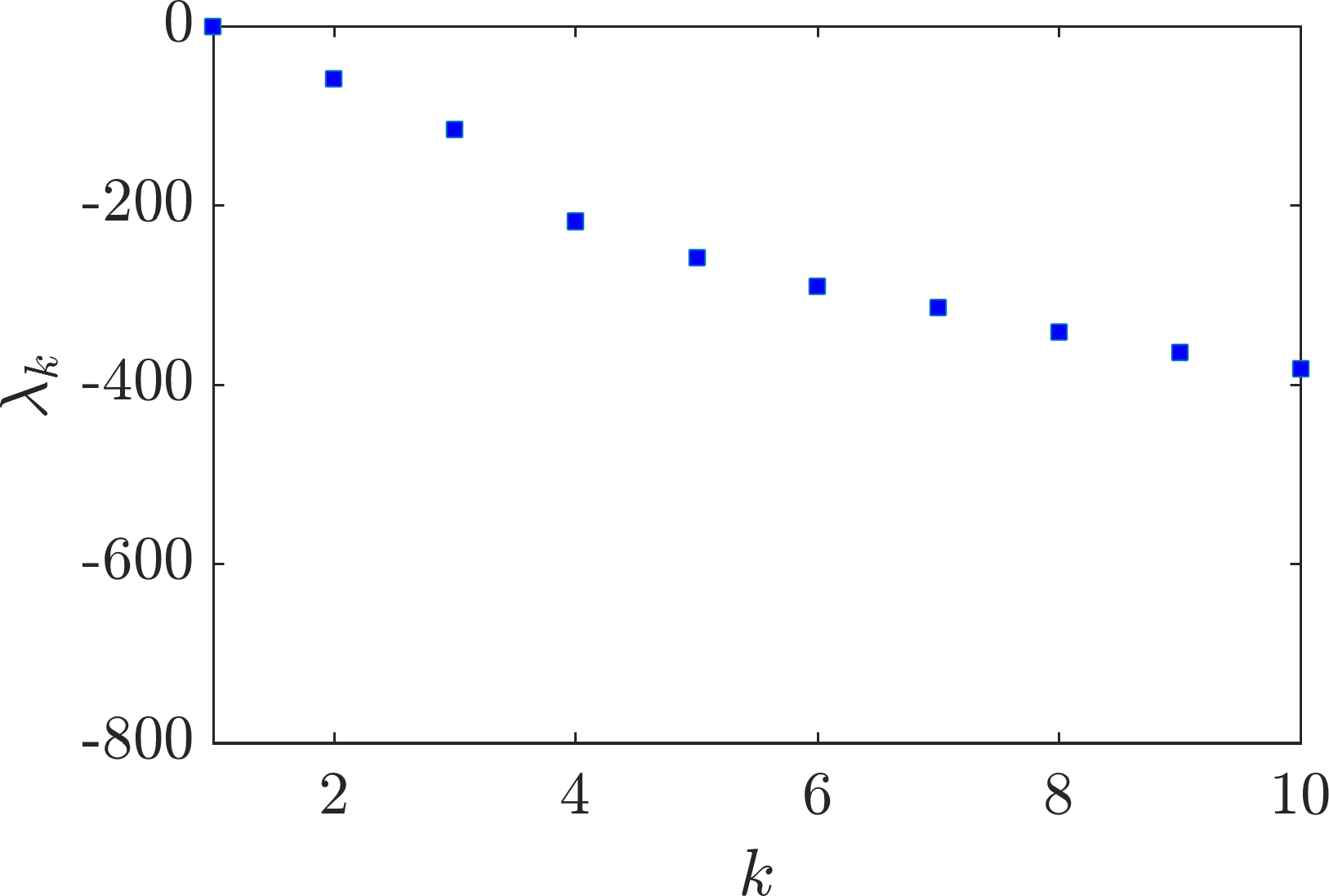}
\includegraphics[width=0.45\textwidth]{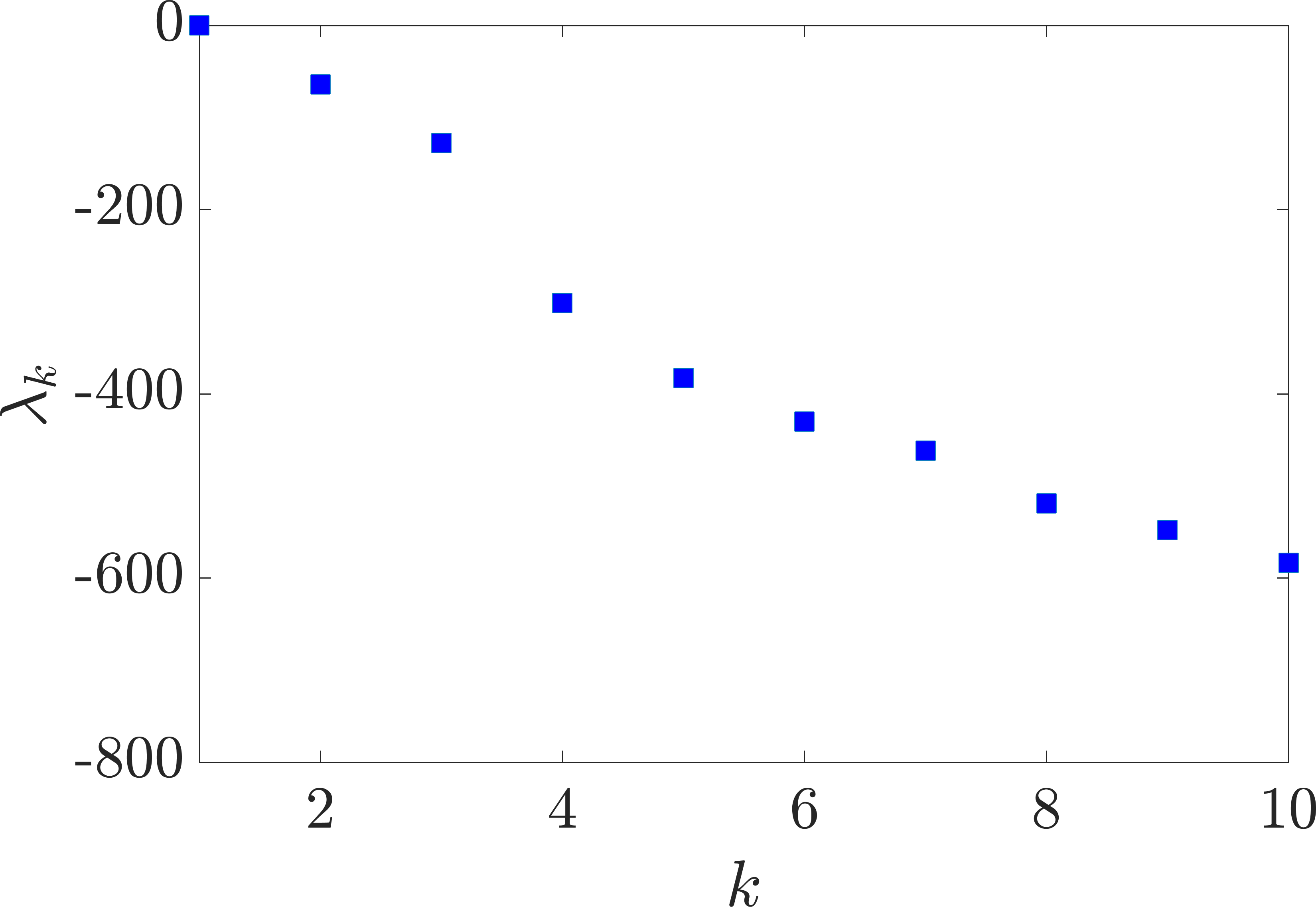}
\caption{Rotating double gyre: spectra of the dynamic Laplacian on a triangulation of a regular $25\times 25$ grid  using the transfer operator approach from Section~\ref{sec:collnonadap} (left) and from Section~\ref{sec:colladap}.}
\label{fig:Meiss_TO_spectra}
\end{center}
\end{figure}
\begin{figure}[htbp]
\begin{center}
\includegraphics[width=0.32\textwidth]{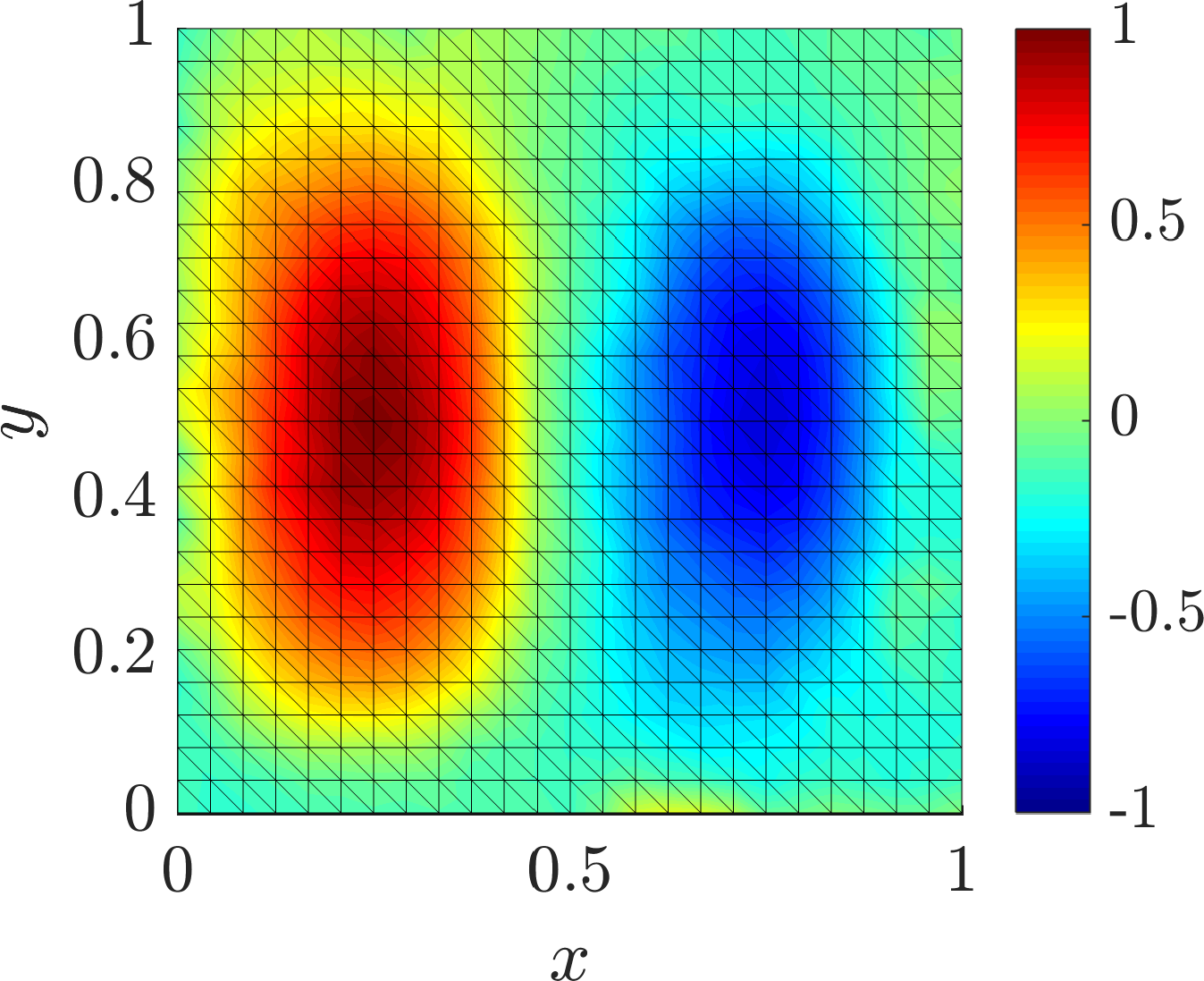}
\includegraphics[width=0.32\textwidth]{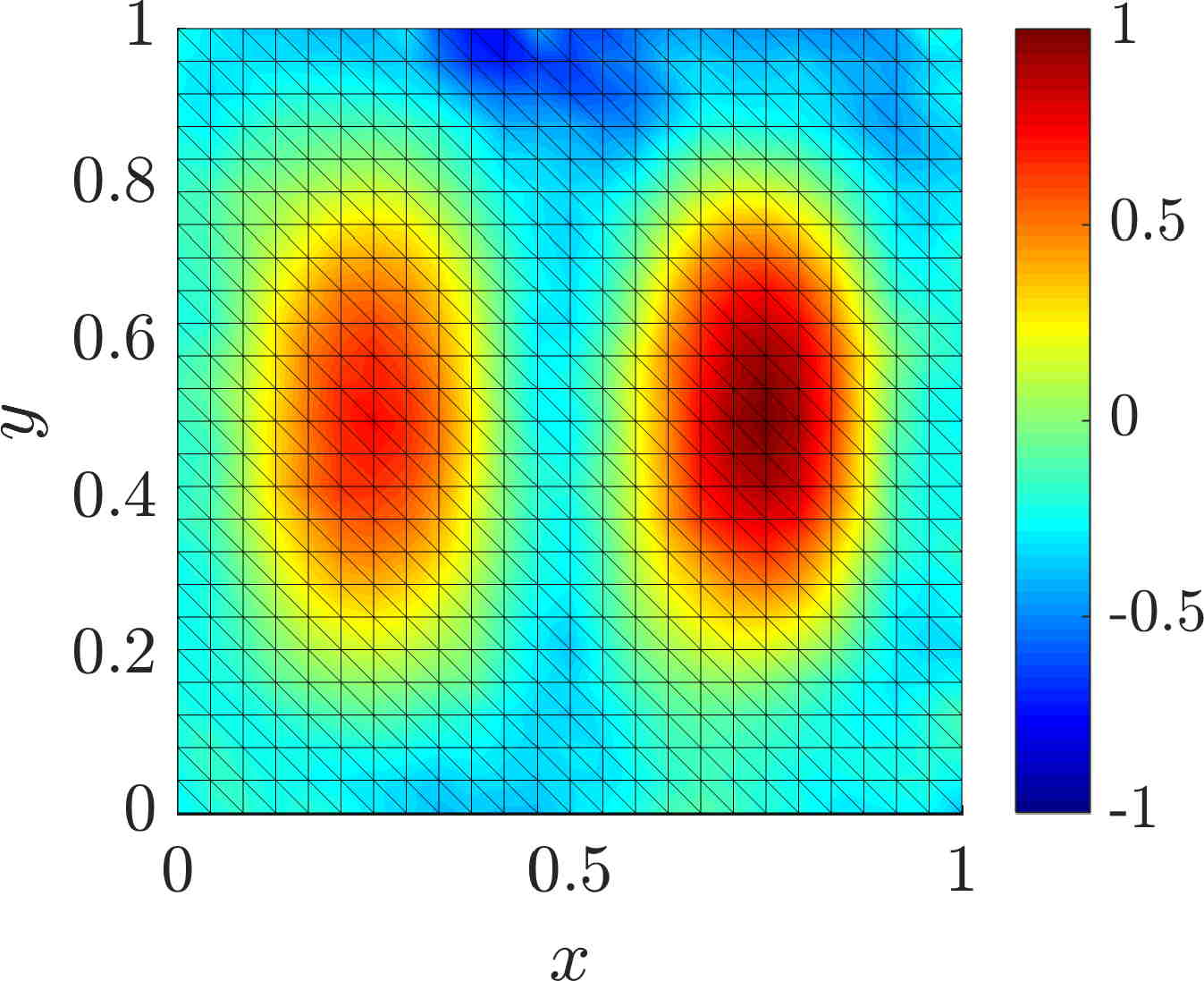}
\includegraphics[width=0.32\textwidth]{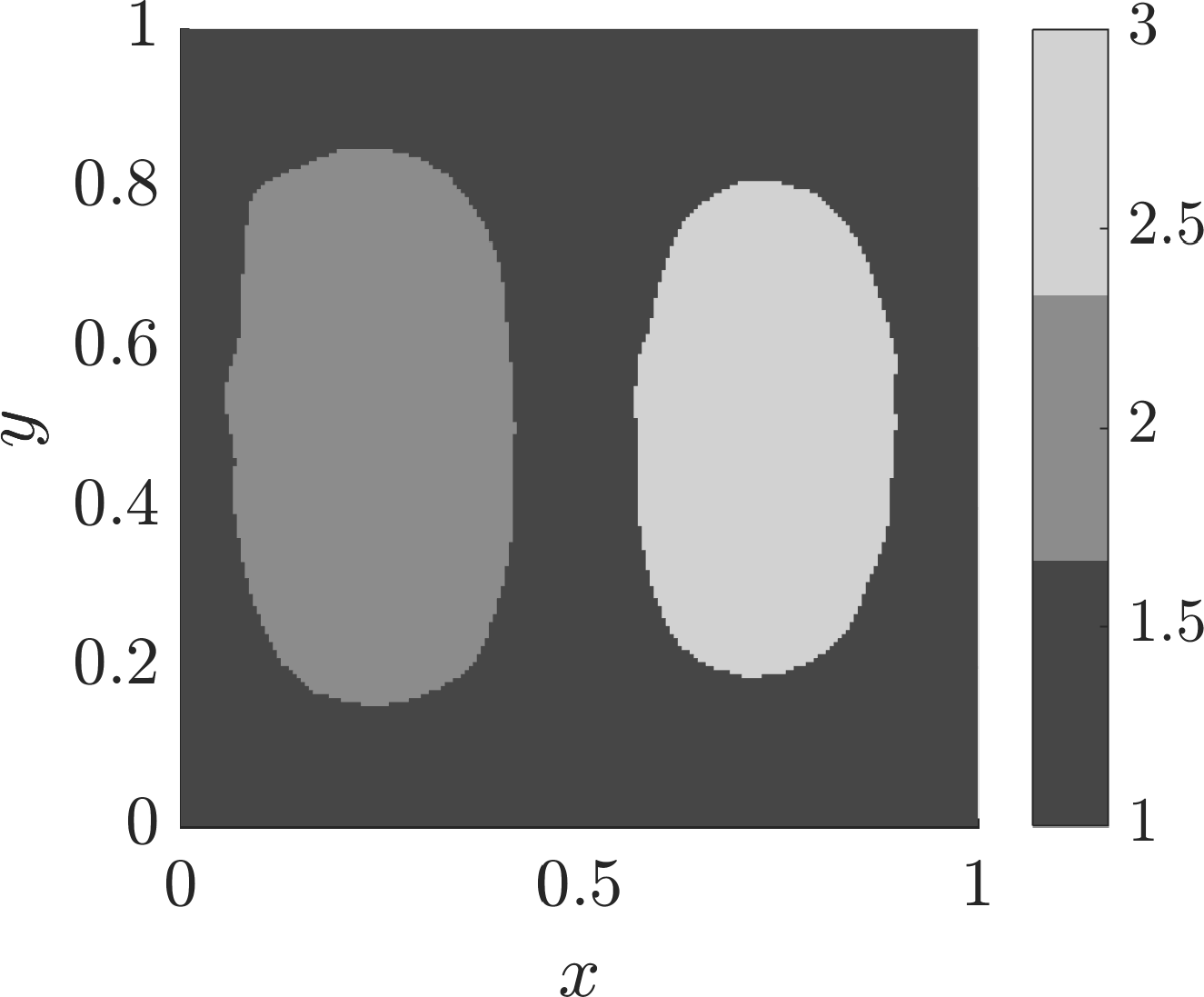}
\caption{Rotating double gyre: 2nd and 3rd eigenvector and coherent 3-partiton (from left to right) on a triangulation of a regular $25\times 25$ grid, using the transfer operator approach from Section~\ref{sec:collnonadap}.}
\label{fig:Meiss_TO}
\end{center}
\end{figure}

\paragraph{Adaptive transfer operator approach.}

Using the approach from Section~\ref{sec:colladap} on this experiment, we obtain the spectrum in Figure~\ref{fig:Meiss_TO_spectra} (right), with a gap after the third eigenvalue, and the corresponding eigenvectors and clustered coherent sets in Figure~\ref{fig:Meiss_TO_adap}. As for the (non-adaptive) transfer operator approach, the flow map (without variational equation) had to be evaluated $25\cdot 25=625$ times only, leading to the same computation times.
\begin{figure}[htbp]
\begin{center}
\includegraphics[width=0.32\textwidth]{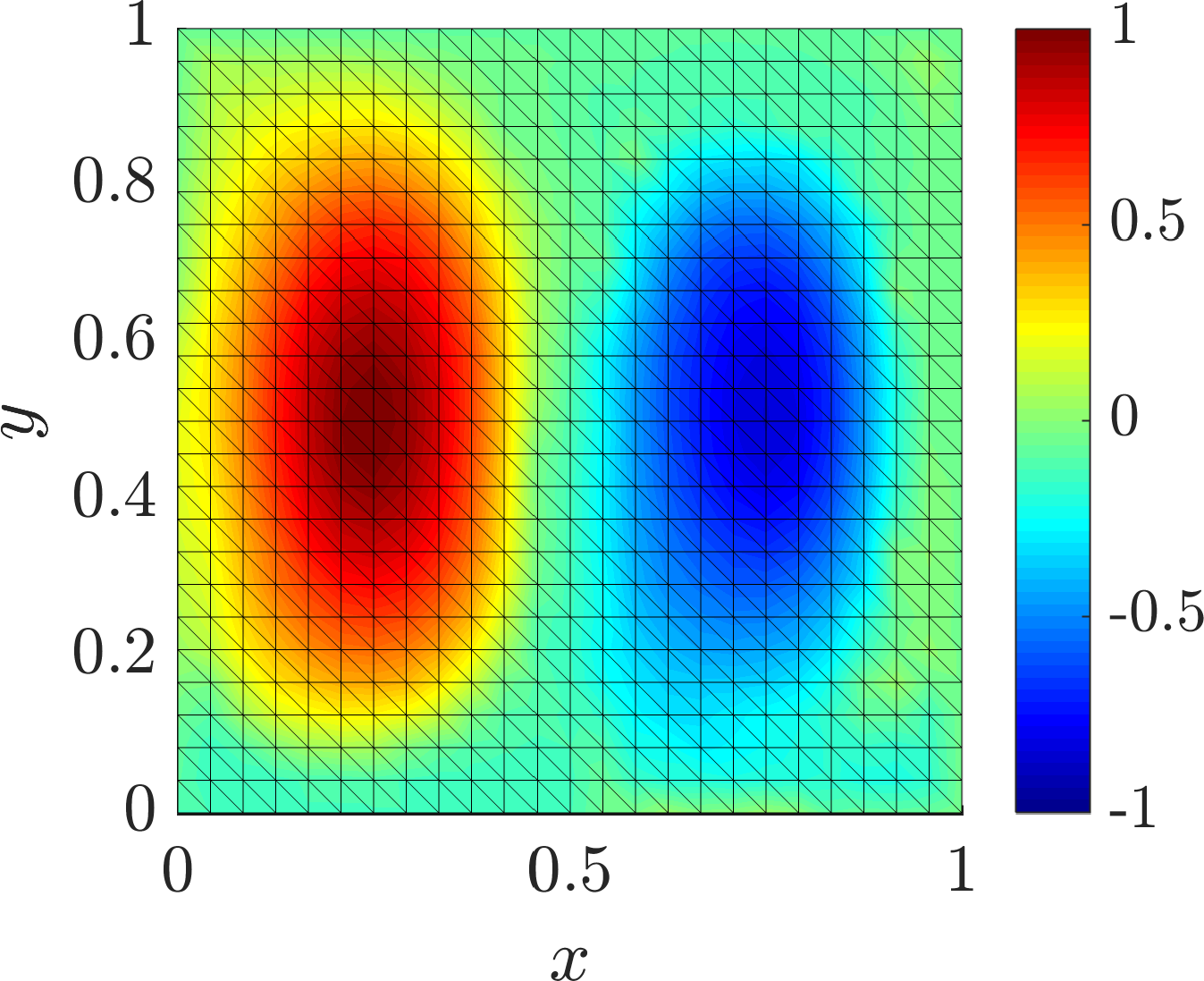}
\includegraphics[width=0.32\textwidth]{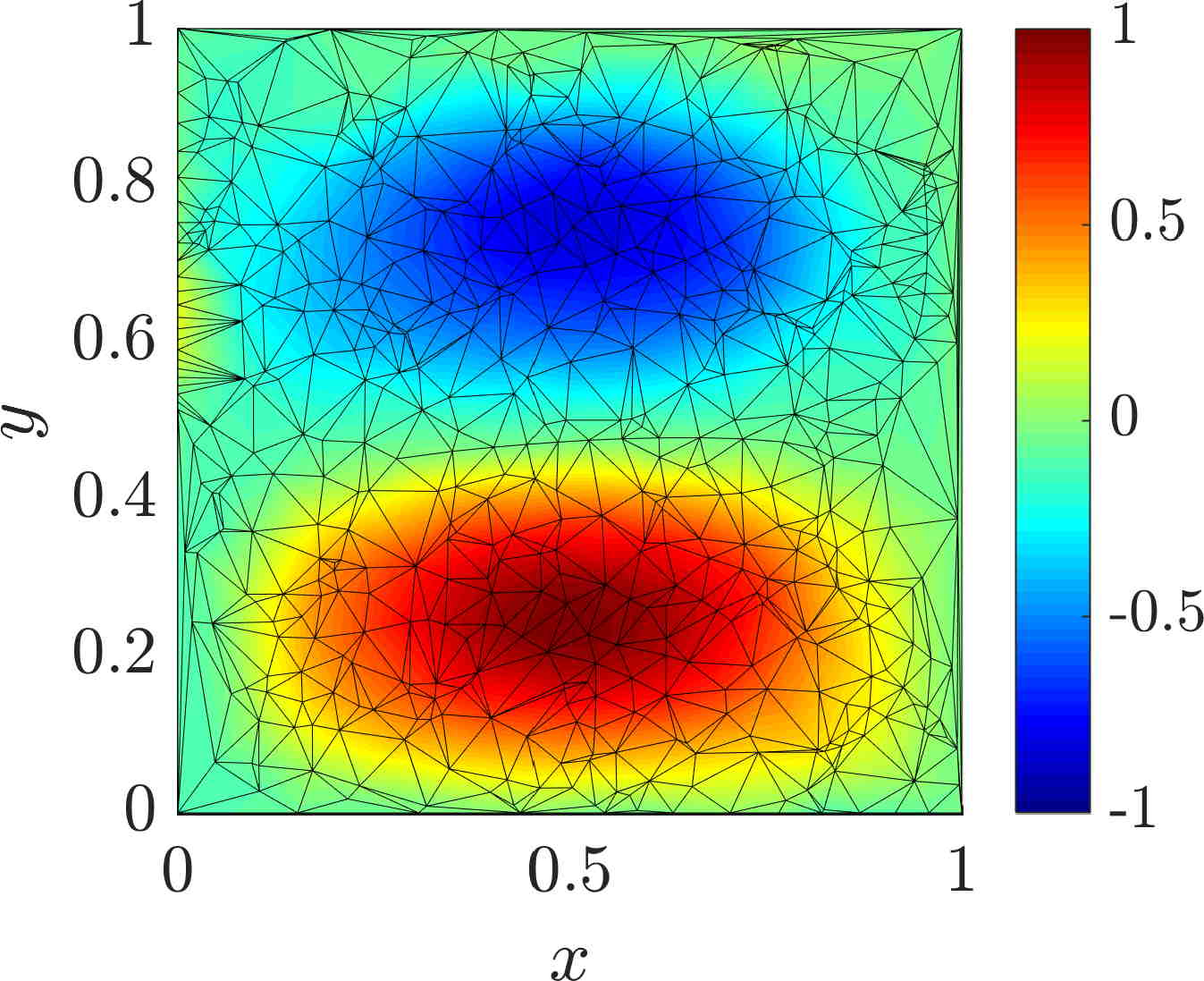}
\includegraphics[width=0.32\textwidth]{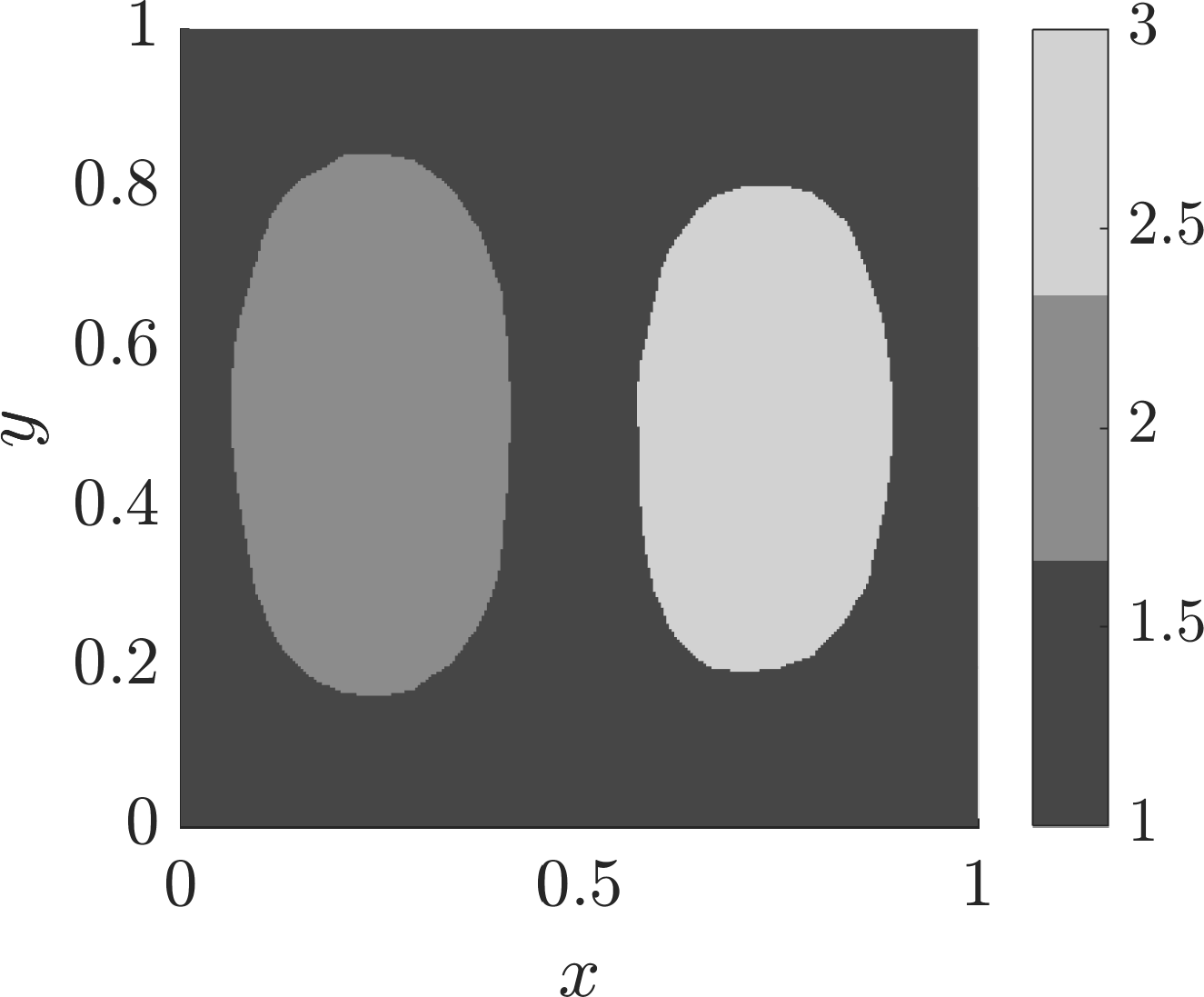}
\caption{Rotating double gyre: 2nd eigenvector on a triangulation of a regular $25\times 25$ grid (left) and its image on the Delaunay triangulation of the image points (center) as well as the resulting coherent 3-partition using the adaptive transfer operator approach from Section~\ref{sec:colladap}.}
\label{fig:Meiss_TO_adap}
\end{center}
\end{figure}
In order to show that this approach is robust with respect to non-uniformly sampled data, we repeat this experiment with a highly non-uniform triangulation: We randomly choose 20000 points in $[0,1]\times [0,0.5]$ and 200 points in $[0,1]\times [0.5,1]$ (both according to a uniform distribution) and use the Delaunay triangulation (cf.~Figure~\ref{fig:Meiss_TO_sampling}) of the union of these two point sets.  The resulting eigenvectors and coherent 3-partition are almost identical to the one obtained by a uniform point sampling.
\begin{figure}[htbp]
\begin{center}
\includegraphics[width=0.32\textwidth]{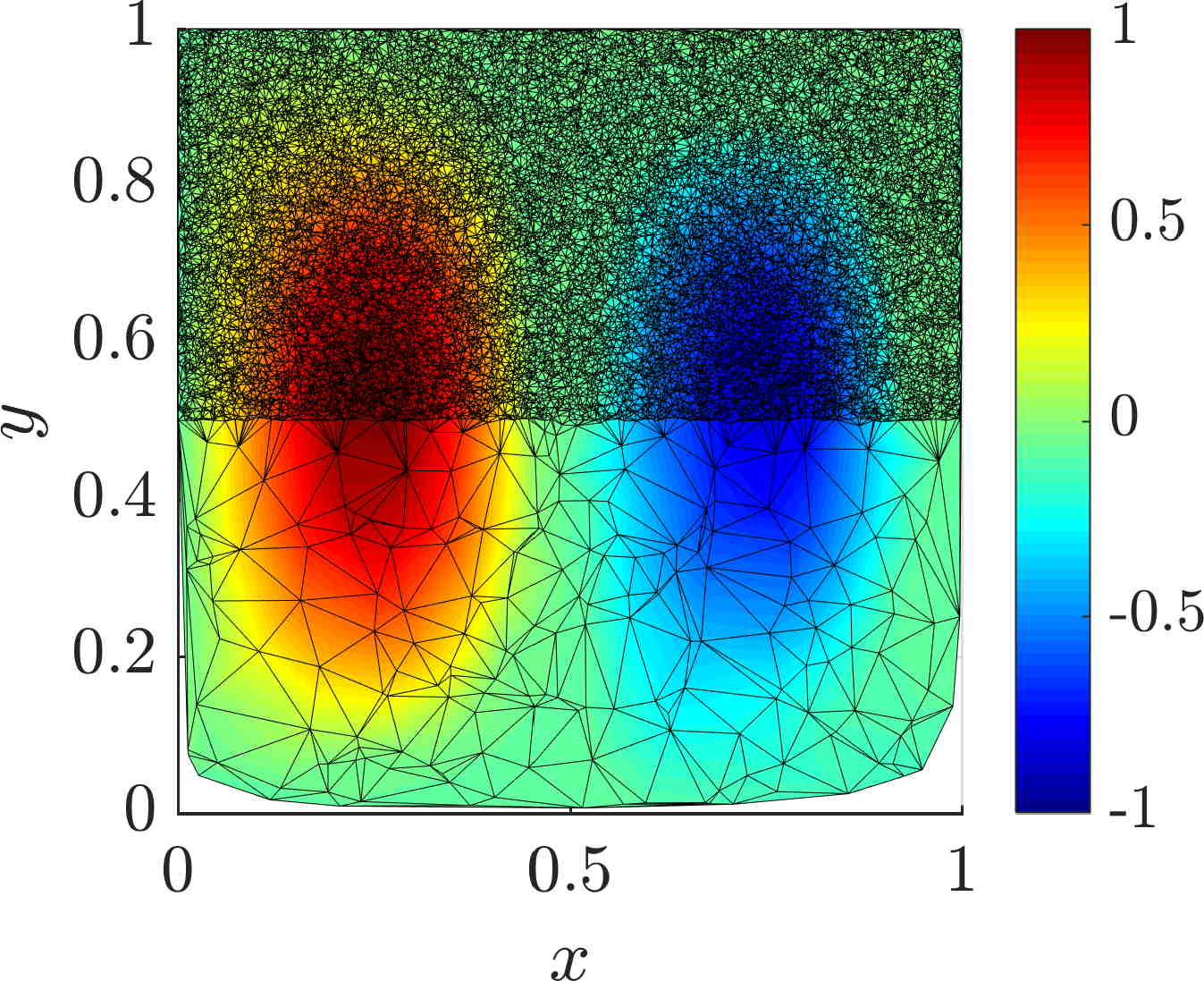}
\includegraphics[width=0.32\textwidth]{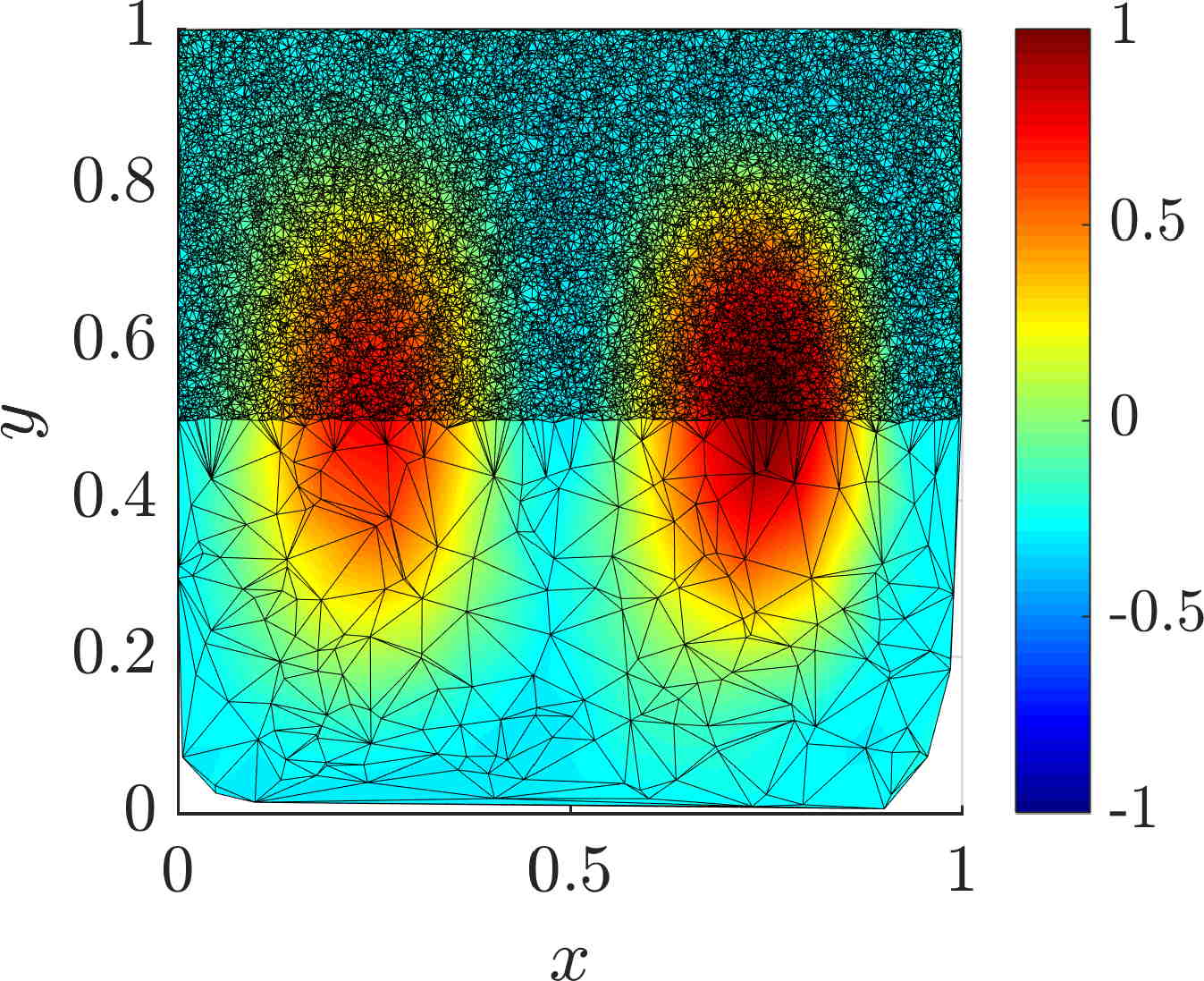}
\includegraphics[width=0.32\textwidth]{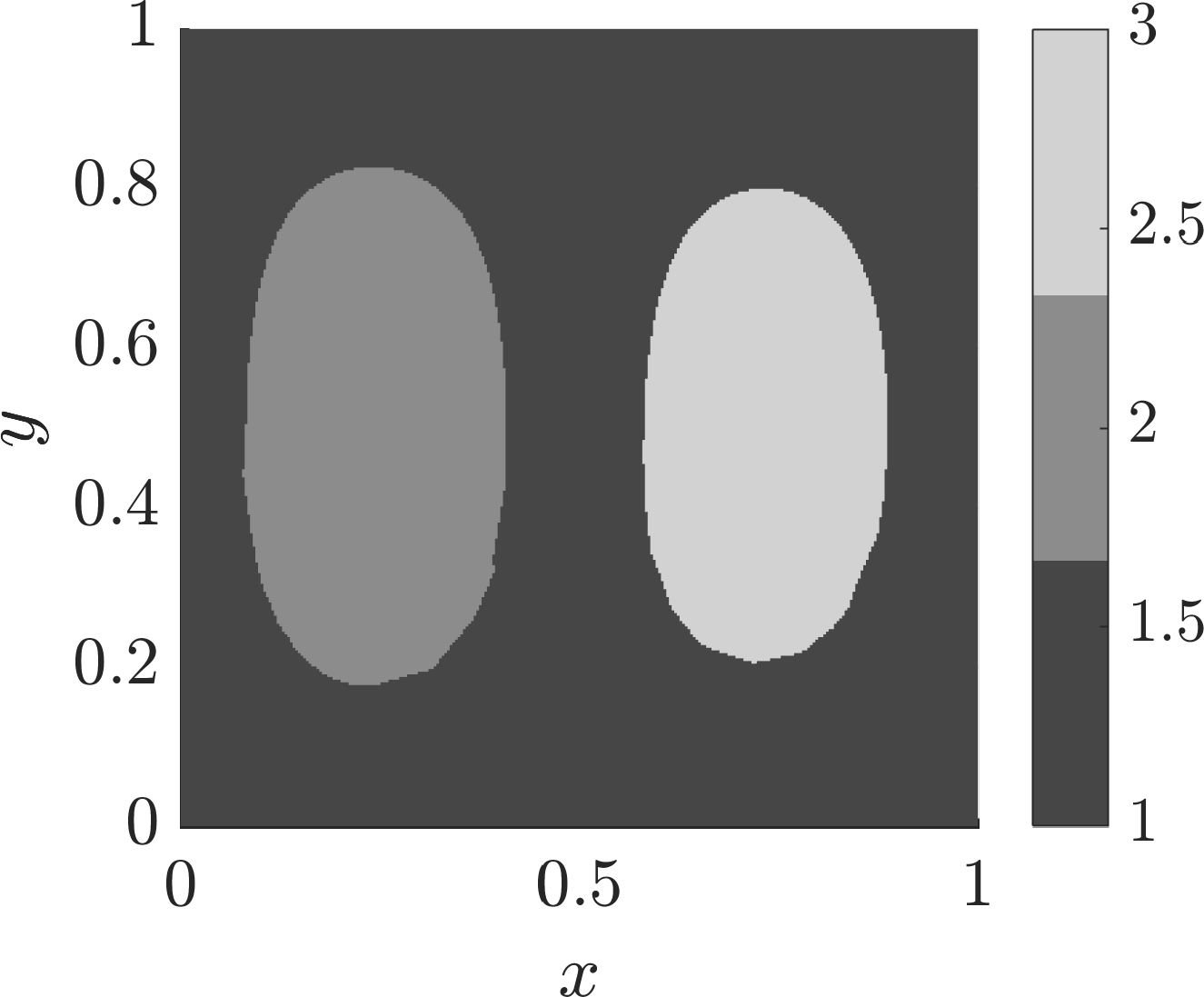}
\caption{Rotating double gyre on a highly non-uniform data set: 2nd and 3rd eigenvector as well as the resulting coherent 3-partition (from left to right), using the adaptive transfer operator approach from Section~\ref{sec:colladap}.}
\label{fig:Meiss_TO_sampling}
\end{center}
\end{figure}

\paragraph{Missing data.}

We finally analyze how well the transfer operator approach from Section~\ref{sec:TO} performs when some of the trajectory data is missing: To this end, we again use 625 randomly scattered points, but now with a denser sampling in time, namely $\mathcal{T} =\{0,0.2,0.4,0.6,0.8,1\}$. Figure~\ref{fig:Meiss_TO_missing} (top row) shows the result for the adaptive transfer operator method, which is essentially the same as in Figs.~\ref{fig:Meiss_TO}-\ref{fig:Meiss_TO_sampling}.

We now randomly delete 60\% of the data, reducing the data from 625 points to 250 data points per time step (and the total number of data points is the same as in Figs.~\ref{fig:Meiss_TO} and \ref{fig:Meiss_TO_adap}). Repeating the same experiment using the approach for missing data described in Section~\ref{sec:missing} we obtain the results shown in the bottom row of Figure~\ref{fig:Meiss_TO_missing} which still clearly show the relevant structures. The corresponding spectrum is shown in Fig.~\ref{fig:Meiss_TO_missing_spectrum}, the gap after the third eigenvalue is still visible.
\begin{figure}[htbp]
\centering
\includegraphics[width=0.32\textwidth]{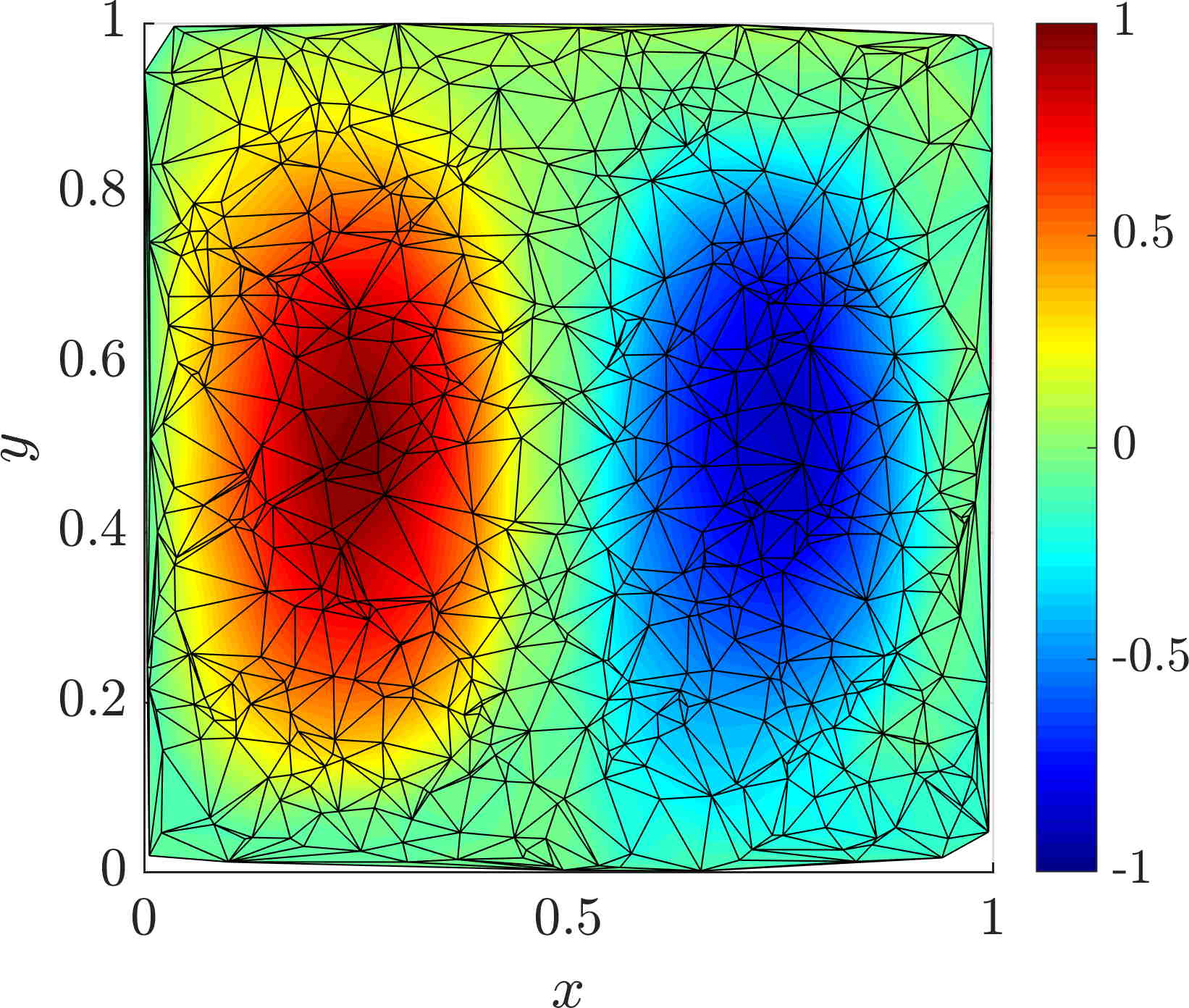}
\includegraphics[width=0.32\textwidth]{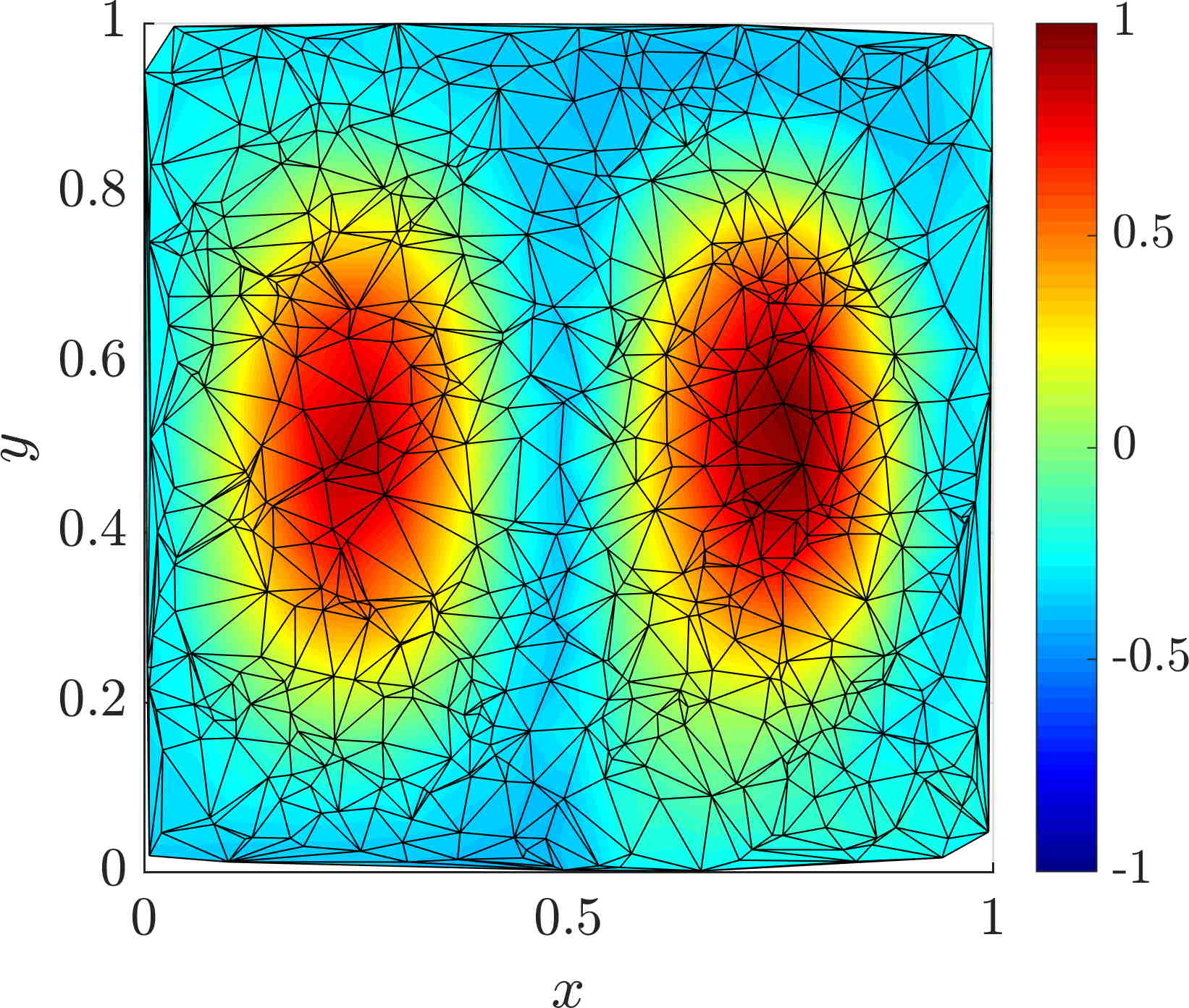}
\includegraphics[width=0.32\textwidth]{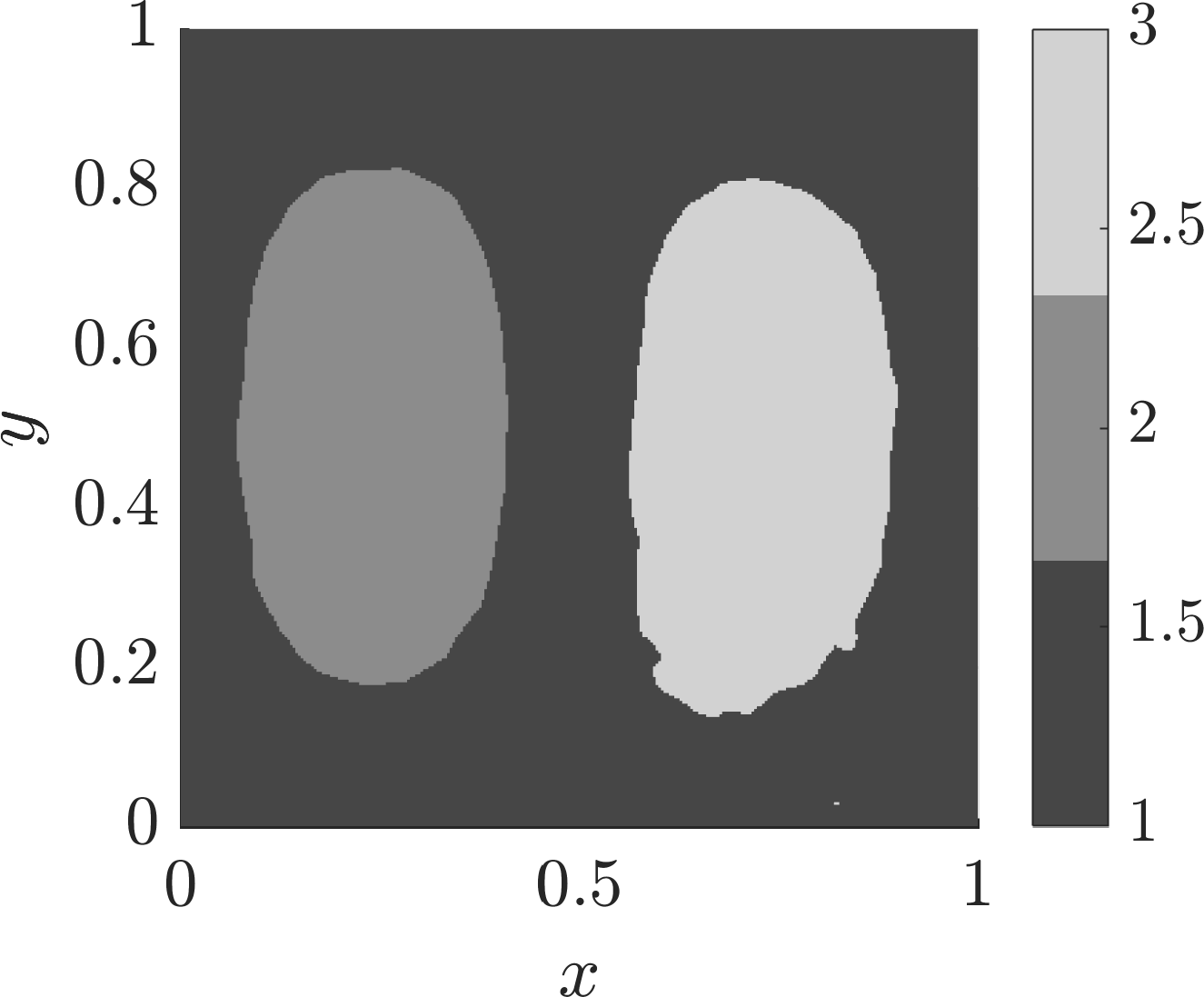}
\includegraphics[width=0.32\textwidth]{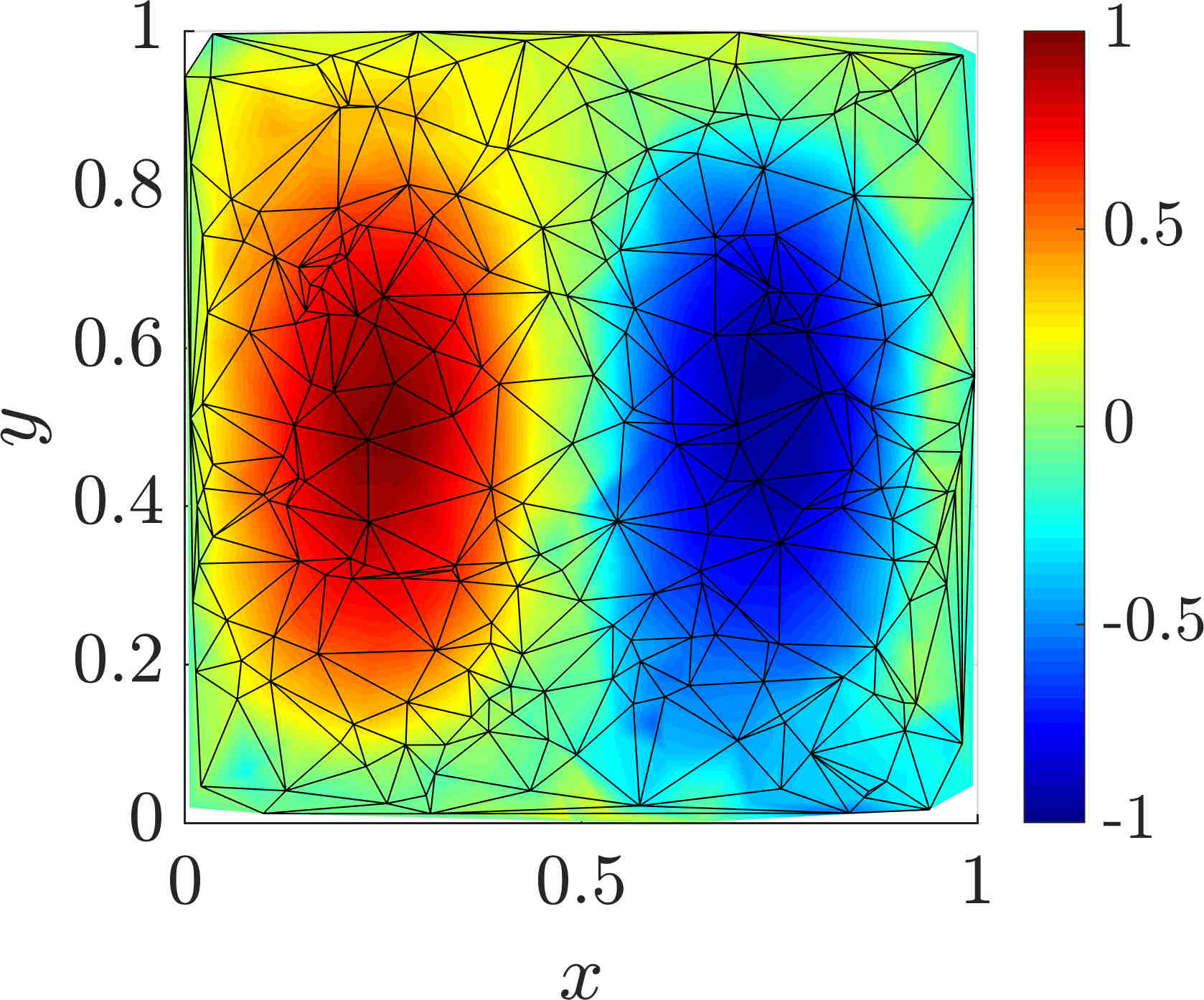}
\includegraphics[width=0.32\textwidth]{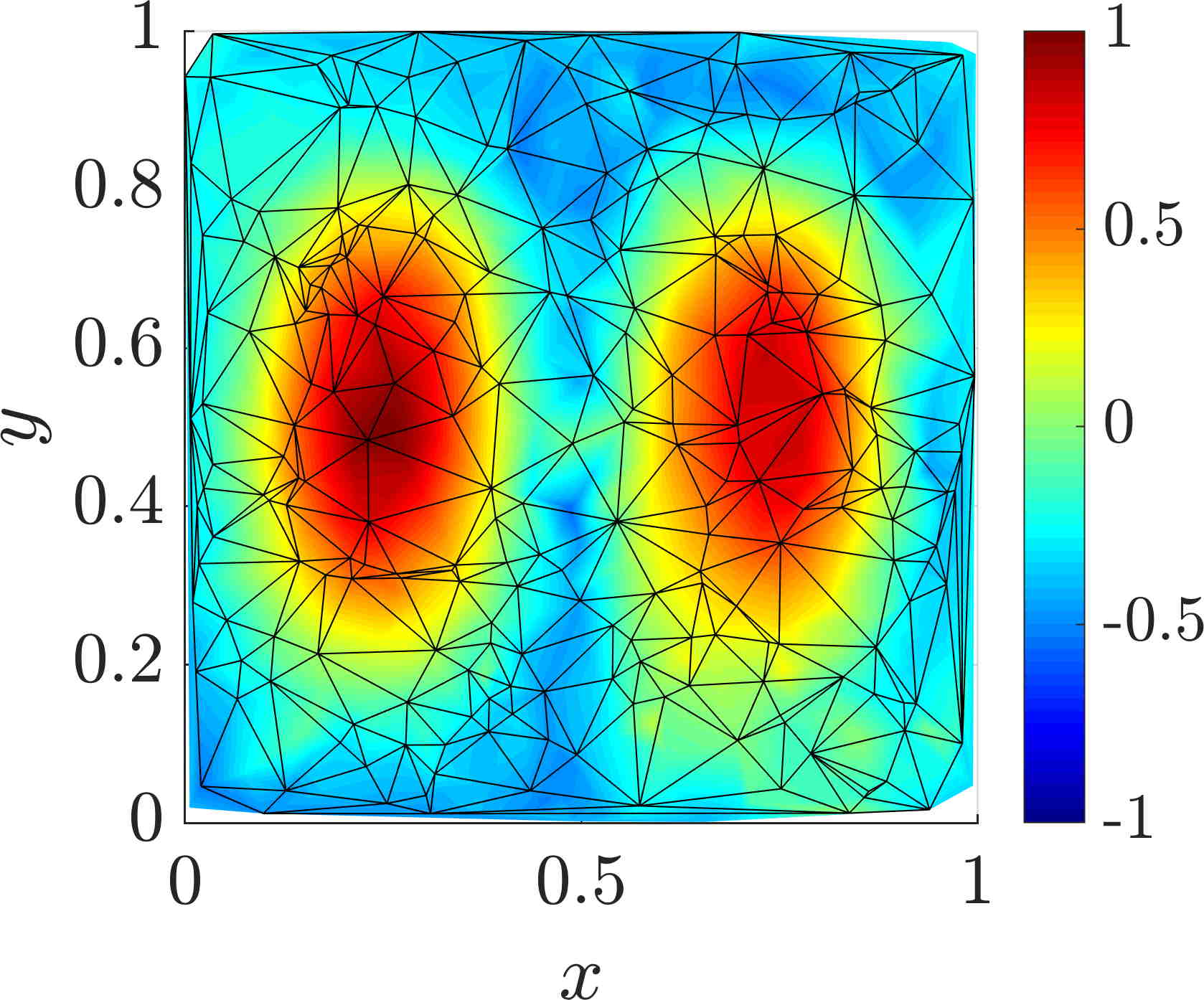}
\includegraphics[width=0.32\textwidth]{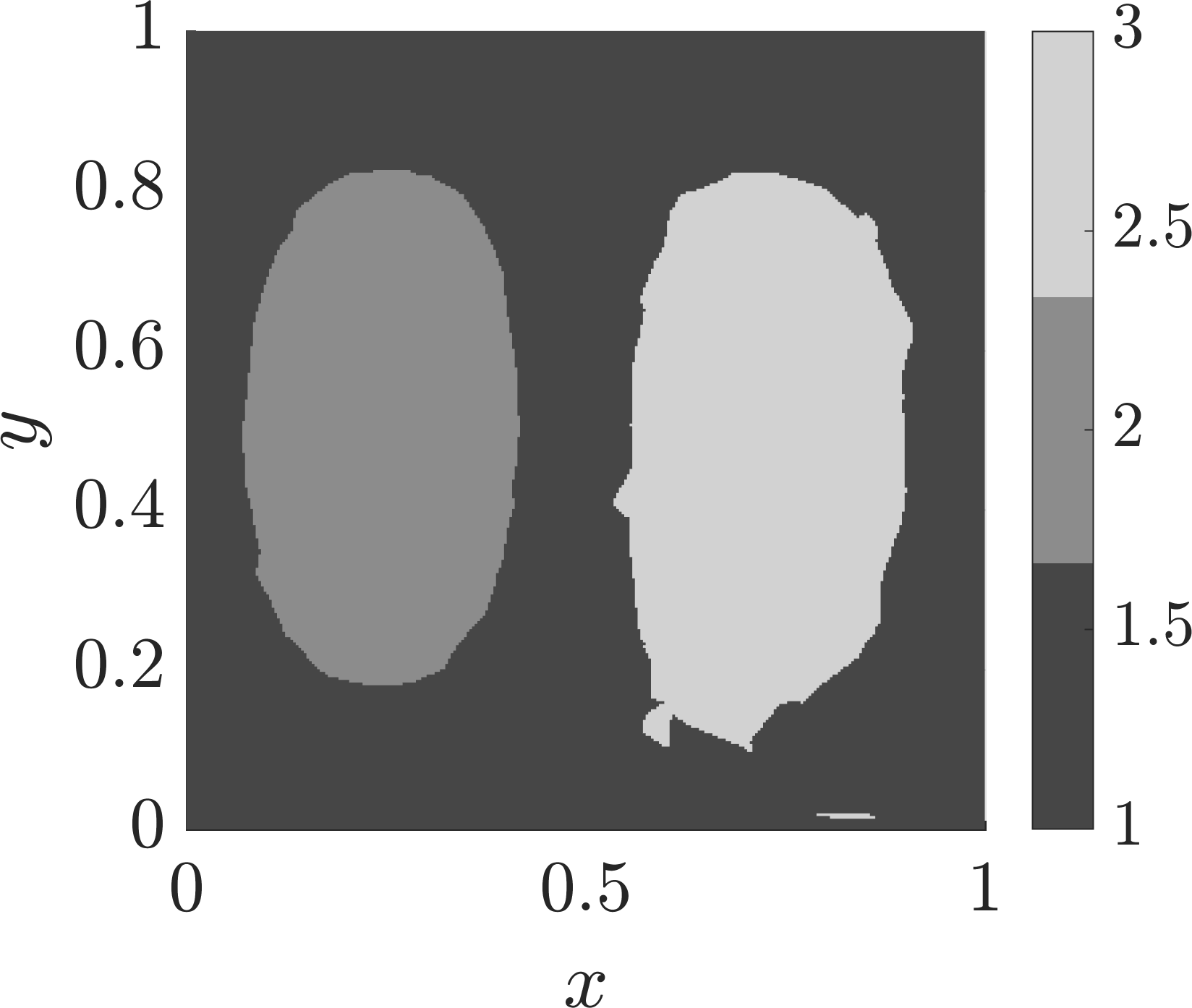}
\caption{Rotating double gyre with missing data: 2nd (left) and 3rd (center) eigenvectors as well as coherent 3-partition (right) for $\mathcal{T} =\{0,0.2,0.4,0.6,0.8,1\}$ with full data (top row) and with 60\% of the data points randomly removed (bottom row), using the adaptive transfer operator approach from Section~\ref{sec:colladap} and \ref{sec:missing}.}
\label{fig:Meiss_TO_missing}
\end{figure}

\begin{figure}[htbp]
\centering
\includegraphics[width=0.5\textwidth]{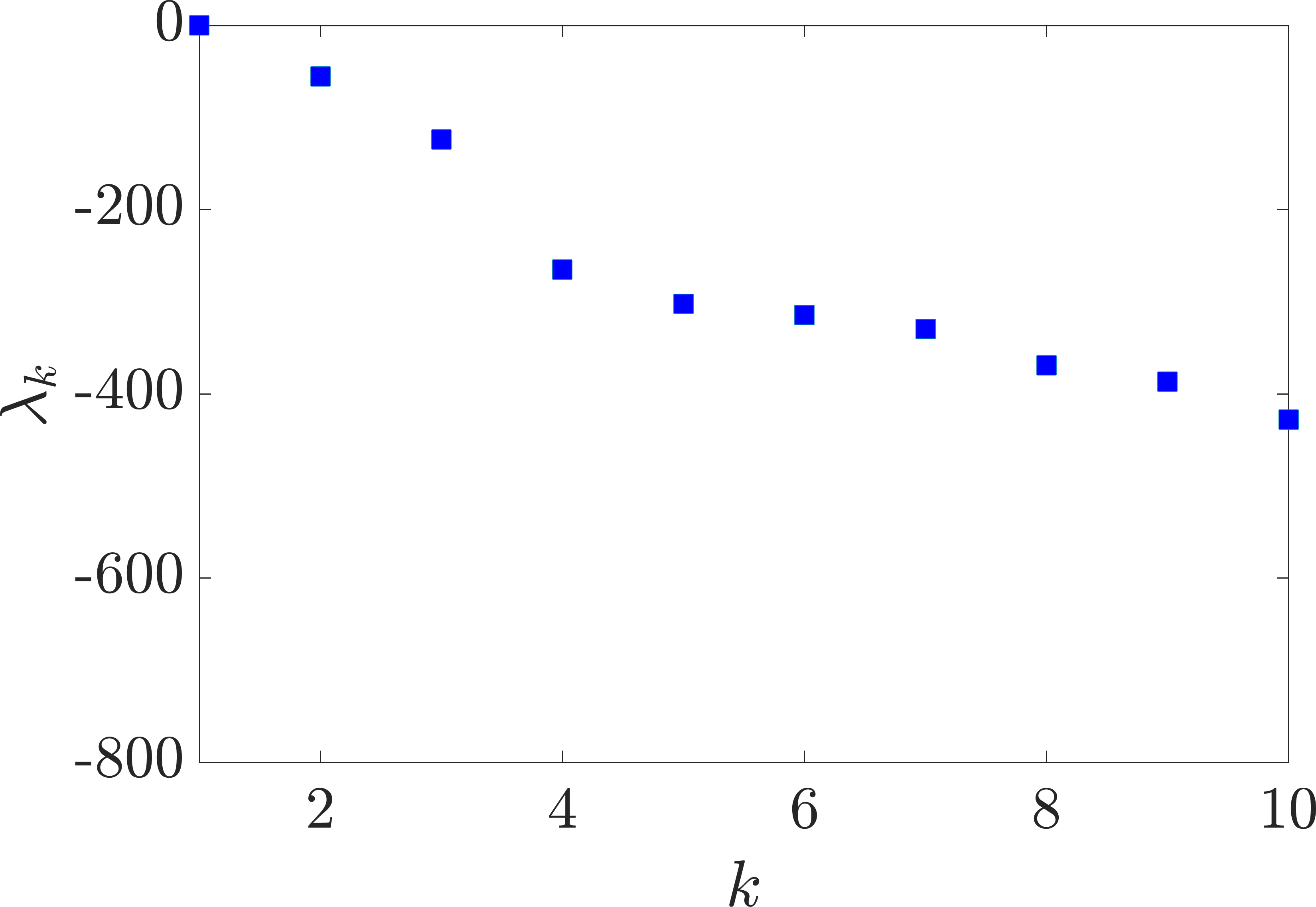}
\caption{Rotating double gyre with missing data: spectrum of the dynamic Laplacian using the approach for missing data from Section~\ref{sec:colladap}/\ref{sec:missing}.}
\label{fig:Meiss_TO_missing_spectrum}
\end{figure}

\subsection{Experiment: A non-volume preserving standard map and nonuniform initial density}

For an experiment with a non-volume preserving map and a non-uniform initial measure $\mu^0$ on $\M$, we use a test case from the literature \cite{FrKw16}.
We consider the map $\Phi = \Phi_2\circ \Phi_1$ on the flat 2-torus $\M = 2\pi S^1\times 2\pi S^1$ with
\begin{align*}
\Phi_1(x,y) & = \left(x + 0.3\cos(2x),y\right)\\
\Phi_2(x,y) & = \left(x + y, y + 8\sin(x+y)\right).
\end{align*}
The map $\Phi_1$ compresses horizontally at $x=\pi/2, 3\pi/2$ and expands horizontally at $x=0, \pi$.
The map $\Phi_2$ is the volume-preserving standard map with parameter 8.
The initial measure $\mu^0$ has density $h^0(x,y) = \frac{1}{8\pi^2}\left(\sin(y-\frac{\pi}{2})+1\right)$, shown in Figure \ref{fig:stdmap_TO_adap} (upper left);  the ratio of its maximum to minimum values is infinite as the density dips to zero at $y=0$.
This highly varying density and strongly nonlinear $\Phi$ make this example particularly challenging.

We used the adaptive transfer operator approach from Section~\ref{sec:colladap} and the formulae for the non-volume preserving case from Section \ref{sect:nvp} with $\mathcal{T}=\{0,1\}$ (one iteration of $\Phi$) on a regular triangulation with $40\times 40$ nodes.
We provide an implementation of scattered node triangulations on a periodic domain (which is not standard in MATLAB) at \verb"https://github.com/gaioguy/FEMDL".

The results are shown in Fig.~\ref{fig:stdmap_TO_adap}, which should be compared to Fig.~4 in \cite{FrKw16}.
The boundaries of the optimal coherent sets in these figures are the cyan level sets, which approximately follow lines of constant $x + y$ value and constant $x$ value, respectively (as was the case in Fig. 7 of \cite{F15}).
However, the distortion from $\Phi_1$ and the non-uniform weight $\mu_0$ causes deviations of these level sets from straight lines (see Section 6.4 \cite{FrKw16} for further discussion).
Note that the eigenvectors are captured reasonably well in Figure~\ref{fig:stdmap_TO_adap} (upper right, lower right), even though the approximation to the image density is apparently rather crude (Figure~\ref{fig:stdmap_TO_adap} (lower left)).

\begin{figure}[htbp]
\centering
\includegraphics[width=0.300\textwidth]{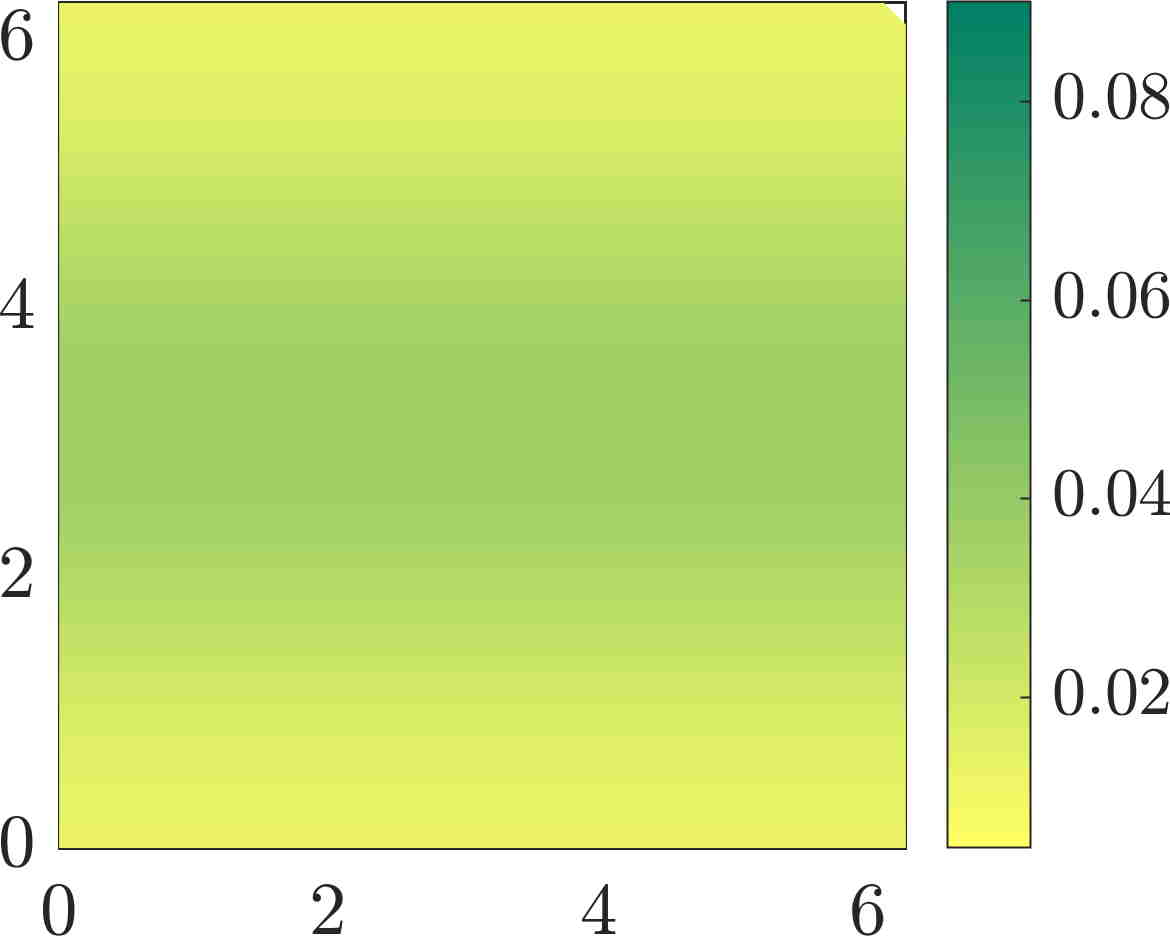}\quad
\includegraphics[width=0.230\textwidth]{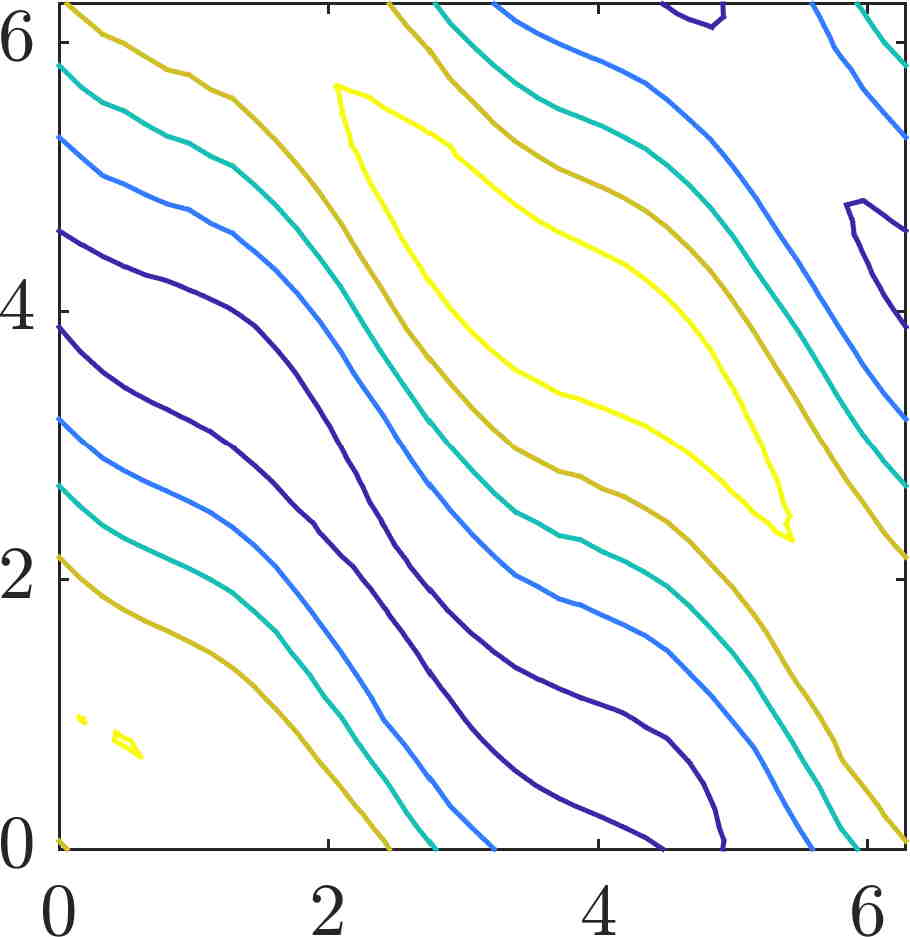}\\[2mm]
\includegraphics[width=0.300\textwidth]{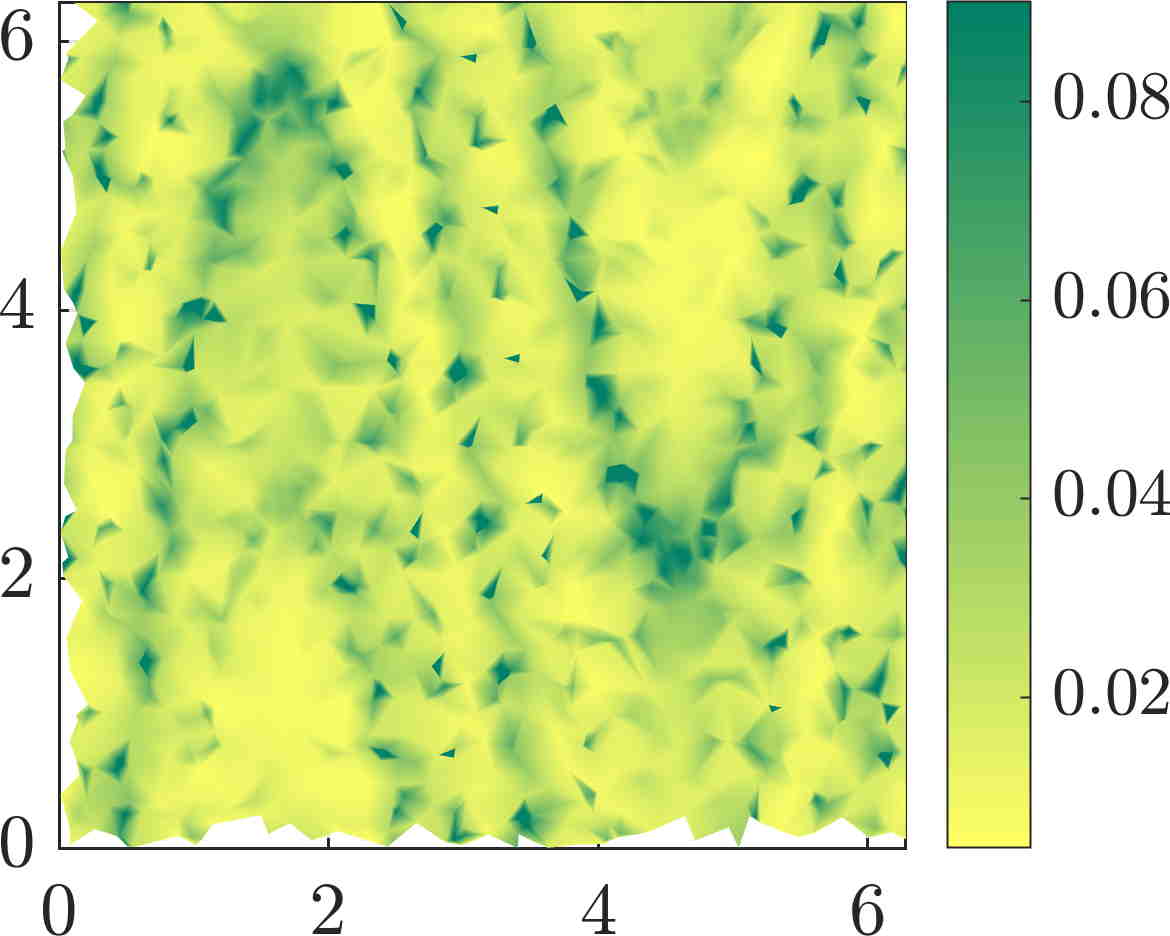}\quad
\includegraphics[width=0.230\textwidth]{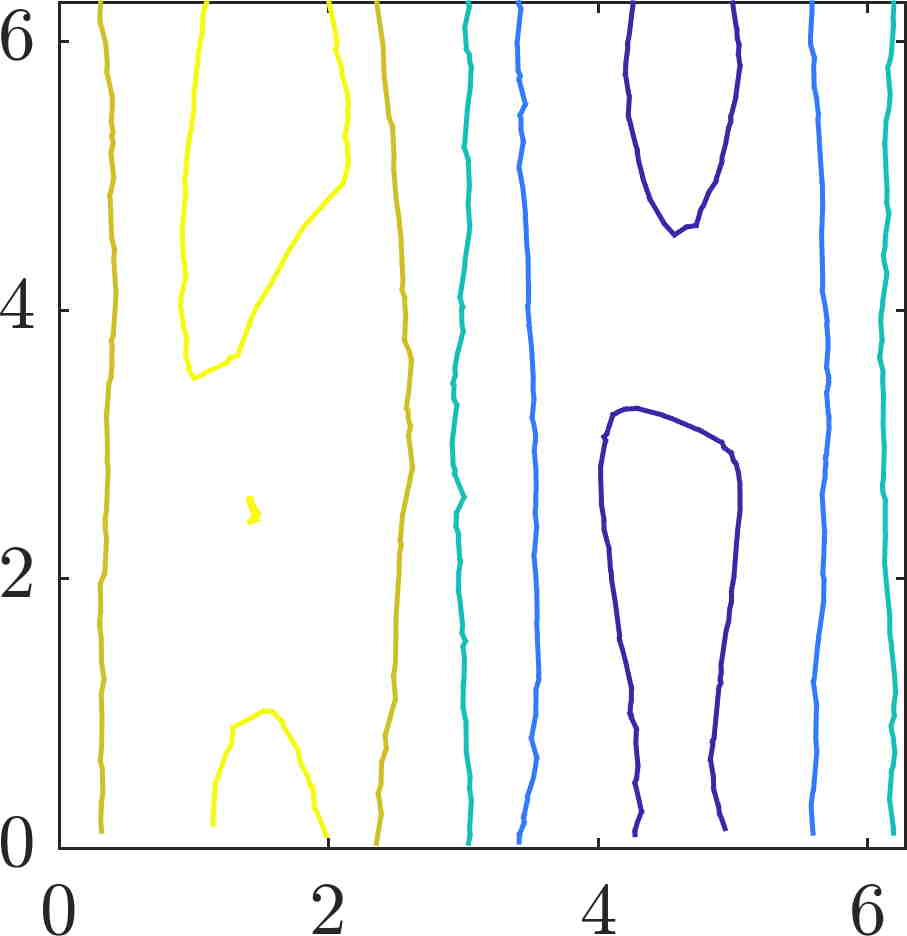}
\caption{Non-volume-preserving standard map on the flat 2-torus: initial density (upper left), 2nd eigenvector (upper right), image density (lower left) and image of 2nd eigenvector (lower right)  using the adaptive transfer operator approach from Section~\ref{sec:colladap}. The colour scheme has been chosen for easy comparison with Fig.~4 in \cite{FrKw16}.}
\label{fig:stdmap_TO_adap}
\end{figure}

\subsection{Experiment: The Bickley Jet}

\label{exp:bickleyjet}

As a second example, we consider the Bickley jet flow introduced in \cite{Ryetal07a}, modeling a meandering jet with additional rotating vortices. The defining stream function for this Hamiltonian system is given by
\[
\psi(x,y,t) = -U_0L_0 \tanh(y/L_0)+\sum_{i=1}^3 A_i U_0 L \text{sech}^2(y/L)\cos(k_i (x - c_i t))
\]
with parameter values $U_0=62.66 \text{ ms}^{-1}$, $L_0 = 1770$ km and $A_1 = 0.0075, A_2 = 0.15, A_3 = 0.3,  c_1 = 0.1446 U_0, c_2 = 0.205 U_0, c_3 = 0.461 U_0, k_1 = 2/r_e, k_2 = 4/r_e, k_3 = 6/r_e, r_e = 6371$ km as in \cite{HKTH16}. We consider the associated flow on the domain $\M = [0,20]\times [-3,3]$ with periodic boundary conditions in $x$-direction on the time interval $t\in [0,40]$ days.

\paragraph{Cauchy-Green approach.}

Employing $\mathcal{T}=\{0,40\}$ and the Delaunay triangulation on a regular grid of $100\times 30$ points as well as Gauss quadrature of degree 1, we obtain the results shown in Figure~\ref{fig:bickley_CG}. The evaluation of $C_t^{-1}$ on roughly 6000 quadrature nodes takes about 8 s, the assembly of the matrices 0.07 s and the solution of the eigenproblem 0.4 seconds.
\begin{figure}[htbp]
\begin{center}
\includegraphics[width=0.49\textwidth]{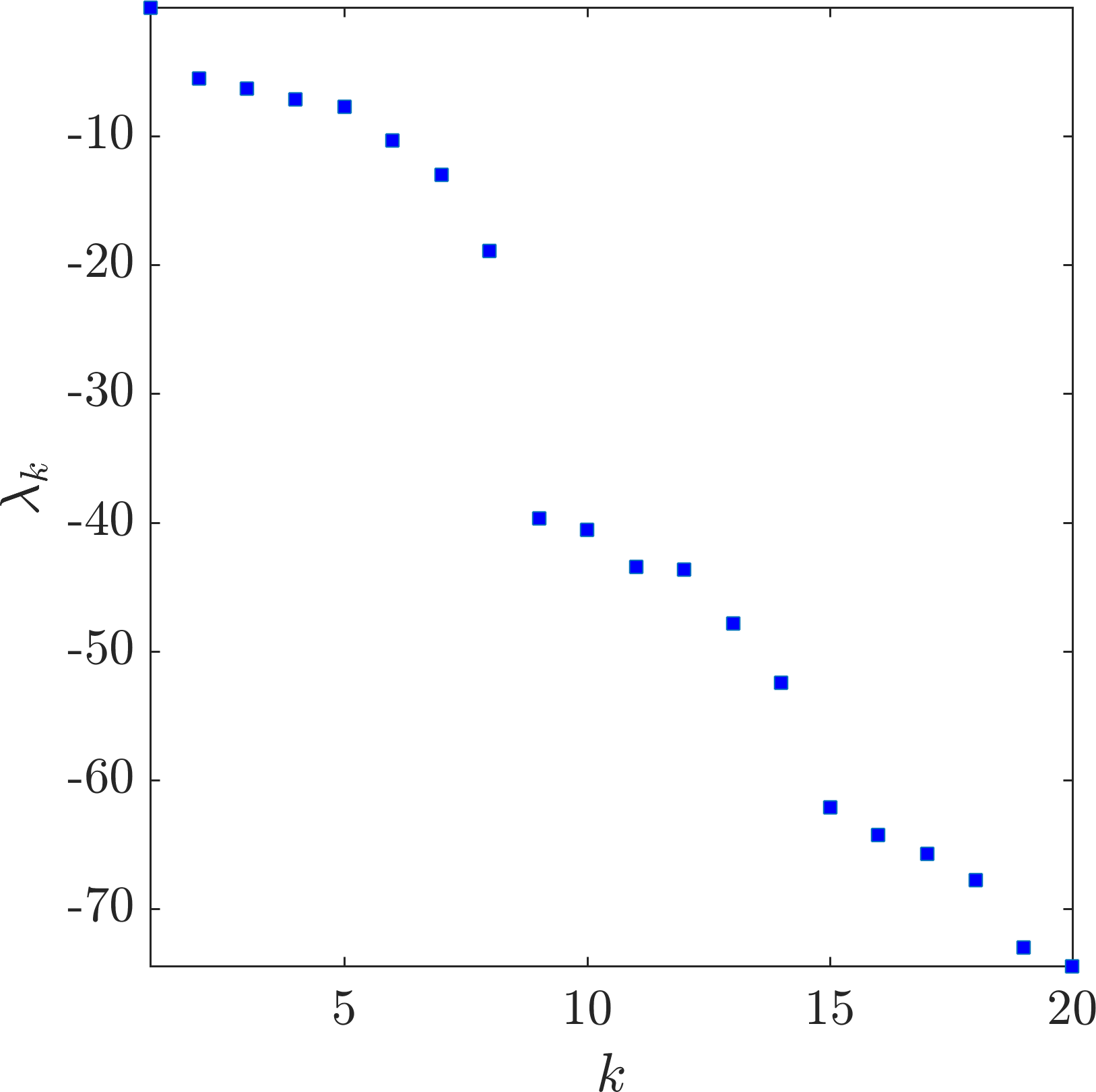}
\includegraphics[width=0.49\textwidth]{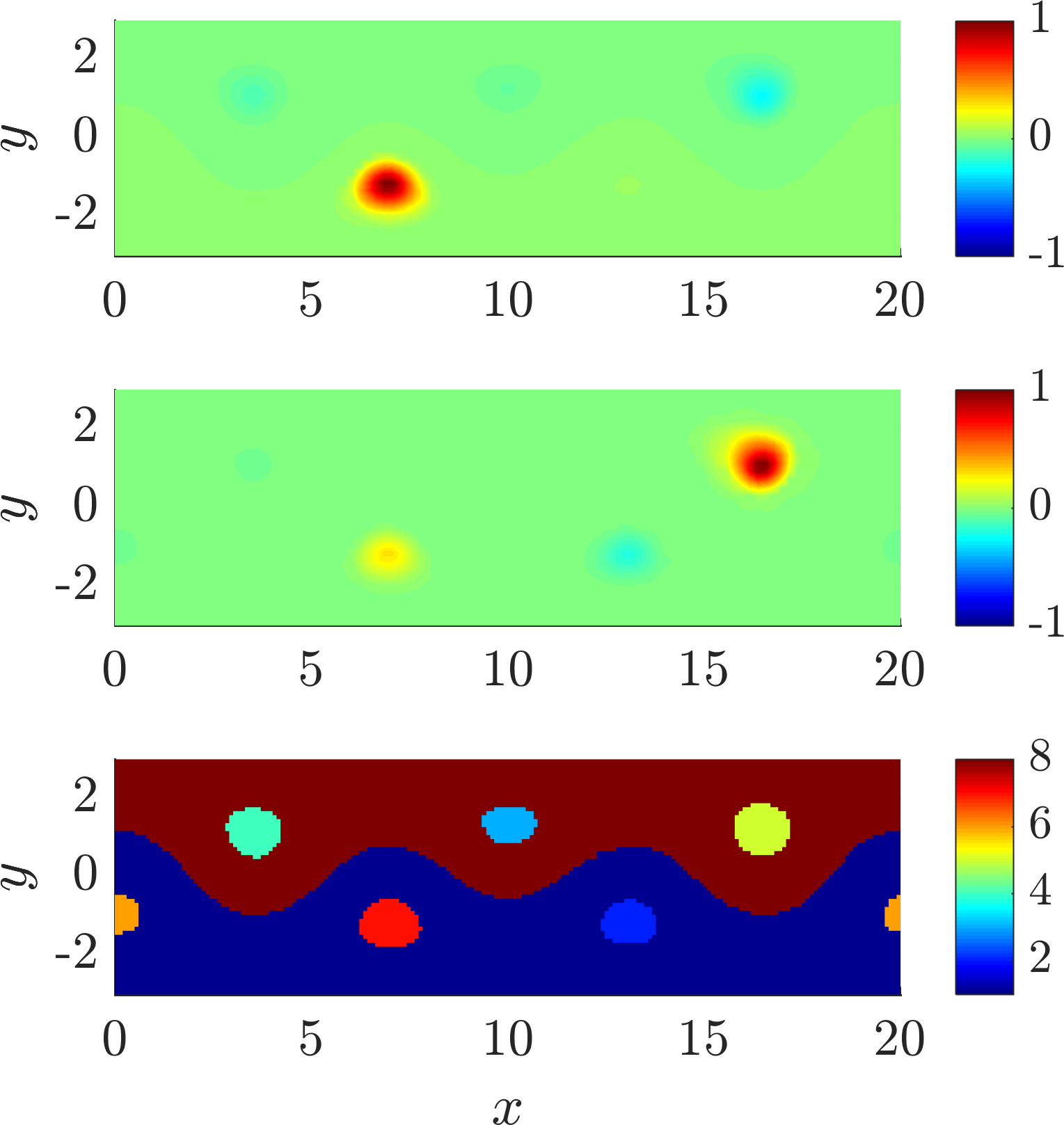}
\caption{Bickley jet: Spectrum (left), 2nd (right top) and 3rd (right center) eigenvector and a coherent $8$-partition (right bottom) on a triangulation of a regular $100\times 30$ grid of nodes using the Cauchy-Green approach from Section~\ref{sec:CG}.}
\label{fig:bickley_CG}
\end{center}
\end{figure}

\paragraph{Adaptive transfer operator approach.}

Using the same triangulation as for the Cauchy-Green approach, we obtain the results shown in Figure~\ref{fig:bickley_TO_adap}.
The spectrum in Figure 13 (left) shows a gap after the 2nd eigenvalue (corresponding to the upper/lower separation shown in Figure 13 (top right)), followed by a further six eigenvalues before the next gap (corresponding to the 6 vortices present in the flow).
The evaluation of the flow map on the $3000$ nodes of the triangulation takes 1 s, the assembly of the matrices 0.1 s and the solution of the eigenproblem 0.5 seconds.
\begin{figure}[htbp]
\begin{center}
\includegraphics[width=0.49\textwidth]{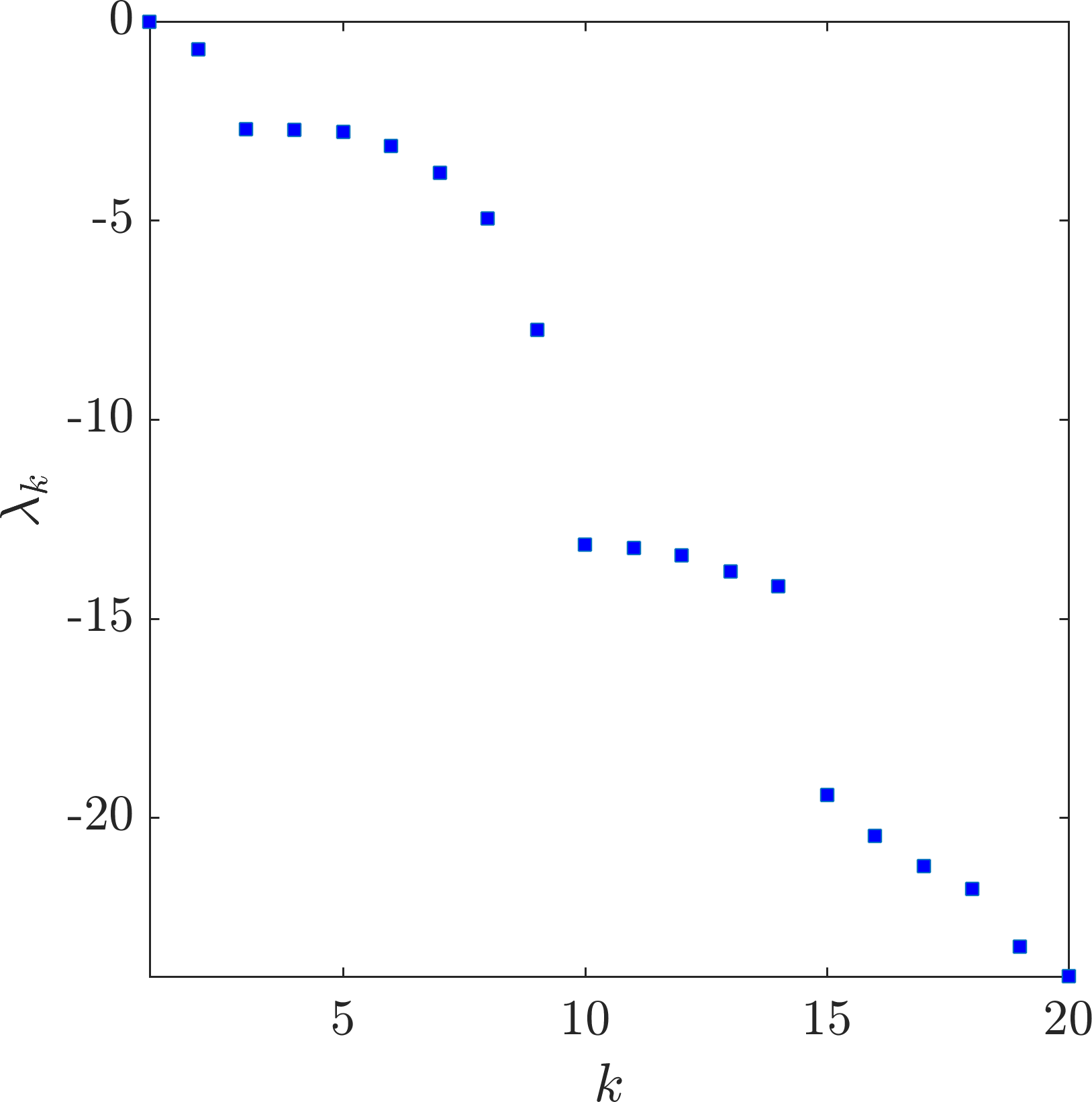}
\includegraphics[width=0.49\textwidth]{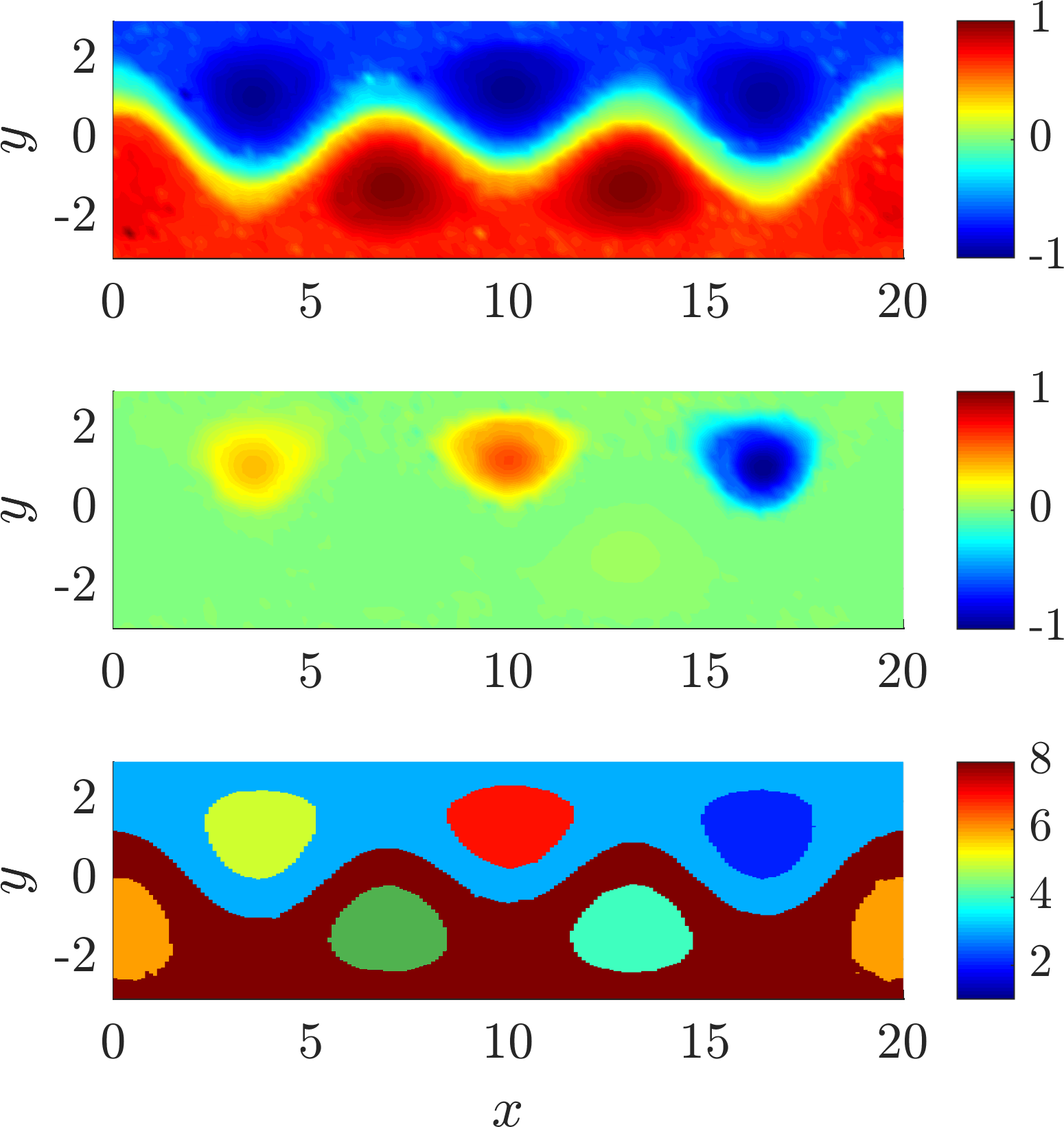}
\caption{Bickley jet: Spectrum (left), 2nd (right top) and 3rd (right center) eigenvector and a coherent $8$-partition (right bottom) on a triangulation of a regular $100\times 30$ grid of nodes using the adaptive transfer operator approach from Section~\ref{sec:colladap}.}
\label{fig:bickley_TO_adap}
\end{center}
\end{figure}

\paragraph{Missing data.}

In order to analyze the performance when some of the data is missing we use the same  $100\times 30$ grid, but now with ten intermediate time samplings of the trajectories, i.e.~$\mathcal{T}=\{0,4,8,\ldots,40\}$.  The result of the adaptive transfer operator approach is essentially the same as in Fig.~\ref{fig:bickley_TO_adap}.  We then randomly delete 80\% of the data, reducing it from 3000 points to 600 points per time step, which in total yields the same number of data points as with $\mathcal{T}=\{0,40\}$ and no missing data.  The approach from Section~\ref{sec:missing} leads to the results shown in Fig.~\ref{fig:Bickley_missing}, still clearly showing similar structures to Fig.~\ref{fig:bickley_TO_adap}. Furthermore, the spectral gap structure shown in Fig.~\ref{fig:Bickley_missing} (left) remains broadly consistent with the results using 3000 points shown in Fig.~\ref{fig:bickley_TO_adap} (left), with an additional separation between eigenvalues 6 and 7, possibly due to a loss of coherence information from the data destruction.
\begin{figure}[htbp]
\centering
\includegraphics[width=0.49\textwidth]{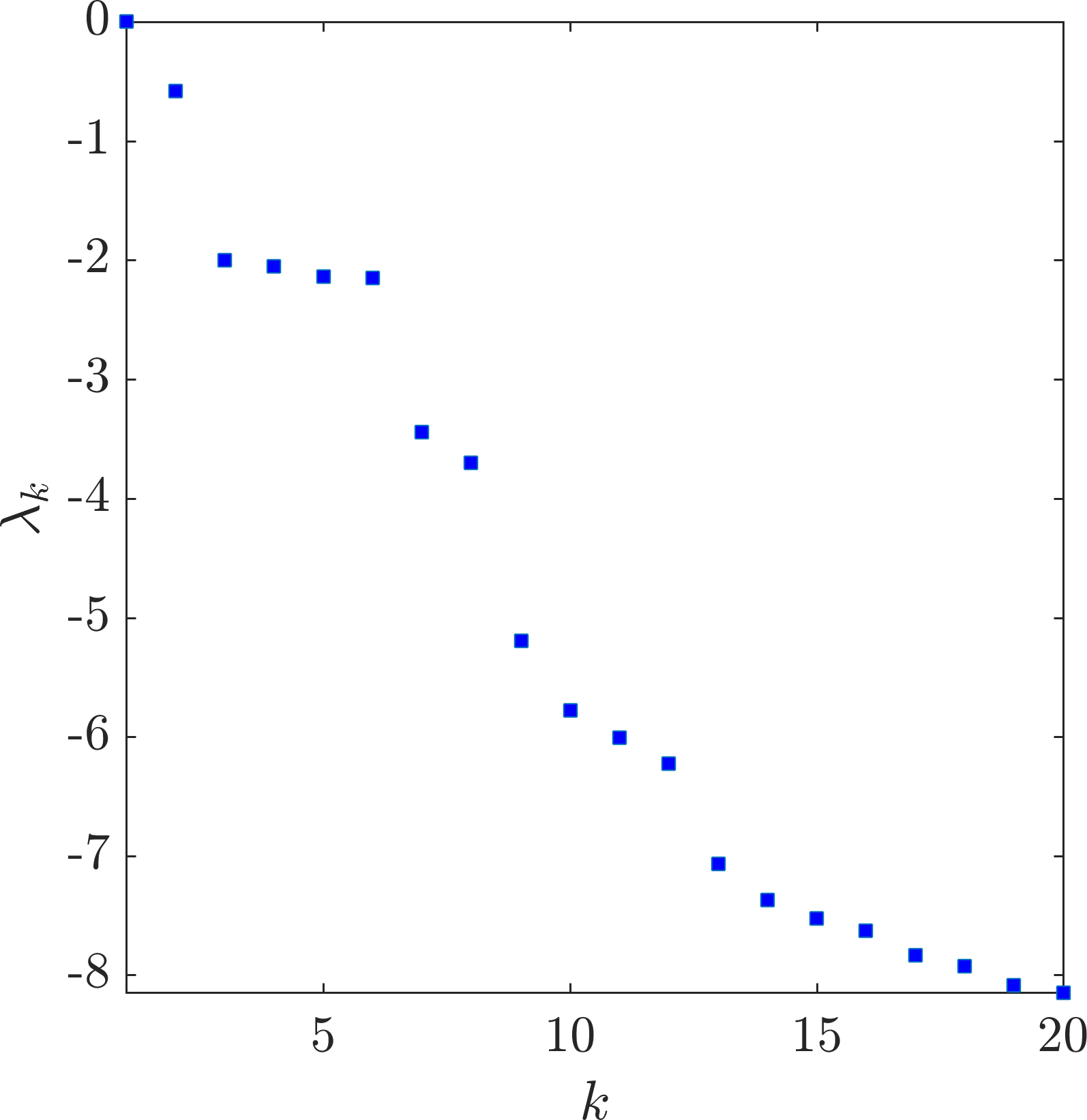}
\includegraphics[width=0.49\textwidth]{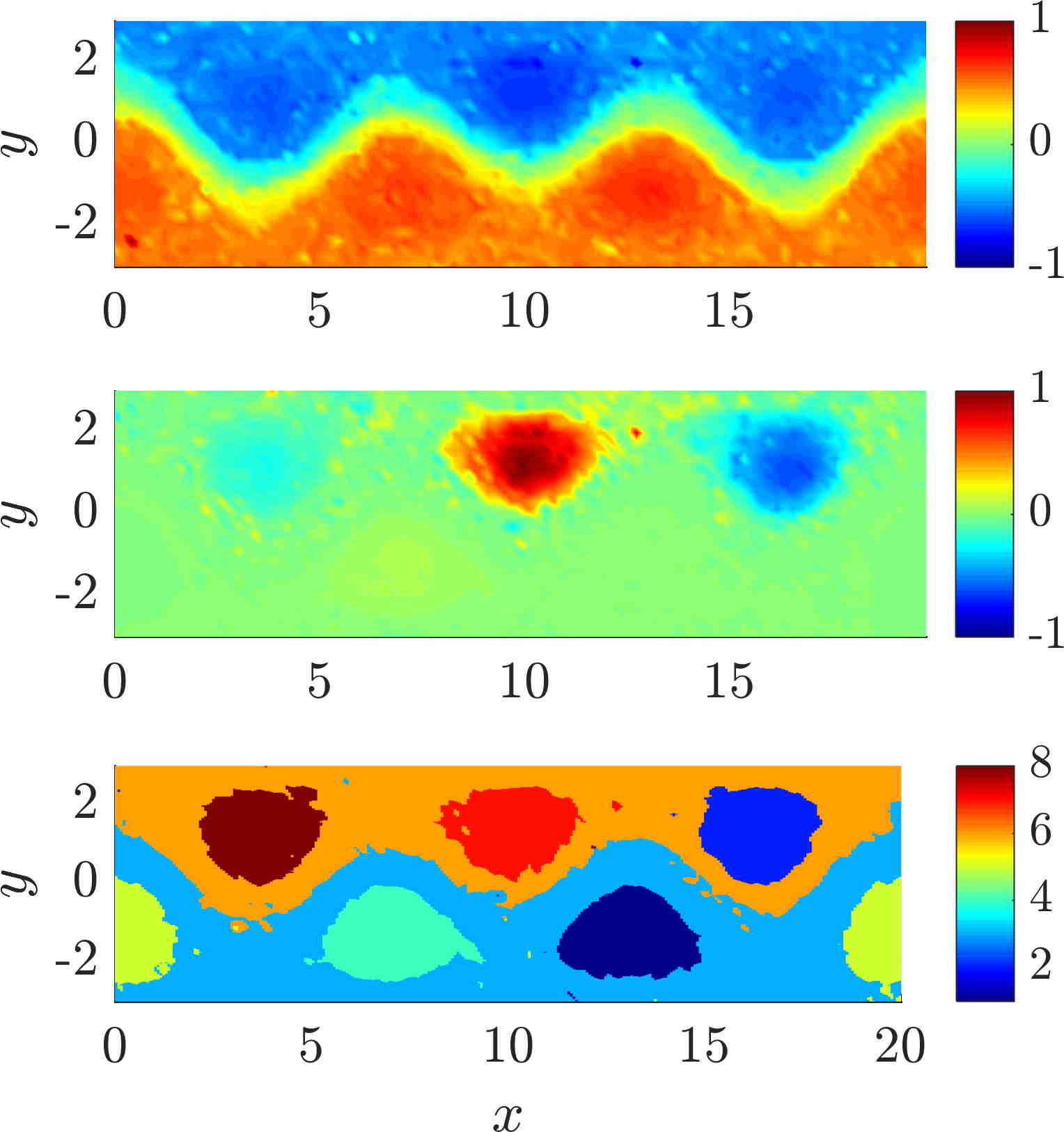}
\caption{Bickley jet with missing data: Spectrum (left), 2nd (right top) and 3rd (right center) eigenvector as well as coherent 8-partition (right bottom) with 80\% of the data  randomly removed (cf.~Fig.~\ref{fig:bickley_TO_adap}).}
\label{fig:Bickley_missing}
\end{figure}

\subsection{Experiment: Ocean flow from satellite data}
\label{exp:ocean_CG}

We now consider an unsteady velocity field derived from AVISO satellite altimetry measurements in the domain of the Agulhas leakage in the South Atlantic Ocean. Under the assumption of geostrophicity, the sea surface height $h$ yields a stream function for the velocity of water at the surface and the corresponding equations of motion for particle trajectories are
\begin{align}
	\dot\varphi &= -A(\theta)\partial_\theta h(\varphi,\theta,t)\\
	\dot \theta &= A(\theta)\partial_\varphi h(\varphi,\theta,t),
\end{align}
where $\varphi$ is the longitude and $\theta$ the latitude of a particle, $A(\theta) = g/{R^22\Omega\sin\theta\cos\theta}$, $g$ is the gravitational constant, $R$ is the mean radius of the earth and $\Omega$ the earth's mean angular velocity.  We choose the same spatial and temporal domain as in \cite{HaHa16a}, namely $[-4,6]\times [-34,-28]$ and a period of 90 days, from $t_0 = $ Nov 11, 2006 on.  Figure~\ref{fig:ocean_FTLE} shows the forward finite time Lyapunov exponent field of the data.
\begin{figure}[htbp]
\begin{center}
\includegraphics[width=0.6\textwidth]{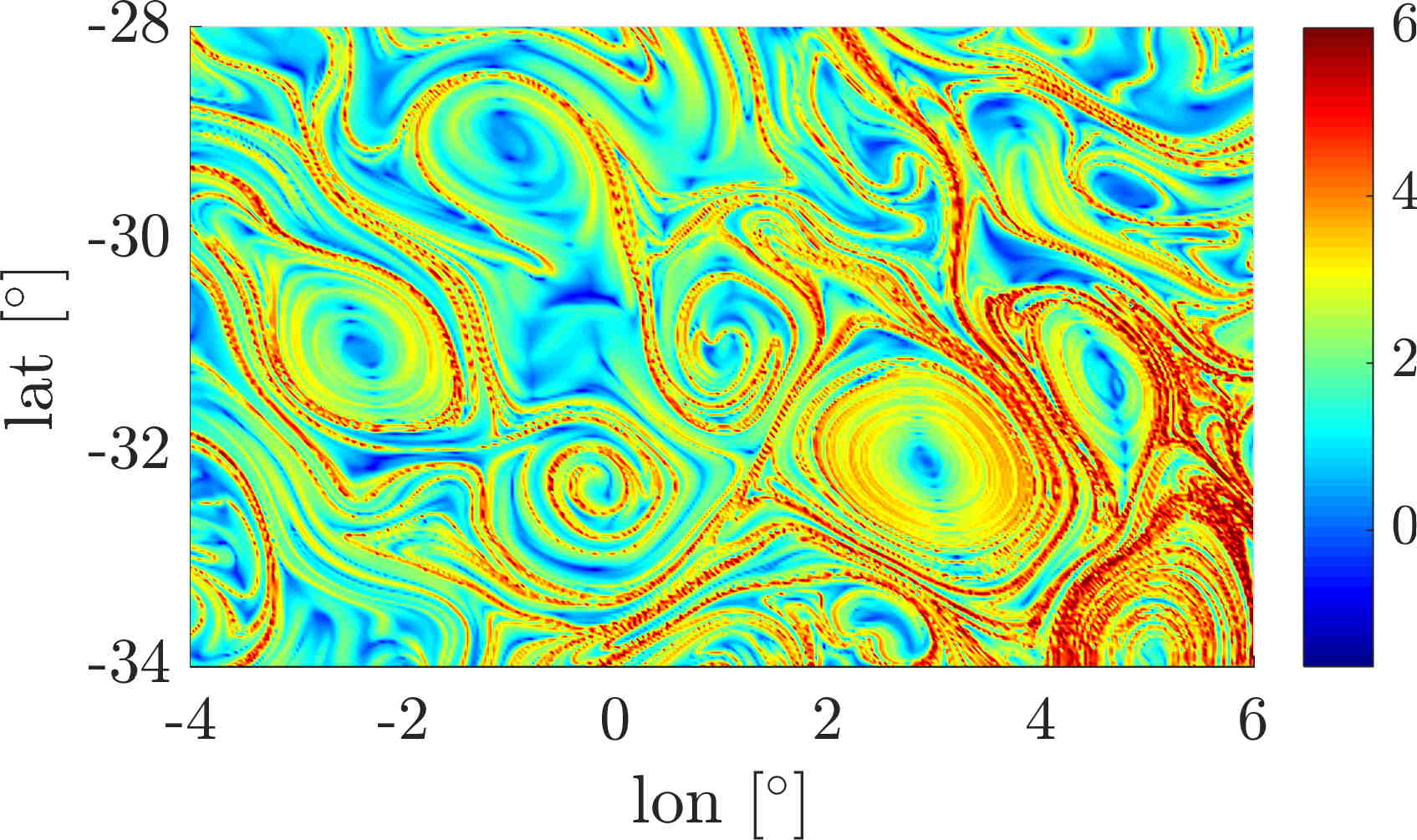}
\caption{Ocean flow: FTLE field ($\log_{10}$ color coding)}
\label{fig:ocean_FTLE}
\end{center}
\end{figure}

After a flow duration of 90 days the initial rectangular domain shown in Figure \ref{fig:ocean_FTLE} becomes strongly distorted and filamented.
In this situation, we are not interested in coherent sets that intersect this heavily filamented boundary and so we restrict to coherent sets in the interior of the domain.
This can be achieved by replacing the (natural) homogeneous Neumann boundary conditions with homogeneous Dirichlet boundary conditions.
In terms of numerical eigenvector computations, to apply homogeneous Dirichlet boundary conditions, one simply sets all rows and columns of $\hat{D}^t$ and $M$ in (\ref{Deqn}) corresponding to nodes on the boundary of the domain to rows and columns of zeros.

\paragraph{Cauchy-Green approach.}

As in \cite{HaHa16a}, we use a uniform grid of $250\times 150$ points, the associated Delaunay triangulation, yielding  roughly 73000 triangles and Gauss quadrature of degree 1 (i.e.~one quadrature point per element).  The evaluation of $C_t^{-1}$ takes ca.~46 s, the assembly of the matrices 0.3 s and the solution of the eigenproblem 3 seconds.

Figure~\ref{fig:ocean_spectra} (left) shows the largest 10 eigenvalues of the dynamic Laplacian with $\mathcal{T}=\{t_0,t_0+90\}$, the results do not change significantly when using more intermediate time steps.  Note that there appears to be a gap after the third eigenvalue. Figure~\ref{fig:ocean_CG_evs} (left column) shows the associated  first three eigenvectors.  The corresponding coherent sets (as identified by k-means clustering the first three eigenvectors, cf.~Fig.~\ref{fig:ocean_cluster4}, left) nicely agree with those identified in \cite{HaHa16a}, Fig.~9 (the rightmost vortex from Fig.~9 in that work appears in a lower eigenfunction which we do not show here).

We deliberately chose the same resolution as in \cite{HaHa16a} here, the results for the leading three eigenvectors do not change significantly, however, when decreasing the grid resolution down to $100\times 60$ points only.

\begin{figure}[htbp]
\begin{center}
\includegraphics[width=0.4\textwidth]{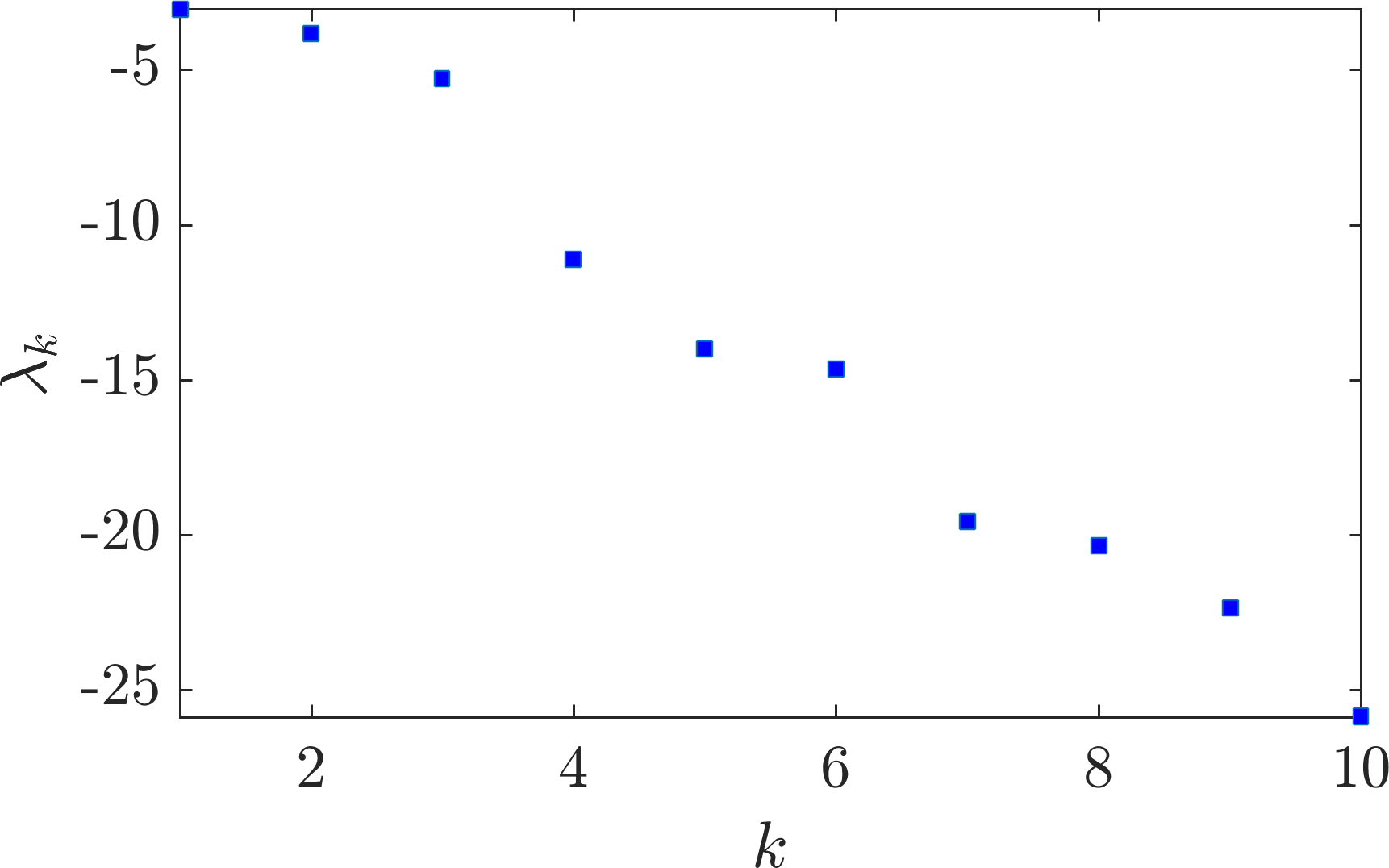}
\includegraphics[width=0.4\textwidth]{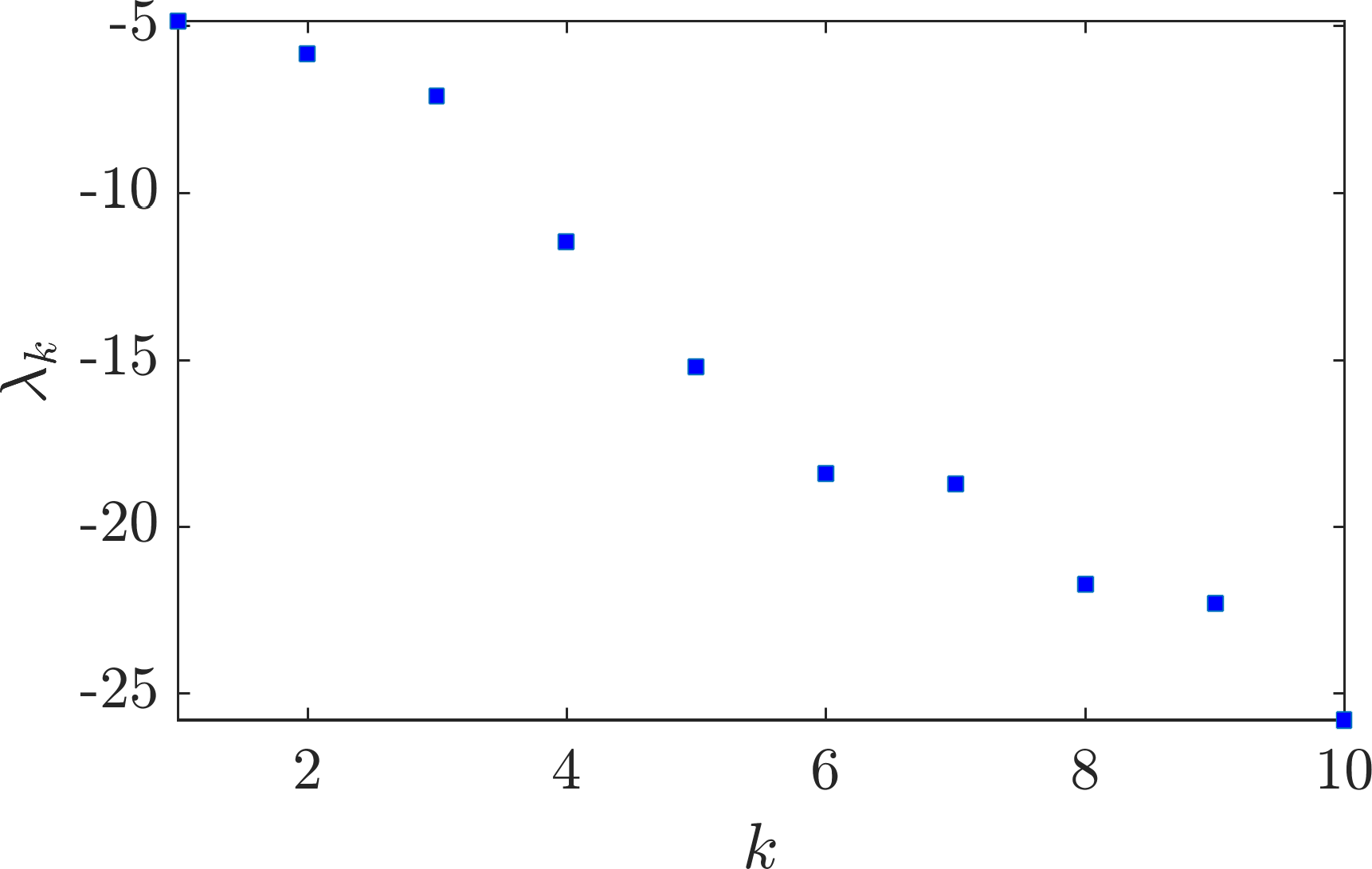}
\caption{Ocean flow: spectrum of the dynamic Laplacian using the Cauchy-Green approach (left) and the adaptive transfer operator approach (right).}
\label{fig:ocean_spectra}
\end{center}
\end{figure}

\begin{figure}[htbp]
\begin{center}
\includegraphics[width=0.48\textwidth]{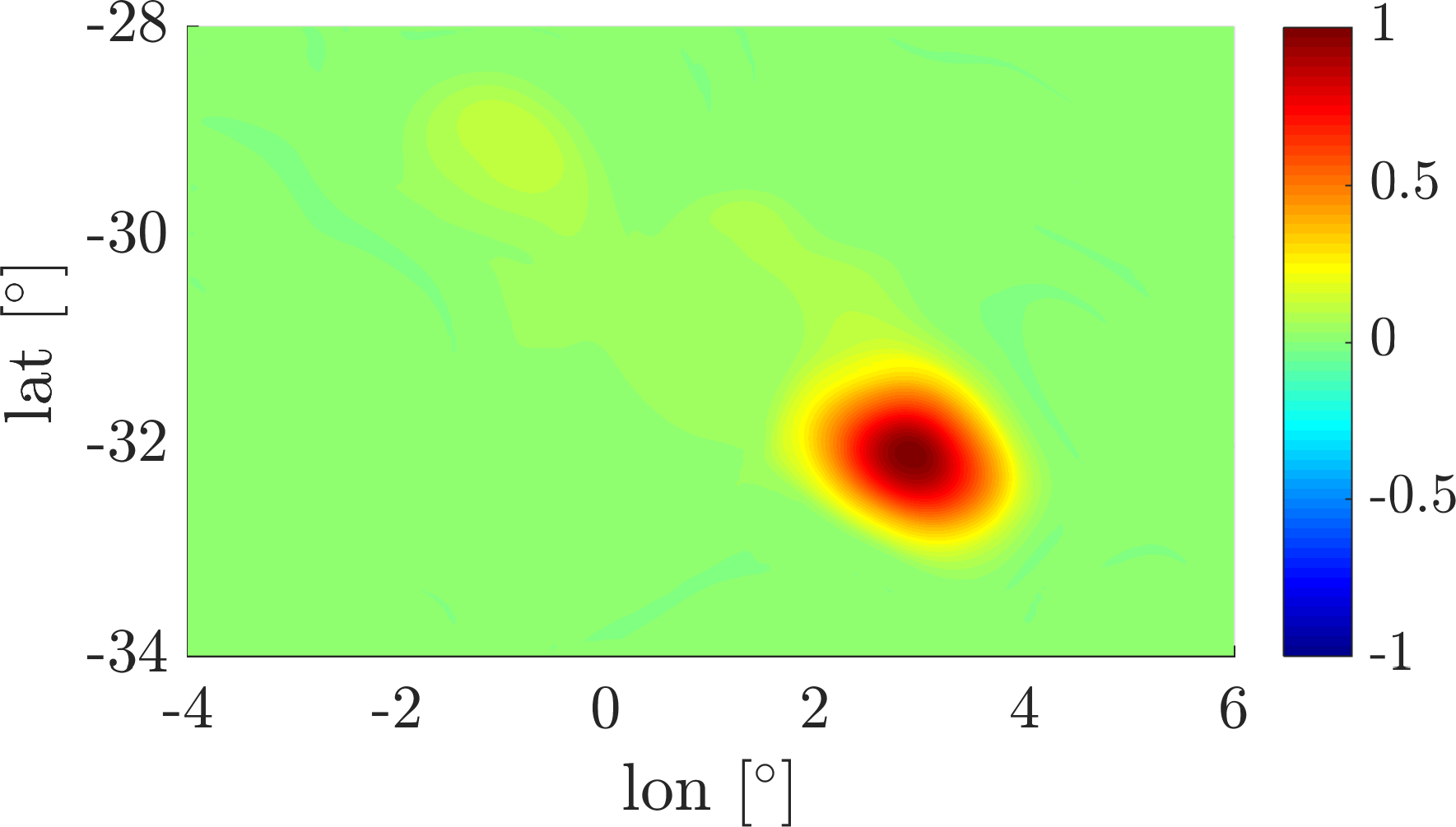}
\includegraphics[width=0.48\textwidth]{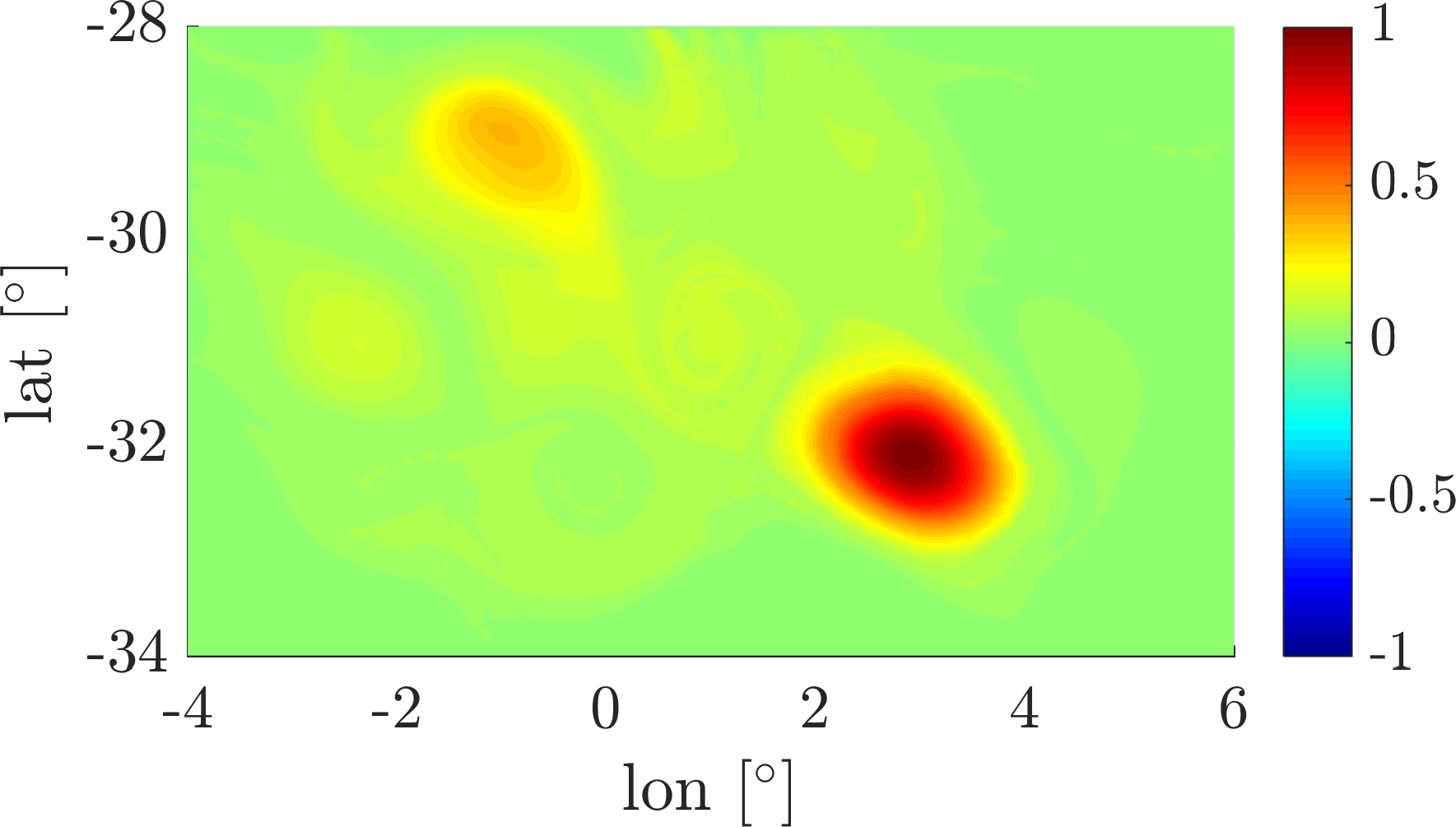}
\includegraphics[width=0.48\textwidth]{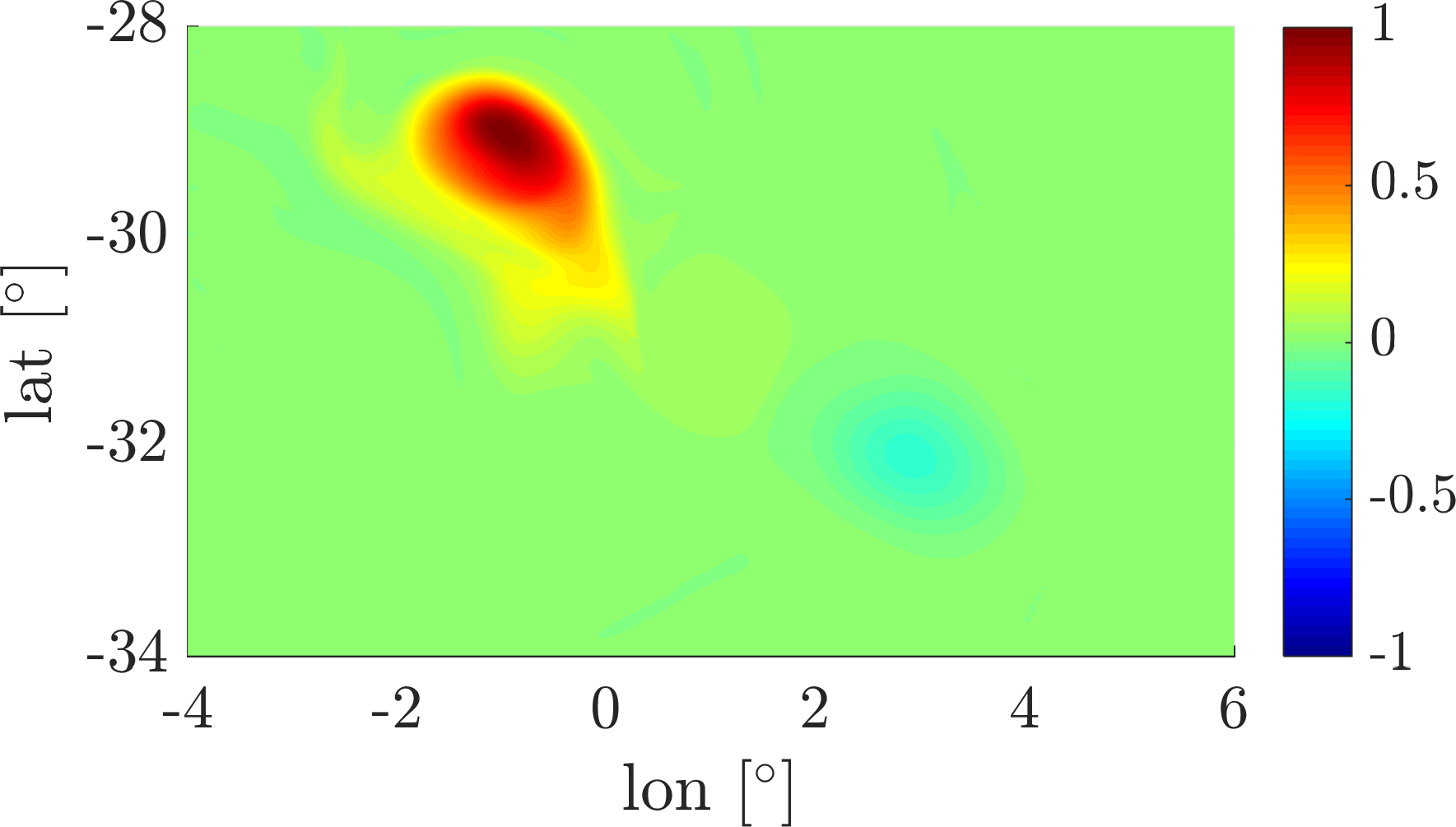}
\includegraphics[width=0.48\textwidth]{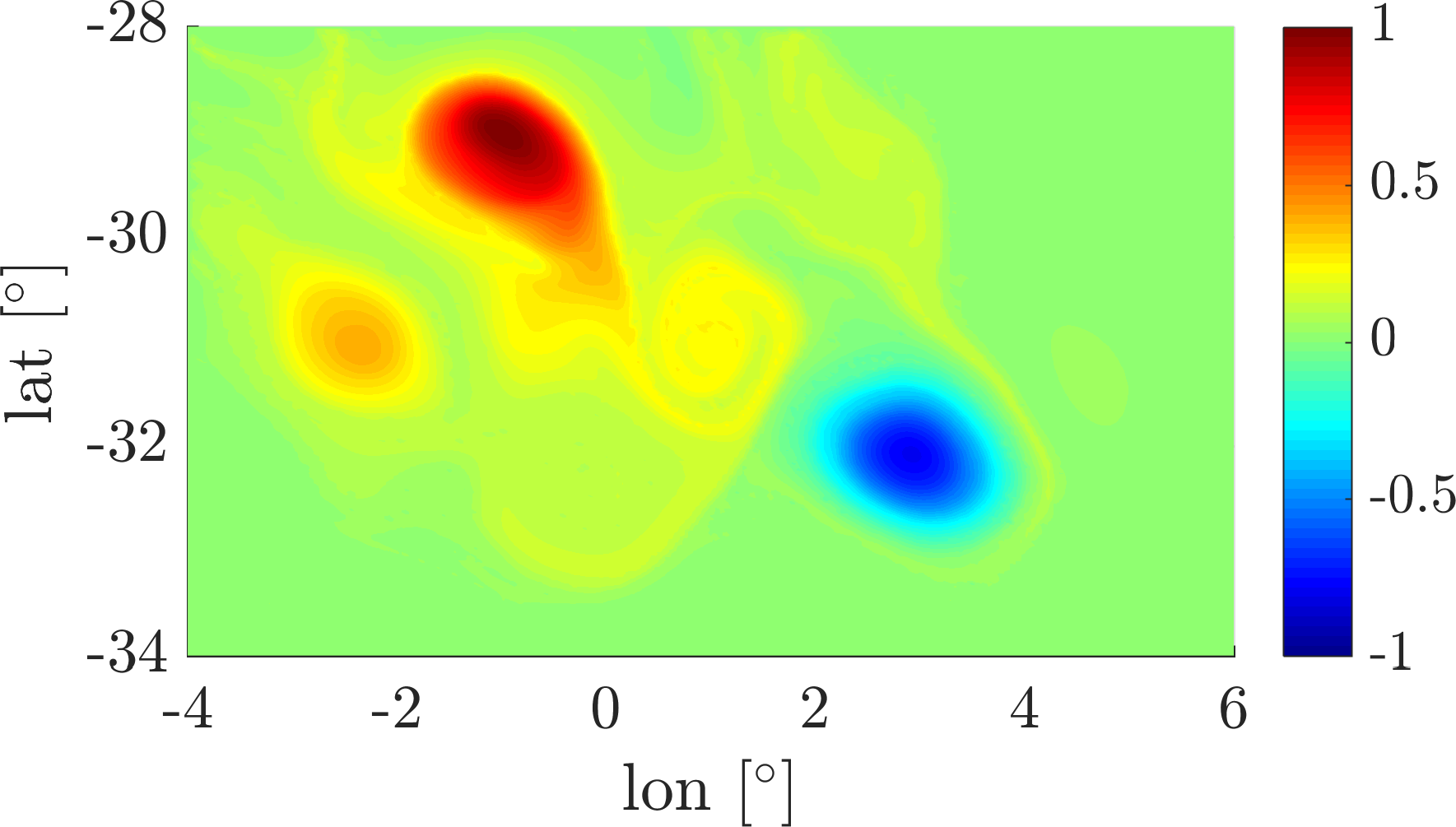}
\includegraphics[width=0.48\textwidth]{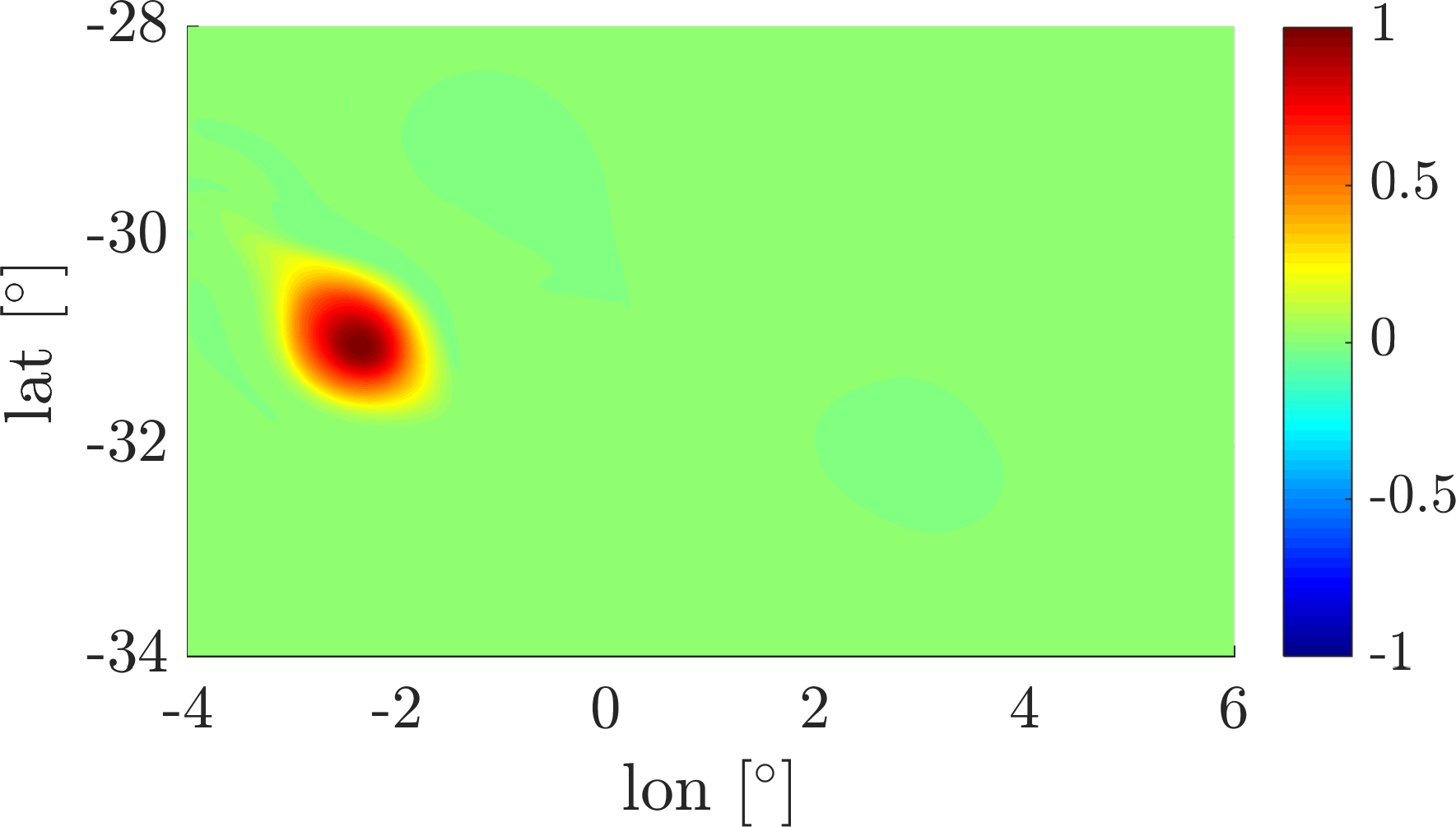}
\includegraphics[width=0.48\textwidth]{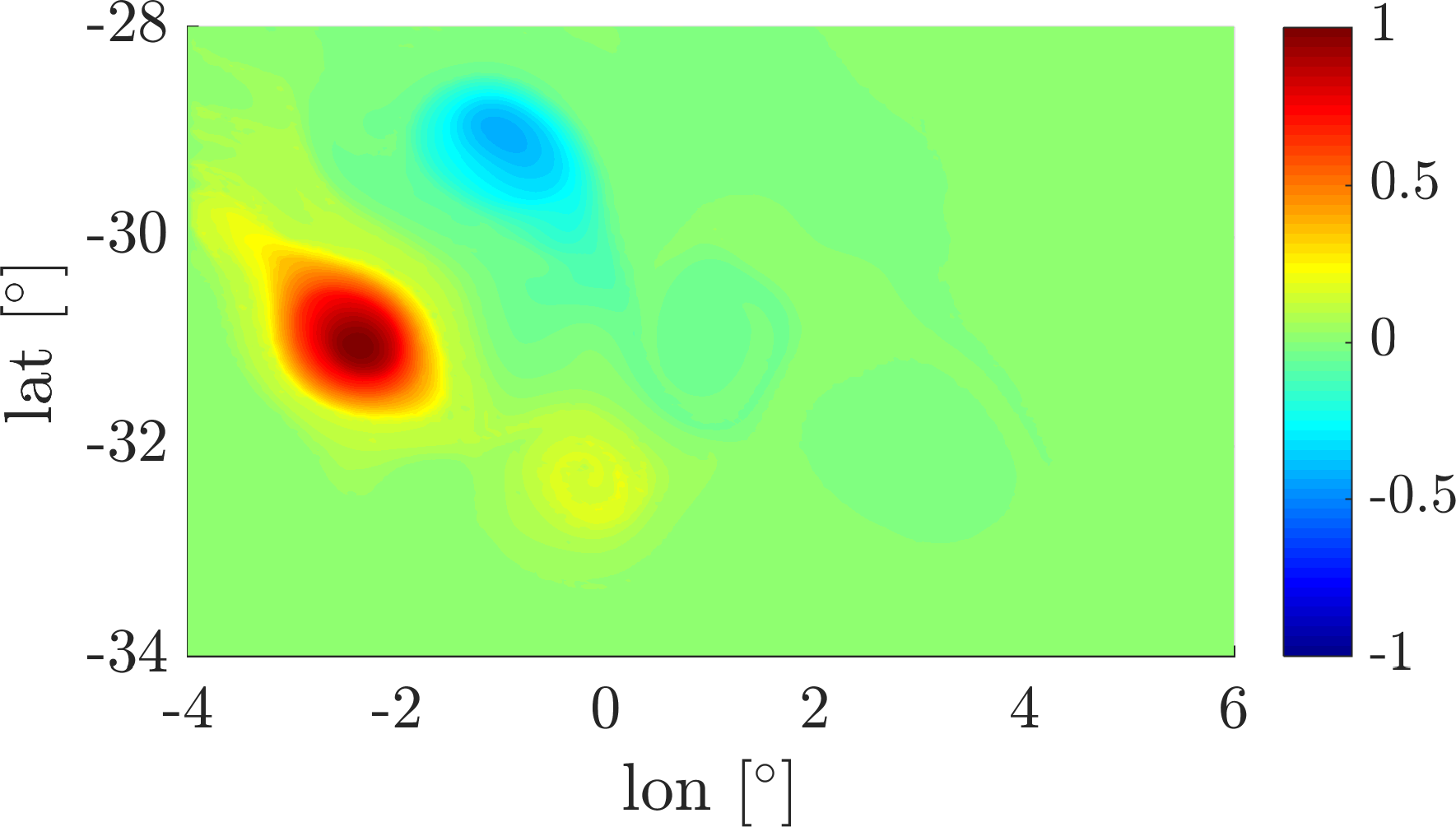}
\caption{Ocean flow: the first three eigenvectors of the dynamic Laplacian using the Cauchy-Green approach (left) and the adaptive transfer operator approach (right).}
\label{fig:ocean_CG_evs}
\end{center}
\end{figure}

\begin{figure}[htbp]
\begin{center}
\includegraphics[width=0.48\textwidth]{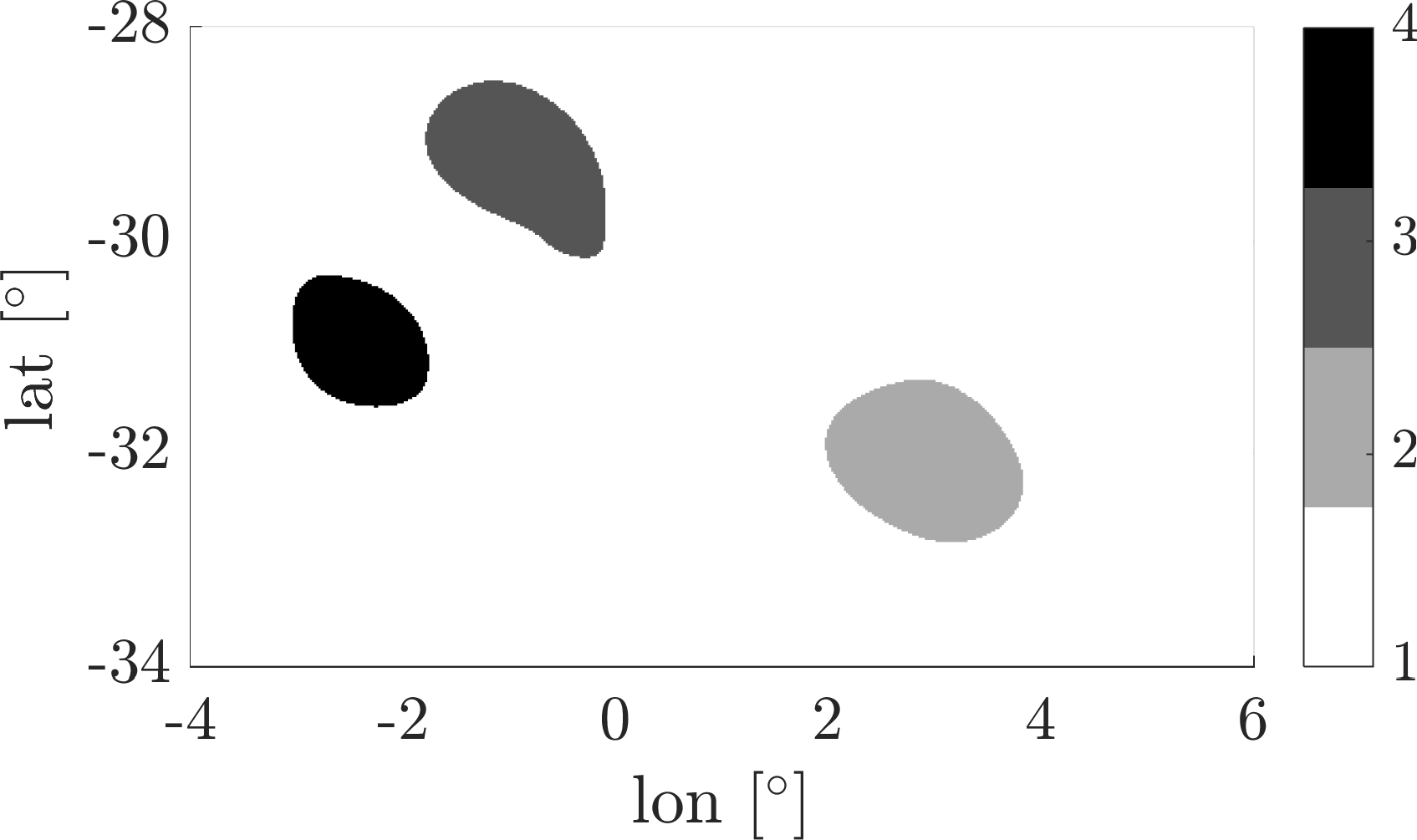}
\includegraphics[width=0.48\textwidth]{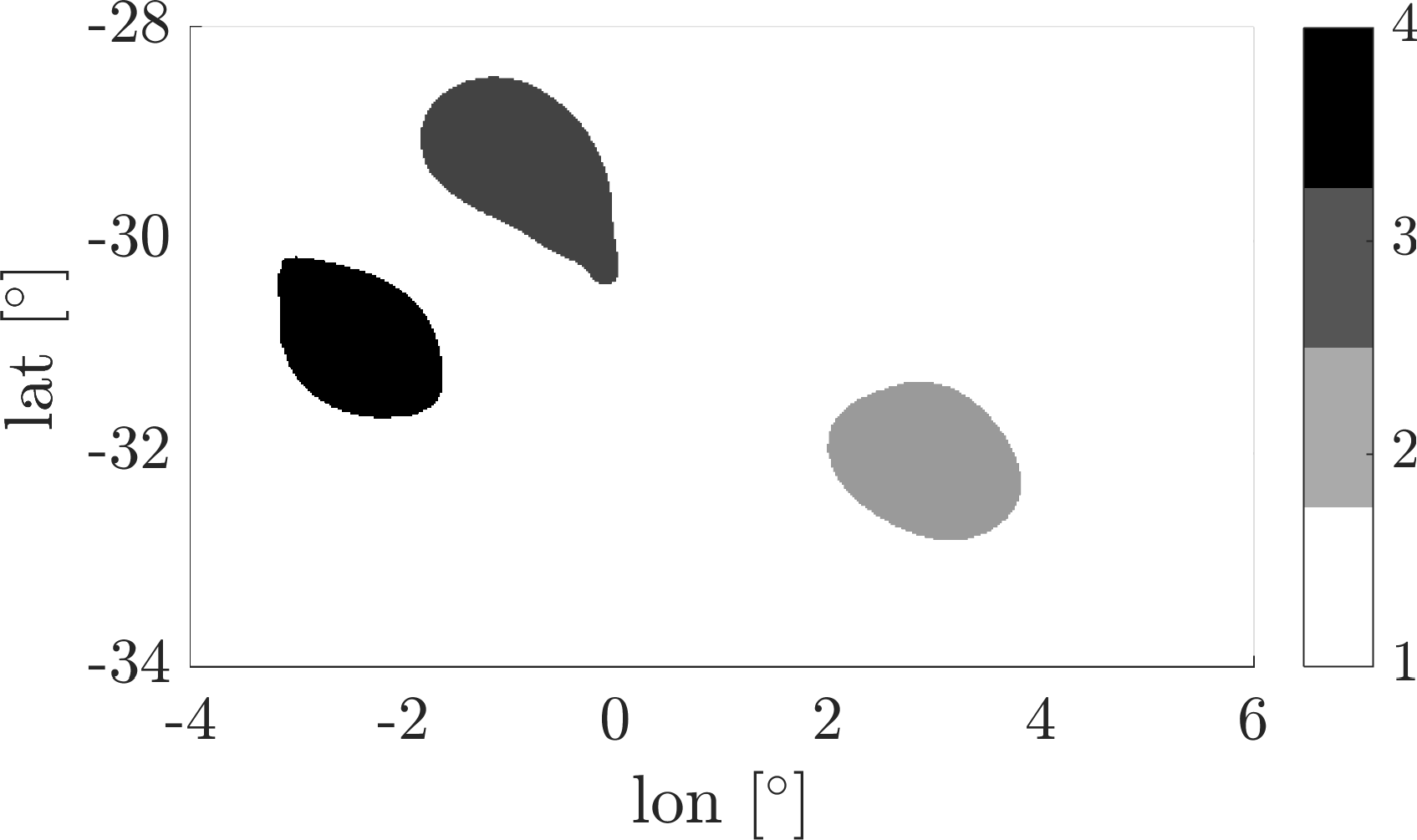}
\caption{Ocean flow: coherent sets identified using the first 3 eigenfunctions by the Cauchy-Green approach (left) and by the adaptive transfer operator approach (right).}
\label{fig:ocean_cluster4}
\end{center}
\end{figure}

\begin{figure}[htbp]
\begin{center}
\includegraphics[width=0.55\textwidth]{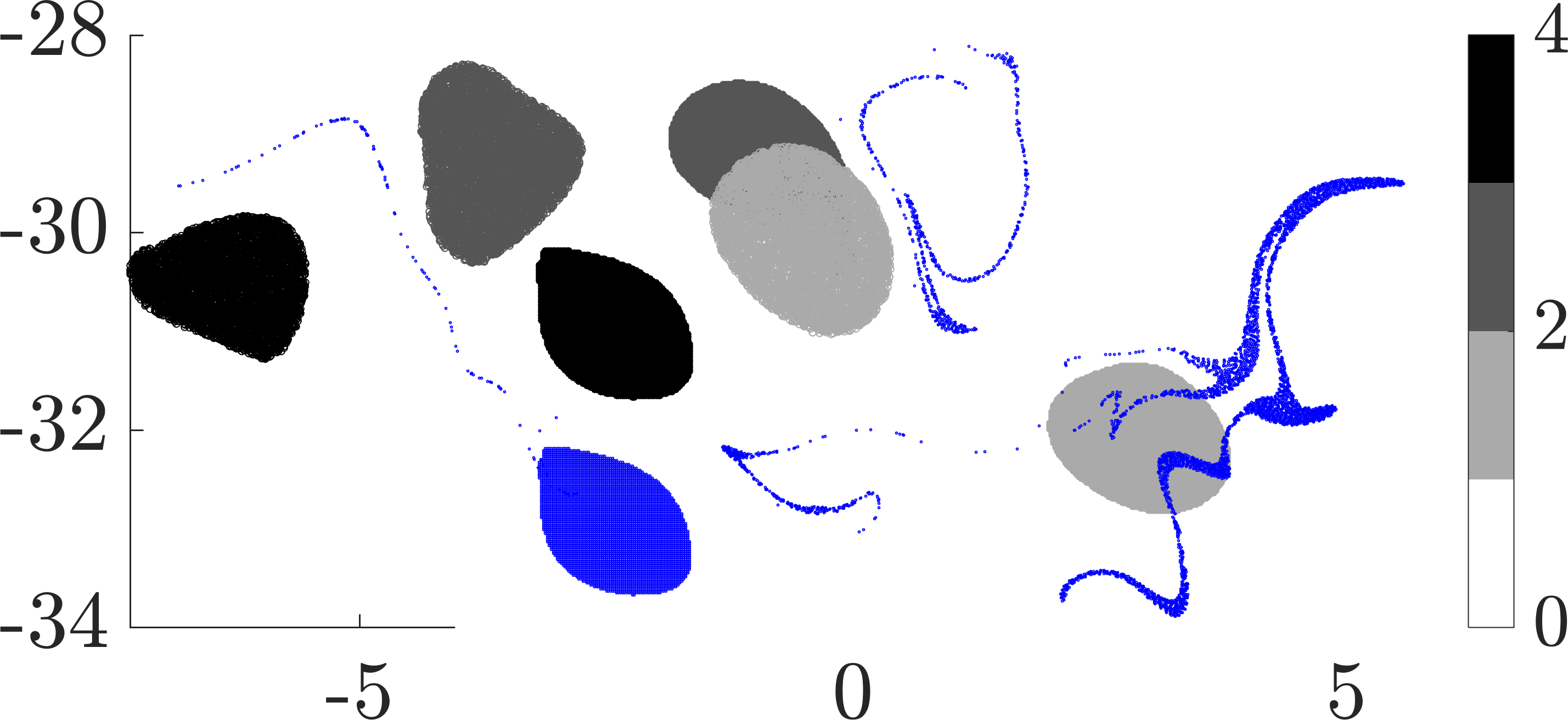}
\caption{Ocean flow: coherent sets from Fig.~\ref{fig:ocean_cluster4} (left) and their images (gray) as well as a non-coherent set and its image (blue).}
\label{fig:ocean_cs_evol}
\end{center}
\end{figure}

\paragraph{Adaptive transfer operator approach.}

The approach from Section~\ref{sec:colladap} yields  comparable results as shown in the right columns of Fig.~\ref{fig:ocean_spectra}.  The computation times are: time integration: 4 s, assembly: 1.4 s, solution of the eigenproblem: 16 seconds.
We can decrease the resolution down to a $150\times 90$ grid before the results start to deteriorate significantly.

\paragraph{Missing data.}

We finally analyze how well the coherent sets can be recovered in the ocean flow experiment when data is missing.  Again, we use the same set of nodes and $|\mathcal{T}|=10$ equidistant intermediate time steps.  We then randomly delete 70\% of the data, yielding roughly 12000 data points per time step.  For comparison, in the experiment with full data and two time steps above we used around 37500 nodes per time step. The result shown in Figure~\ref{fig:ocean_missing} is  qualitatively the same as the one with full data in Figure~\ref{fig:ocean_cluster4}.  We note, however, that when we delete even more data, these coherent sets will not be recovered as clearly any more.
\begin{figure}[htbp]
\centering
\includegraphics[width=0.48\textwidth]{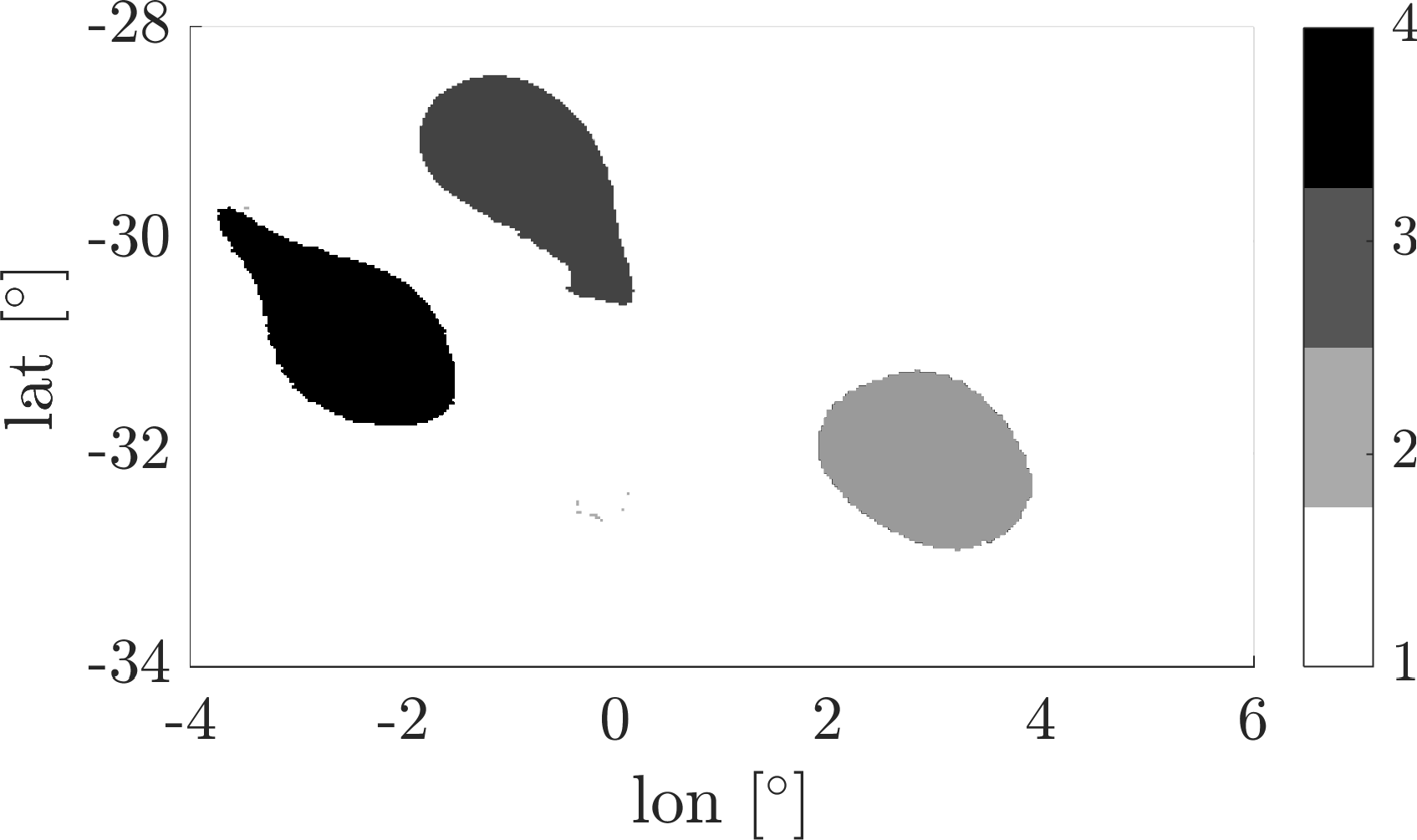}
\caption{Ocean flow with missing data: coherent sets based on the adaptive transfer operator approach from Section~\ref{sec:colladap} with 10 intermediate time steps  and 70\% of the data randomly removed.}
\label{fig:ocean_missing}
\end{figure}

\subsection{Experiment: The unsteady ABC flow (3D)}
\label{exp:ABC_CG}

For a 3D experiment, we consider the unsteady ABC flow, cf.~\cite{haller01}, given by
\begin{align}
\dot x & = (A+\tfrac12 t\sin(\pi t))\sin z + C \cos y\\
\dot y & = B\sin x + (A+\tfrac12 t\sin(\pi t))\cos z\\
\dot z & = C\sin y + B\cos x
\end{align}
on the 3-torus, with parameter values $A = \sqrt{3}, B = \sqrt{2}, C = 1$ on the time interval $t\in [0,1]$.

\paragraph{Cauchy-Green approach.}

We employ a Delaunay triangulation on a regular grid of $25\times 25\times 25=15625$ points, yielding about 83000 tetrahedra, Gauss quadrature of degree 1, i.e.\ one quadrature point per tetrahedron and two time steps, i.e.~$\mathcal{T}=\{0,1\}$. The integration of the variational equation takes 11 s, the assembly of the matrices 0.9 s and the solution of the eigenproblem 9 seconds.  Figure~\ref{fig:ABC_spectra} (left) shows the spectrum of the discrete dynamic Laplacian,  Figure~\ref{fig:ABC_evs} (top) the 2nd (left) and 3rd (right) eigenvector.
\begin{figure}[htbp]
\begin{center}
\includegraphics[width=0.32\textwidth]{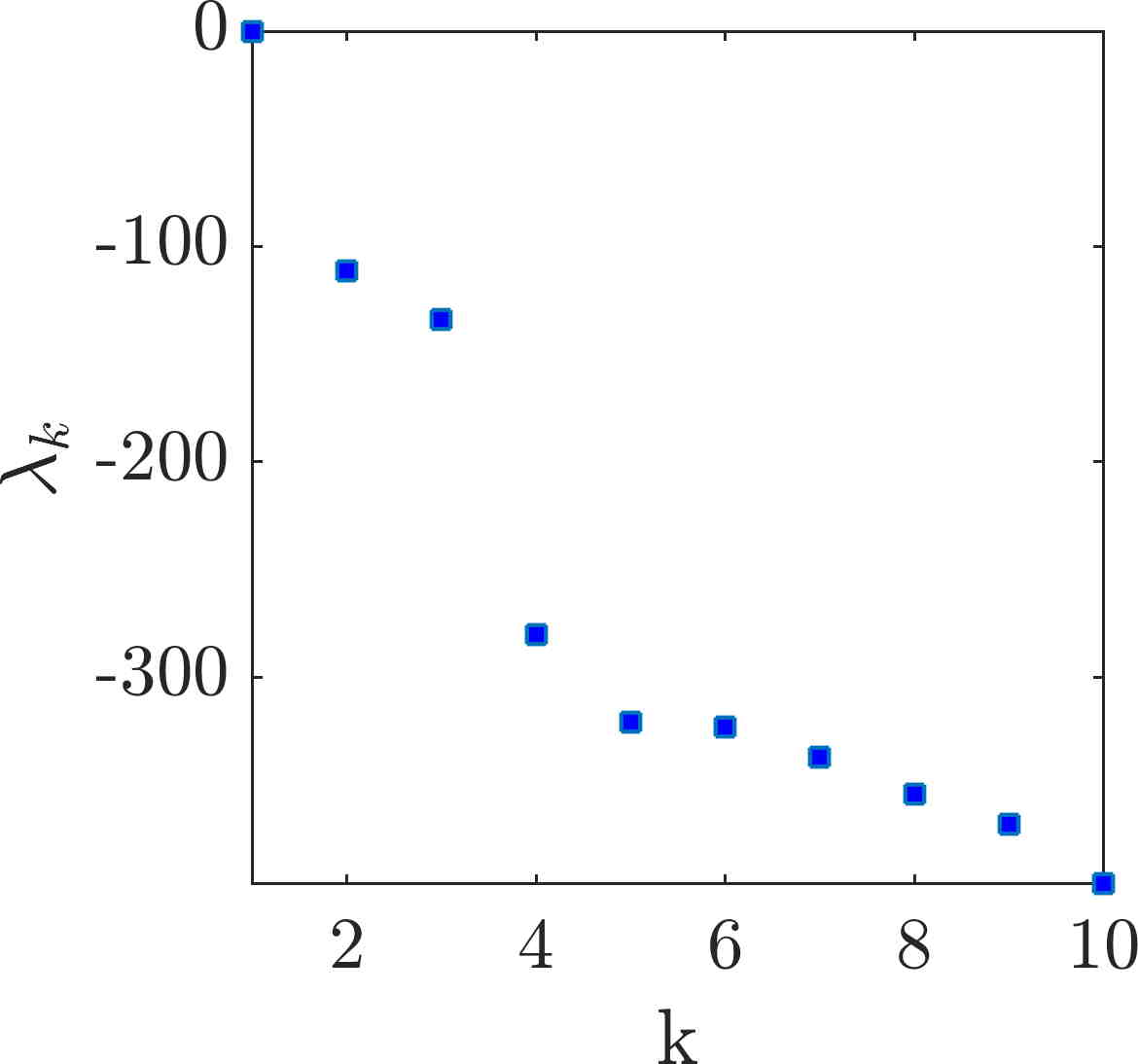}
\includegraphics[width=0.32\textwidth]{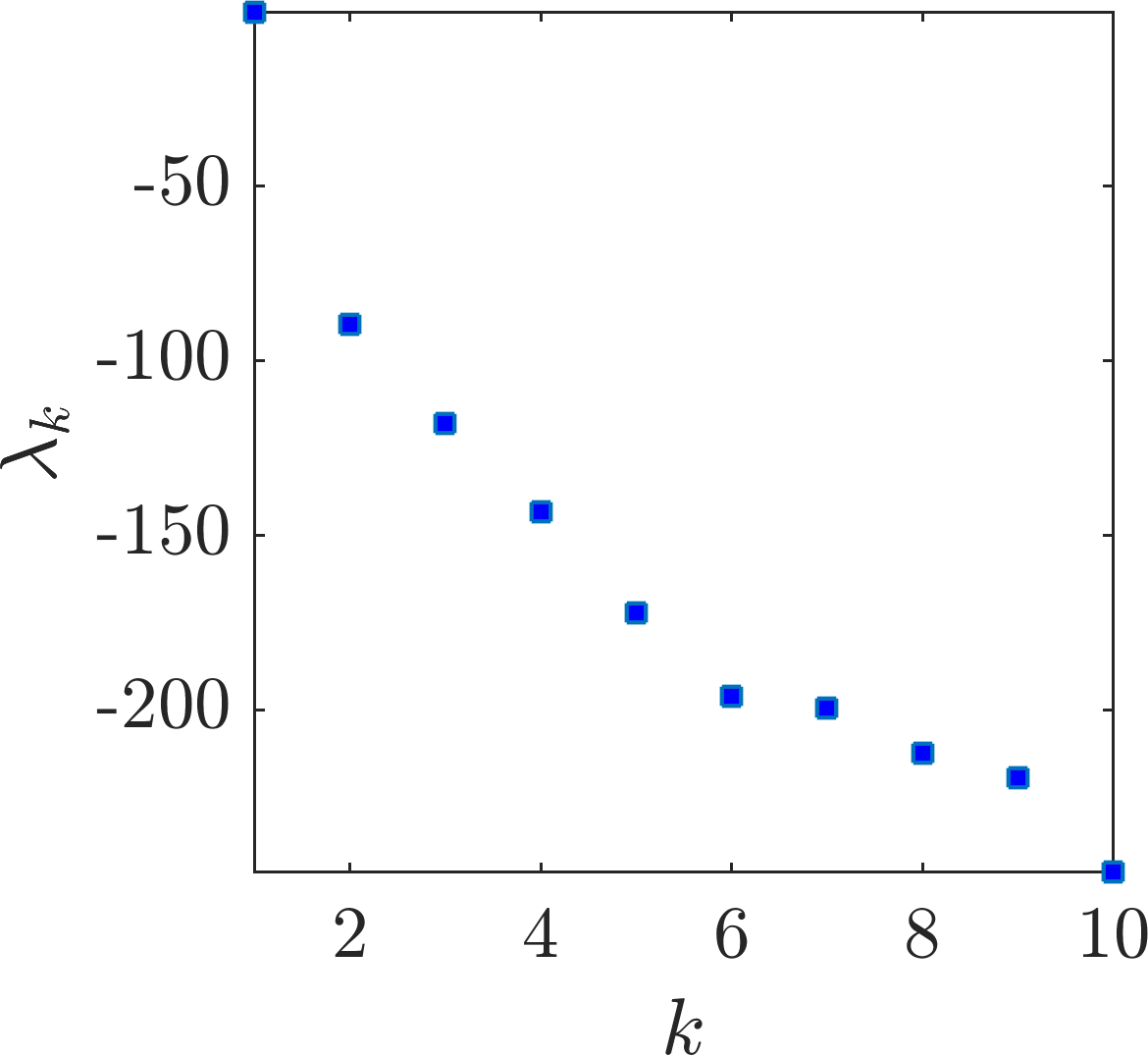}
\includegraphics[width=0.32\textwidth]{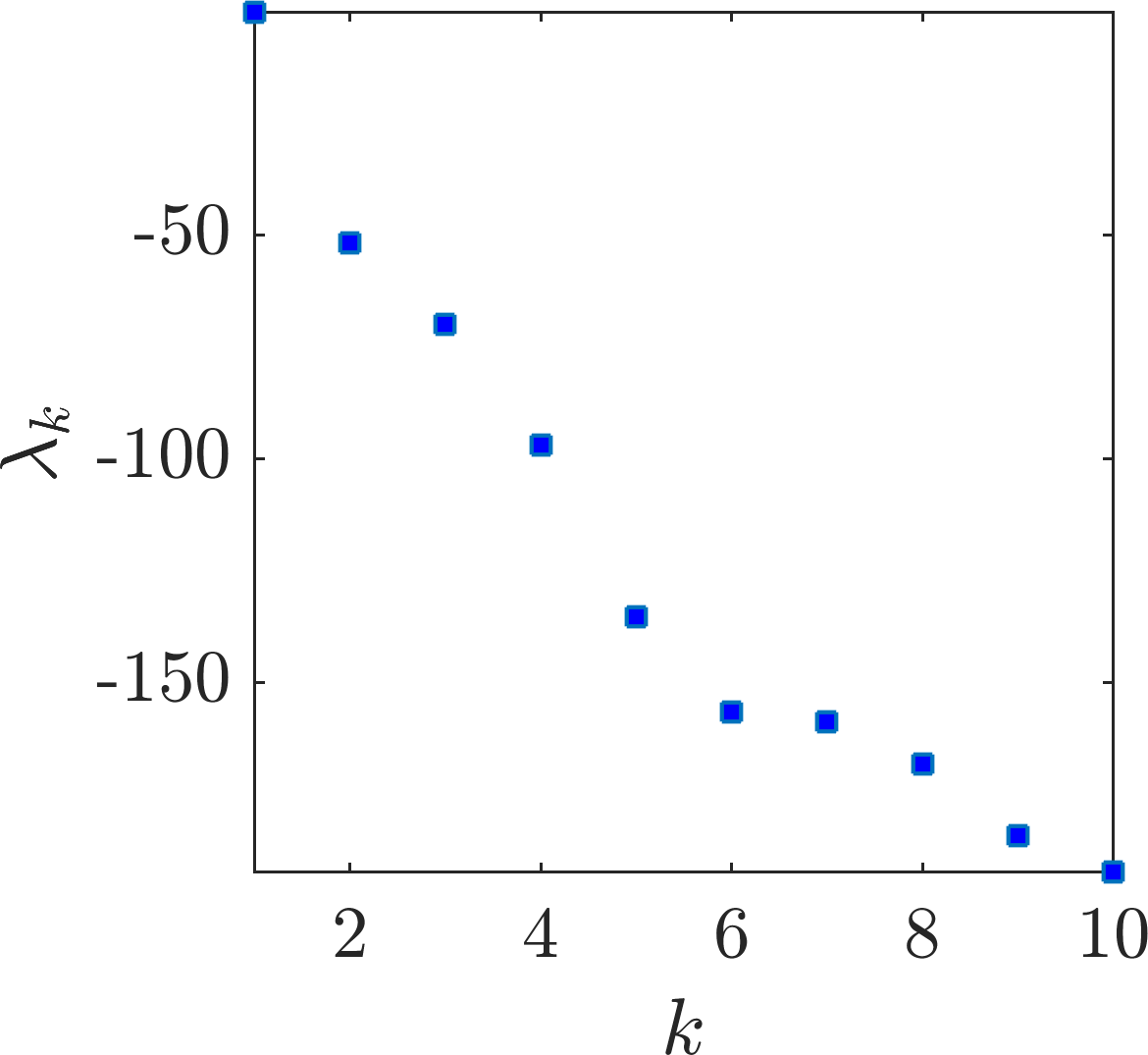}
\caption{ABC flow: Spectrum of the dynamic Laplacian for the Cauchy-Green (left) the non-adaptive (center) and the adaptive (right) transfer operator approach.}
\label{fig:ABC_spectra}
\end{center}
\end{figure}
\begin{figure}[htbp]
\begin{center}
\includegraphics[width=0.49\textwidth]{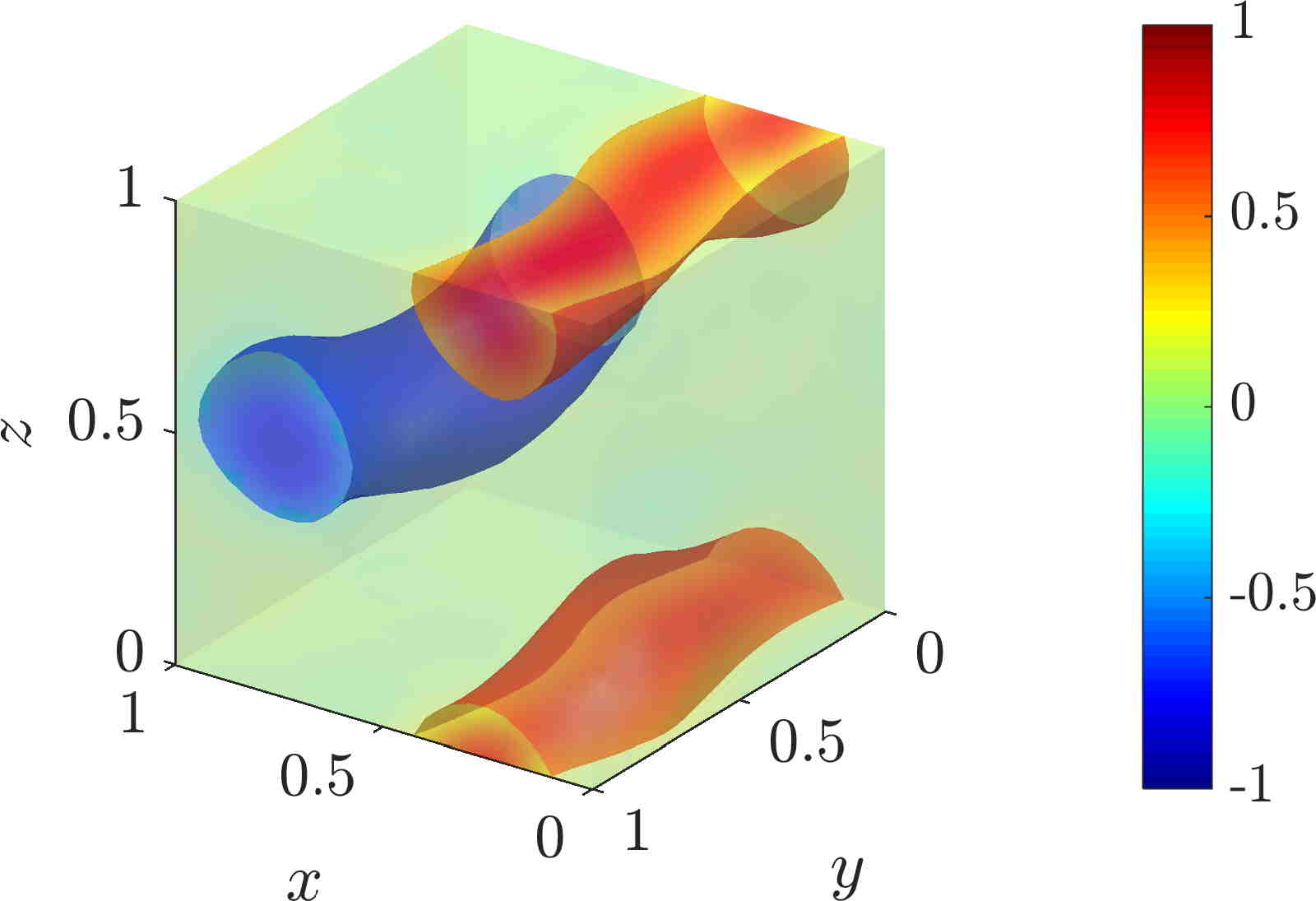}
\includegraphics[width=0.49\textwidth]{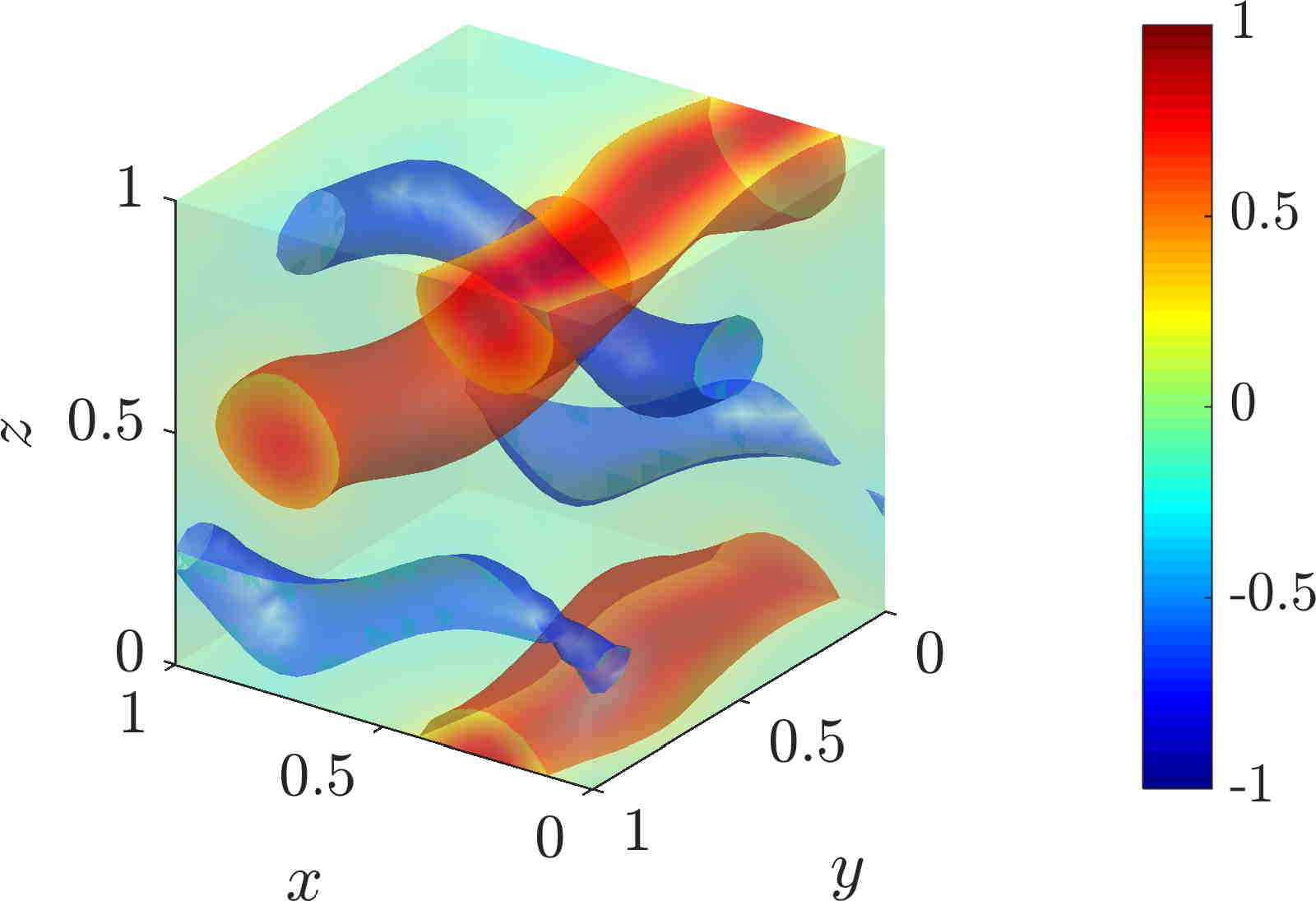}
\includegraphics[width=0.49\textwidth]{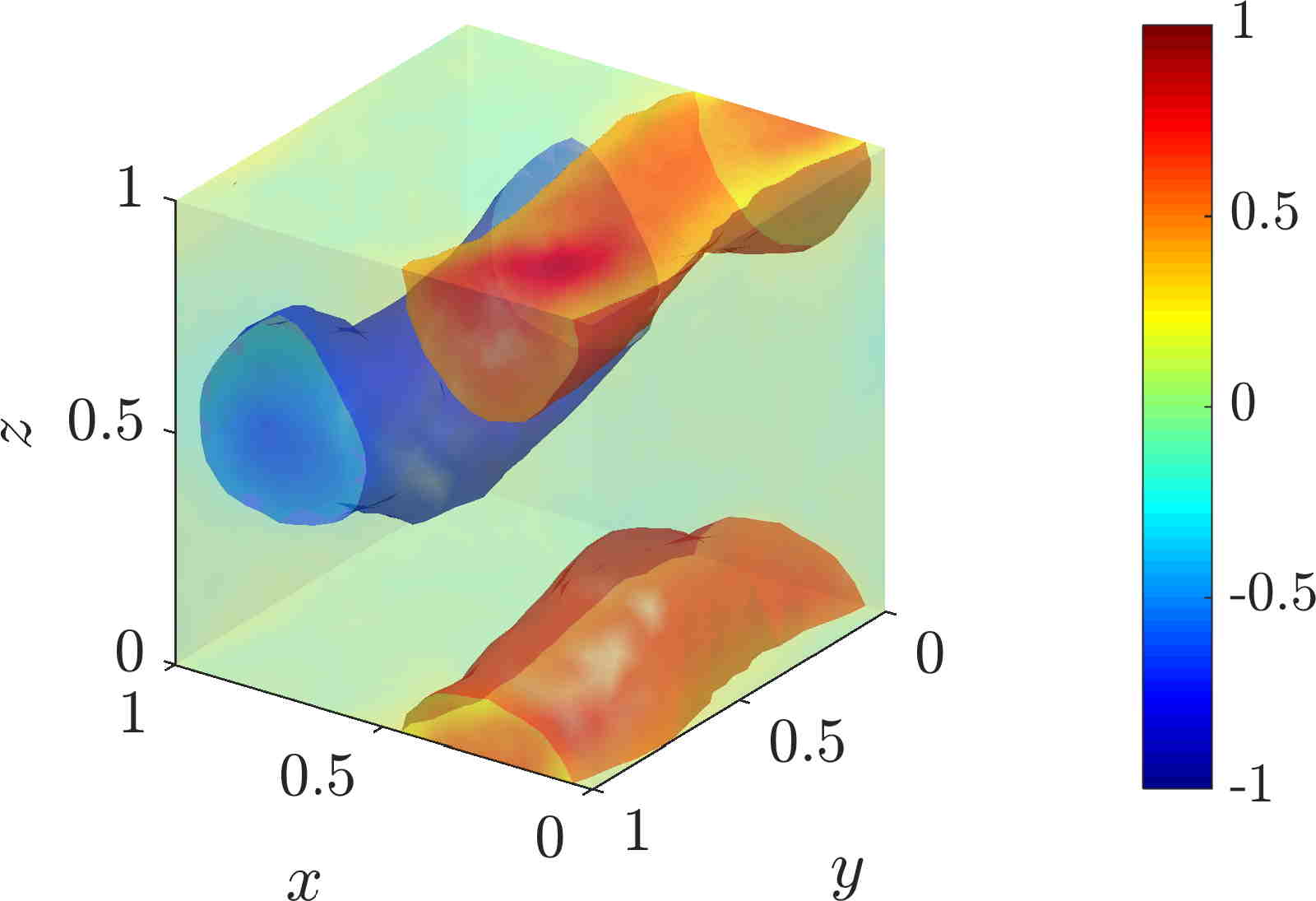}
\includegraphics[width=0.49\textwidth]{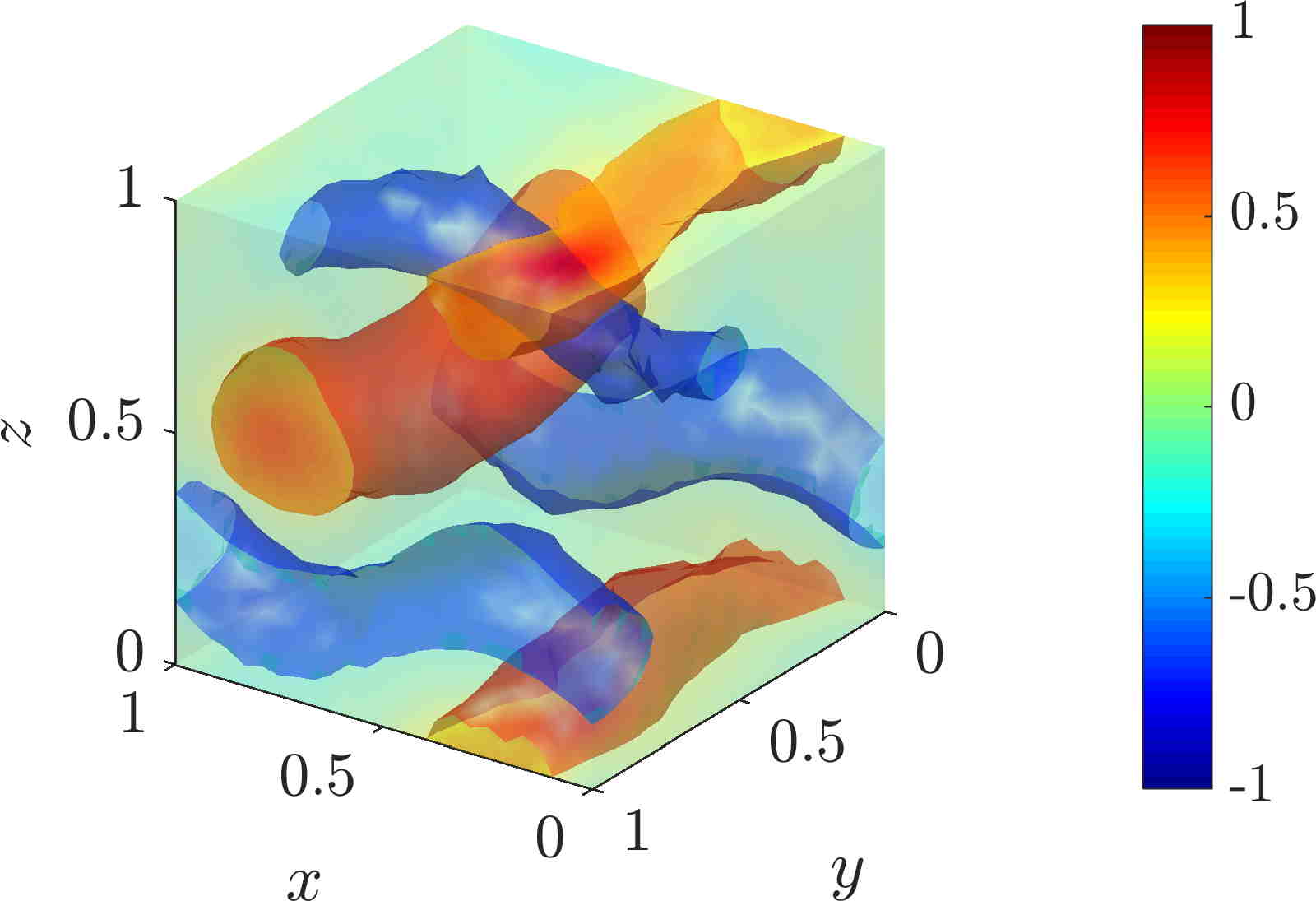}
\includegraphics[width=0.49\textwidth]{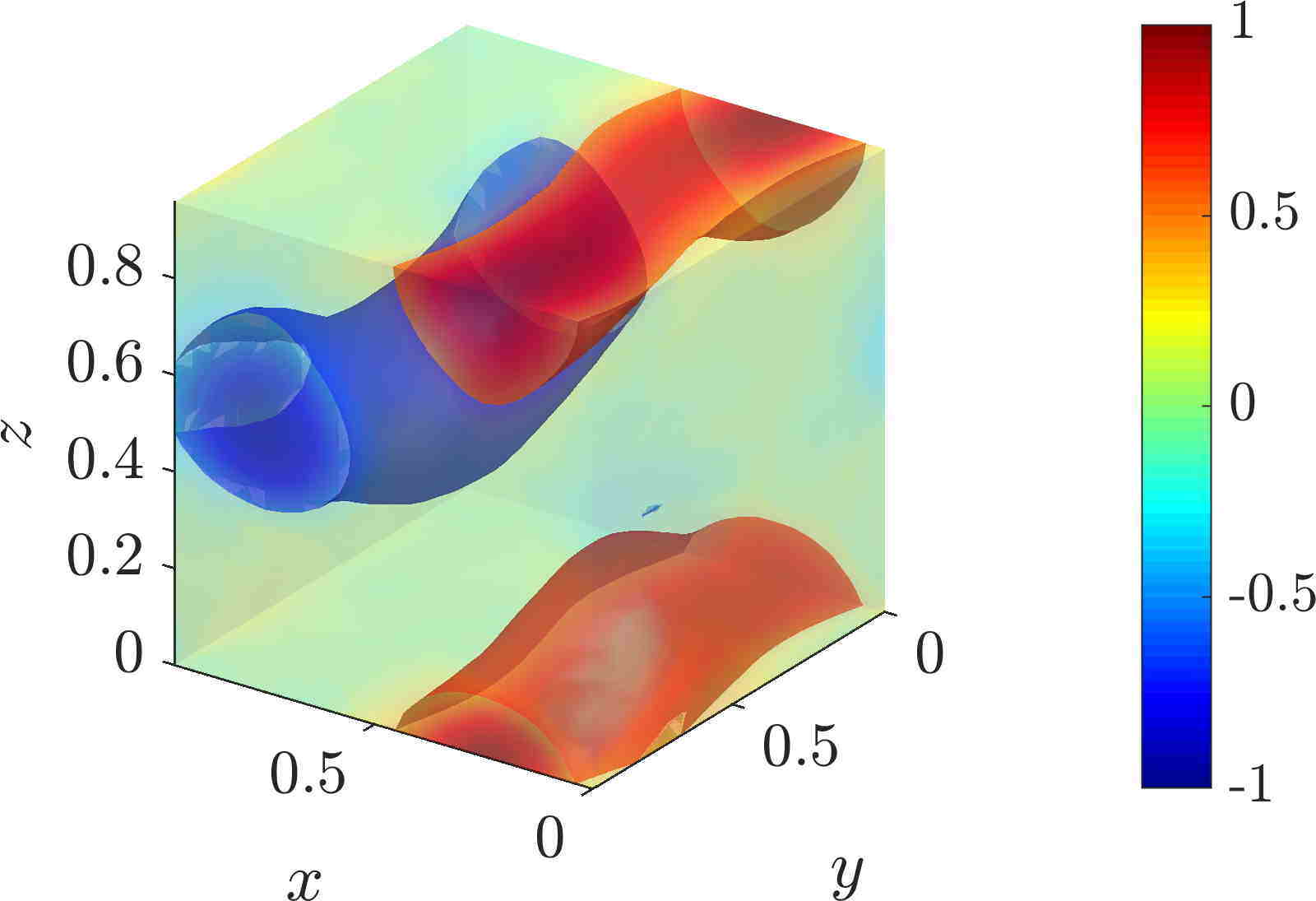}
\includegraphics[width=0.49\textwidth]{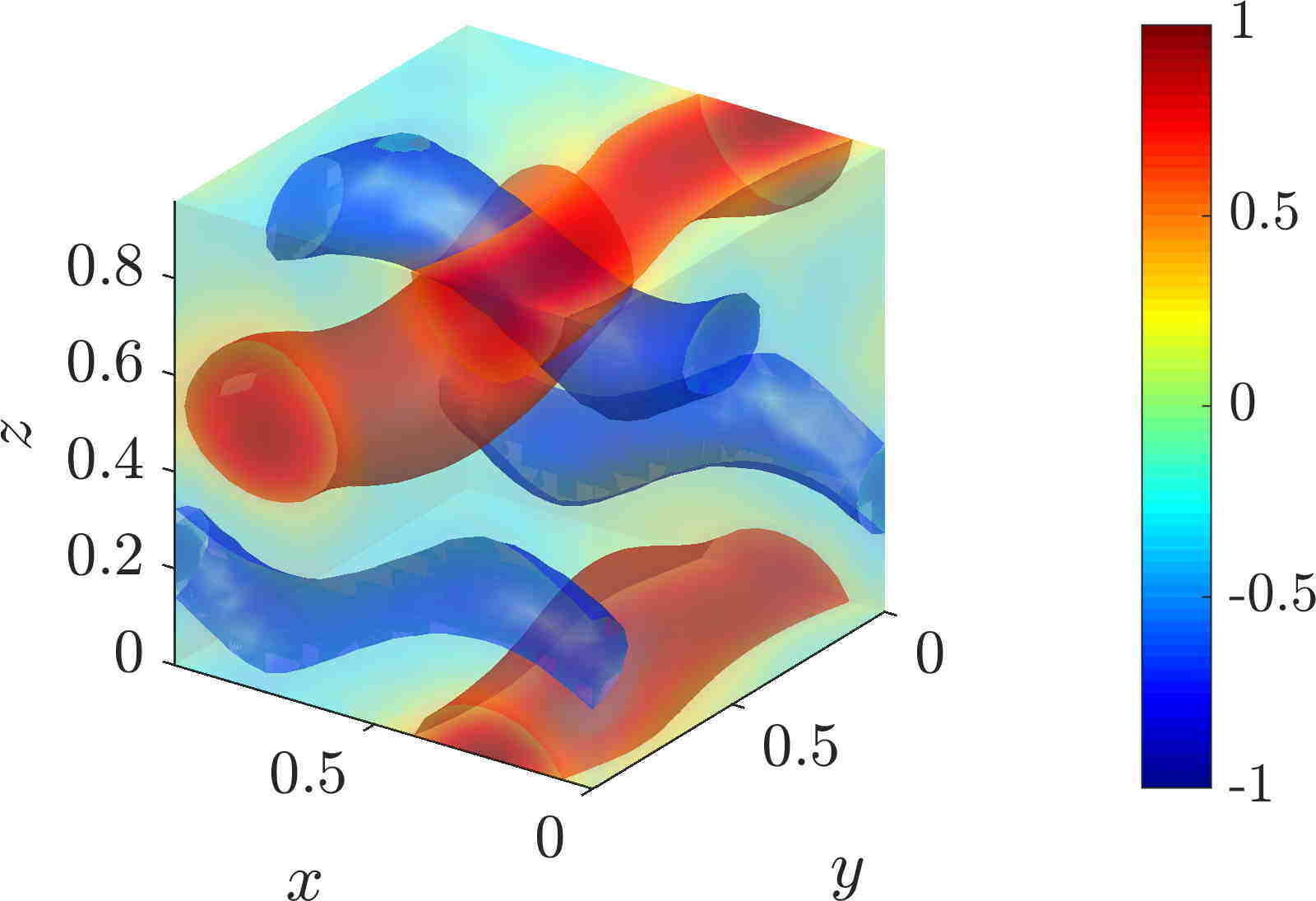}
\caption{ABC flow: 2nd (left) and 3rd (right) eigenvector using the Cauchy-Green  (top), the transfer operator (center) and the adaptive transfer operator approach (bottom).}
\label{fig:ABC_evs}
\end{center}
\end{figure}

\paragraph{Transfer operator approaches.}

Using the same parameters as for the Cauchy-Green approach, we obtain the spectra shown in Figure~\ref{fig:ABC_spectra} (center and right) and the eigenvectors in Figure~\ref{fig:ABC_evs} (center and bottom).  This yields qualitatively the same results as the Cauchy-Green approach, the eigenfunctions from the non-adaptive transfer operator approach, however, appear to be less smooth. This seems to be due to the relatively coarse approximation of the transfer operator.  Computation times are: non-adaptive: time integration 0.4 s, assembly 0.6 s, computation of the $\alpha$-matrix 1.4 s, eigenproblem 160 s; adaptive: integration 0.4 s, assembly 1.4 s, eigenproblem 40 s.  Note that the increased computation time for solving the eigenproblem is due to the fact that the stiffness matrix has more nonzero entries ($10\times$ as many for the non-adaptive and twice for the adaptive TO approach in comparison to the CG approach, cf.\ the remark in Section~\ref{sec:collnonadap} on the sparseness of the stiffness matrix in this case).

\section*{Acknowledgements}

We thank Daniel Karrasch for helpful comments and contributions to the code, and Michael Feischl for helpful comments.  GF thanks the Department of Mathematics at the Technical University Munich for hospitality during a visit in May 2016, and the John von Neumann Professorship scheme for financial support for this visit.  GF was supported by an ARC Future Fellowship. OJ was supported by the Priority Programme SPP 1881 Turbulent Superstructures of the Deutsche Forschungsgemeinschaft.
\begin{appendix}

\section{The non volume-preserving case}

We now briefly sketch the theory behind the case of the underlying flow $\Phi^t$ not being volume-preserving, where we wish to track coherent masses according to some smooth initial mass distribution on $\M\subset\R^d$ given by a probability measure $\mu^0$ ($\mu^0=m$ would correspond to volume), and where the domain is possibly curved with the curvature described by a Riemannian metric.
The measure $\mu^0$ is evolved by $\Phi^t$ according to $\mu^t:=\mu\circ \Phi^{-t}$.
All computational aspects for data embedded in Euclidean space are described in Section \ref{sect:nvp}.

The expressions (\ref{dyniso}) and (\ref{sobolev}) can be naturally extended to cover this situation.
Firstly (\ref{dyniso}), where the size of the disconnector $\Gamma$ is computed according to the evolved measure $\mu^t$:
\begin{equation*}
\label{cheeger0}
\mathbf{h}(\Gamma):=\frac{\frac{1}{|\mathcal{T}|}\sum_{t\in\mathcal{T}} \mu_{d-1}^t(\Phi^t\Gamma)}{\min\{\mu_{d-1}^0(\M_1),\mu_{d-1}^0(\M_2)\}},
\end{equation*}
where $\mu_{d-1}^t$ is the induced measure on $d-1$-dimensional surfaces at time $t$.
Secondly, (\ref{sobolev}),
\begin{equation*}
\mathbf{s}:=\inf_{f\in C^\infty(\M,\mathbb{R})} \frac{\frac{1}{|\mathcal{T}|}\sum_{t\in\mathcal{T}}\|(|\nabla_{m^t}\Phi_*^tf|_{m^t})\|_{\mu^t}}{\inf_\alpha \|f-\alpha\|_{\mu^{0}}},
\end{equation*}
where the subscripts $m^t$ denote that the computations are taken with respect to the Riemannian metric $m^t$ on $\Phi^t(\M)$, which in many cases will be either the Euclidean metric, or if $\M$ is of dimension lower than $d$, the induced metric arising from the Euclidean metric.
See \cite{FrKw16} for details.

Finally, we replace (\ref{eq:eigenproblem}) with
\begin{equation}
 \label{eq:lap_weighted}
 \left(\frac{1}{|\mathcal{T}|}\sum_{t\in\mathcal{T}}(\Phi^t)^*\Delta_{\mu^t} \Phi^t_*\right)v=\lambda v
 \end{equation}
 where $\Delta_{\mu^t}$ is the $\mu^t$-weighted Laplace operator on $\Phi^t(\M)$ (see \cite{FrKw16} for details).

The weak form of the eigenproblem (\ref{eq:lap_weighted}) can be written as \cite{FrKw16}
\begin{equation*}
 \label{weakweightedeigen}
 \frac{1}{|\mathcal{T}|}\sum_{t\in\mathcal{T}}\int_{\Phi^t(\M)}\nabla_{m^t}(\Phi^t_*\psi)\cdot\nabla_{m^t}(\Phi^t_{*}v)\ d\mu^t=\lambda\int_{\M}\psi v\ d\mu,\quad\text{for all }\psi\in
C^{\infty}(\Omega).
\end{equation*}
Again, if we want to have a continuous-time version of the above eigenproblem, we simply replace the average over $t\in\mathcal{T}$ with an integral.

\section{Appendix: Code example}

We here provide sample code which performs the computations reported on in Section \ref{exp:rotating_double_gyre}.  Note that this code has been stripped down for readability. A more efficient version can be downloaded from \href{https://github.com/gaioguy/FEMDL}{\texttt{https://github.com/gaioguy/FEMDL}}.

\begin{Code}[H]
\hspace*{12pt}\scalebox{0.95}{
\begin{lstlisting}[breaklines=false,numbers=left,linerange={1-50}]
clear all
% flow map, data points and time integration
t0 = 0; tf = 1; nt = 5;                           % init, final time, steps 
T = @(x) flow_map(x,linspace(t0,tf,nt));            
n = 625; p0 = rand(n,2);                            % data points
P = T(p0);  for k = 1:nt, p{k}=P(:,[k k+nt]); end;  % time integration

% assembly
D = sparse(n,n); M = sparse(n,n);
for k = 1:nt                                        
    t{k} = delaunay(p{k});                          
    [Dt,Mt{k}] = assemble(p{k},t{k});                   
    D = D + Dt/nt;  M = M + Mt{k}/nt;
end;

% solve eigenproblem
[V,L] = eigs(D,M,10,'SM');                          
[lam,I] = sort(diag(L),'descend'); V = V(:,I);

% compute coherent partition
m = 200; x = linspace(0,1,m); [X,Y] = ndgrid(x,x);  
k = 3; for j = 1:k,
    v = scatteredInterpolant(p0(:,1),p0(:,2),V(:,j));
    W(:,j) = reshape(v(X,Y),m*m,1); 
end
idx = kmeans(W,k);

% plot spectrum, 2nd eigenvector and partition
figure(1); clf; plot(lam,'*'); title('spectrum')         
figure(2); clf; trisurf(t{1},p0(:,1),p0(:,2),zeros(n,1),V(:,2)); 
shading interp, view(2), axis equal, axis tight, colorbar 
hold on, triplot(t{1},p0(:,1),p0(:,2),'k'); title('2nd eigenvector');
figure(3); clf; surf(X,Y,reshape(idx,m,m)); shading flat 
view(2), axis equal, axis tight, title('coherent partition')

\end{lstlisting}}
\caption{Main script for the rotating double gyre flow, using the adaptive transfer operator method.}
\end{Code}

\begin{Code}[H]
\hspace*{12pt}\scalebox{0.95}{
\begin{lstlisting}[breaklines=false,numbers=left,linerange={1-50}]
function [D,M] = assemble(p,t)

% p: n by 2 matrix of nodes, t: m by 3 matrix, defining the triangles
% based on http://math.mit.edu/~gs/cse/codes/femcode.m by G. Strang 
% Note: for large n, this code is inefficient, use FEMDL instead

n = size(p,1);                      % number of nodes 
D = sparse(n,n); M = sparse(n,n);
for e = 1:size(t,1)                 % integration over each element 
    ns = t(e,:);                    % nodes of triangle e
    P = [ones(3,1),p(ns,:)];        % 3 x 3 matrix with rows = [1 x y]
    area = abs(det(P))/2;           % area of triangle e
    B = inv(P);                     
    grad = B(2:3,:);                % gradients of shape funtions
    D(ns,ns) = D(ns,ns) - area*grad'*grad;
    M(ns,ns) = M(ns,ns) + area/12*(ones(3,3)+eye(3));
end
\end{lstlisting}}
\caption{Assembly of stiffness and mass matrix.}
\end{Code}

\begin{Code}[H]
\hspace*{12pt}\scalebox{0.95}{
\begin{lstlisting}[breaklines=false,numbers=left,linerange={1-50}]
function y = flow_map(x,tspan)

    function dz = v(t,z)       % vector field of the rotating double gyre    
        n = numel(z)/2; 
        x = z(1:n,1); y = z(n+1:2*n,1);
        st = ((t>0)&(t<1)).*t.^2.*(3-2*t) + (t>1)*1;        
        dxPsi_P = 2*pi*cos(2*pi*x).*sin(pi*y);
        dyPsi_P = pi*sin(2*pi*x).*cos(pi*y);
        dxPsi_F = pi*cos(pi*x).*sin(2*pi*y);
        dyPsi_F = 2*pi*sin(pi*x).*cos(2*pi*y);
        dz(n+1:2*n,1) = (1-st).*dxPsi_P + st.*dxPsi_F;
        dz(1:n,1)   = -((1-st).*dyPsi_P + st.*dyPsi_F);       
    end      
        
[~,F] = ode45(@v,tspan,[x(:,1); x(:,2)]);
y = [F(:,1:end/2)' F(:,end/2+1:end)'];
end
\end{lstlisting}}
\caption{The flow map for the rotating double gyre flow.}
\end{Code}

\end{appendix}

\bibliographystyle{abbrv}

\end{document}